\documentclass{amsart}
\usepackage[utf8]{inputenc}  
\usepackage[T1]{fontenc}     
\usepackage{lmodern}
\usepackage{amssymb, graphics, color, enumitem, mathrsfs}
\usepackage[all]{xy}
\usepackage{hyperref}
\usepackage[cyr]{aeguill}
\usepackage{amsmath}
\usepackage[margin=1in]{geometry}

\newtheorem{theo}{Theorem}[section]
\newtheorem{mydef}[theo]{Definition}
\newtheorem{myexmp}[theo]{Example}
\newtheorem{cor}[theo]{Corollary}
\newtheorem{prop}[theo]{Proposition}
\newtheorem{lemm}[theo]{Lemma}
\newtheorem{remq}[theo]{Remark}
\newtheorem{eg}[theo]{Example}

\newtheorem*{Acknowledgement}{Acknowledgements}

\newcommand{\QAC}{\mathrm{QAC}}  

\begin{document}
	
	\title[An Isomorphism Theorem for Some Linear Elliptic Differential Operators on $\QAC$-manifolds]
	{An Isomorphism Theorem for Some Linear Elliptic Differential Operators on Quasi-Asymptotically Conical Manifolds}
	
	\author{Mohamed Nouidha}
	\address{D\'epartement de Math\'ematiques, Universit\'e du Qu\'ebec \`a Montr\'eal}
	\email{Nouidha.Mohamed@uqam.ca}
	
	\begin{abstract}
		We study a linear elliptic differential operator of the form $\mathcal{P}=\Delta + V - \lambda$ on a quasi-asymptotically conical (\textbf{QAC}) manifold $(M, g)$, where $g$ is a polyhomogeneous metric and $V$ is a $b$-vector field that is unbounded with respect to the metric $g$.
		
		In Section \ref{section_lunadi_theorem}, we study a more general elliptic operator (with unbounded coefficients) of the form $\mathcal{A} = \Delta + V + r - \lambda$, where $r$ is a function bounded above, and prove a global Schauder estimate for the operator $\mathcal{A}$ on a non-compact Riemannian manifold.
		
		Then, in Section \ref{section_isomorphism_theorem}, we develop weighted H\"older spaces that take into account the asymptotic behavior of the operator $\mathcal{P}$ on \textbf{QAC} manifolds, and prove an isomorphism theorem on the defined weighted spaces using the result proved in Section \ref{section_lunadi_theorem}.
		
		Section \ref{section_QAC} contains some results on polyhomogeneous tensors. For lack of a reference, we include proofs of results needed in Chapter 3. For example, we show that the Hodge-Laplace operator of a polyhomogeneous $QAC$-metric is of the form $\Delta = x_{\max}^2 \operatorname{P}{\mathcal{V}{Qb}}$, where $\operatorname{P}{\mathcal{V}{Qb}}$ is a second-order polynomial in $Qb$-vector fields with polyhomogeneous coefficients and without a zeroth-order term.
	\end{abstract}
	
	\maketitle
	\tableofcontents
	\pagenumbering{arabic}
	\section{introduction}

Let $\left(M, g\right)$ be a complete Riemannian manifold, and $\mathcal{P}$ a linear elliptic operator of second order defined by $\mathcal{P} = \Delta + V -\lambda$, where $V$ is a smooth vector field on $M$. When $M$ is compact, the mapping properties of $\mathcal{P}$ are relatively well understood (see for instance theorem \ref{theorem_elliptic_operator}). This ceases to be true on non compact manifolds, where we need \textbf{weighted} Hölder spaces adapted to the asymptotic behavior of $\Delta$ and $V$ at infinity. One of the issues with the operator $\mathcal{P}$ is that the vector $V$ is potentially unbounded. For instance, in \cite{ChaljubSimon1979ProblmesED} they consider a similar operator on \textbf{AE} manifolds but they require that the coefficients of order 1 and 0 be decreasing at infinity.

The theory of elliptic operators on weighted Hölder spaces was introduced by \cite{NIRENBERG1973271} and studied extensively by \cite{Lockhart} and \cite{McOwen}.
It was also used by \cite{ChaljubSimon1979ProblmesED} to study regularity of linear elliptic operator of second order on \textbf{asymptotically euclidean} (\textbf{AE}) manifolds, the work of whom was adapted by Joyce to study \textbf{asymptotically locally euclidean} (\textbf{ALE}) manifolds, and \textbf{quasi-asymptotically locally euclidean} (\textbf{QALE}) manifolds.\\
More recently, \cite{Degeratu_2017} proved Fredholm results of generalized Laplace-type operators for weighted Sobolev and Hölder spaces on \textbf{quasi-asymptotically conical} (\textbf{QAC}) manifolds. We also mention the work  of Conlon, Degeratu and Rochon \cite{rochon_QAC}, where such a result is used to solve a complex Monge-Ampère equation on weighted  Hölder spaces in order to build Ricci-flat $QAC-$metrics.

Before we state our results, we would like to explain the motivation behind the choice of the differential operator $\mathcal{P}$, which is related to the existence of  Kähler Ricci solitons on non compact manifolds.

\subsection{Riemannian Ricci soliton}

On a Riemannian manifold $(X, g_0)$, the Ricci flow is a heat-like equation of the form:
\begin{equation}
		\label{definition_ricci_flow}
	\begin{aligned}
		&\frac{\partial}{\partial t} g(t) = -2\operatorname{Ric}(g(t)),\\
		&g(0) = g_0,
	\end{aligned}
\end{equation}

where $\operatorname{Ric}(g(t))$ denotes the Ricci curvature of the metric $g(t)$. A \textbf{Ricci soliton} is a self-similar solution of this equation. More precisely,  $(M, g(t))$ is called a Ricci soliton if 
\begin{equation*}
	g(t) = \sigma(t)\phi_t^*g(0),
\end{equation*}
where $\phi_t:M\to M$ is a time dependent family of diffeomorphisms of $M$, and $\sigma(t)$ a time dependent scale factor, satisfying   $\phi_0 = \operatorname{Id}$ and $\sigma(0)=1$. If we plug $g(t)$ in equation $(\ref{definition_ricci_flow})$ (and evaluate at $t=0$), we obtain: 
\begin{align*}
	&-2 \phi_t^*\operatorname{Ric}(g_0) = \sigma'(t) \phi_t^*g_0 - \sigma(t)\phi_t^*\mathcal{L}_X g_0,\\
	&\operatorname{Ric}(g_0) -\frac{1}{2}\mathcal{L}_X g_0 +\frac{\sigma'(0)}{2} g_0 = 0,
\end{align*}
where $X = -\frac{d}{dt}_{\vert_{t=0}}\phi_t$.  Letting $\lambda = \frac{\sigma'(0)}{2}$, the soliton is called shrinking, steady or expanding if $\lambda<0,\lambda = 0$ or $\lambda>0$ respectively.\\
We also could define a \textbf{Ricci soliton} as a triple $\left(M, g, X\right)$ where $\left(M, g\right)$ is a Riemannian manifold and $X$ is a smooth vector field satisfying
\begin{equation}
	\operatorname{Ric}(g)-\frac{1}{2}\mathcal{L}_X g +\lambda g = 0,
\end{equation}
for $\lambda\in\left\{-1, 0, 1\right\}$. When $X = \nabla^g f$, we say that $\left(M, g, X\right)$ is a \textbf{gradient Ricci soliton}, and the above equation translates to 
\begin{equation}
	\operatorname{Ric}(g)- \operatorname{Hess}(f) +\lambda g = 0.
\end{equation}

Suppose now that $(M, J, g)$ is a Kähler manifold. If $(M, g)$ is a Ricci soliton and the vector field $X$ is real holomorphic, then the soliton equation can be rewritten as:
\begin{equation}
	\label{kahler_ricci_equation}
	\rho_{\omega}-\frac{1}{2}\mathcal{L}_X \omega +\lambda \omega = 0,
\end{equation}
$\omega(.,.) = g(J.,.)$ being the Kähler form, and $\rho_{\omega}(.,.)=\operatorname{Ric}(J.,.)$ the Ricci form. We say that $(M, g, X)$ is a \textbf{Kähler Ricci soliton}.

\subsubsection{Asymptotically conical Kähler-Ricci soliton}
Let $\pi:M\to C$ be an equivariant (with respect to the real holomorphic torus action generated by the Reeb vector field) crepant resolution of a  Kähler cone $(C, J_C, g_C)$, and let $X$ denote the holomorphic lift of the radial vector field on the cone. If we impose certain topological conditions on $M$ (see propositions 3.1 and 3.2 of \cite{conlon2016expanding}) we could define an asymptotically conical Kähler metric $g$ on $M$ and a function $F\in C^{\infty}(M)$ such that
\begin{align*}
	&\mathcal{L}_{JX}\omega_g = 0, \\
	&\mathcal{L}_{JX}F= 0, \\
\end{align*}
($\omega_g$ being the Kähler form of $g$) and such that $F$ satisfies the following equation
\begin{equation}
	\label{soliton_function_F}
	\rho_{\omega_g}-\frac{1}{2}\mathcal{L}_X \omega_{g} + \omega_{g} = i\partial\overline{\partial} F. 
\end{equation}
Furthermore, the function $F$ can be chosen such that it decays at infinity together with its derivatives.
The metric $g$ which is sometimes referred to as a \textbf{background metric} can also be chosen such that $g= \pi^*(g_C+\operatorname{Ric(g_C)})$ outside of a compact set.
Suppose now that $\omega_{\phi} = \omega_{g}+i\partial\overline{\partial}\phi$ is a Kähler metric that solve the soliton equation $(\ref{kahler_ricci_equation})$. Then, combining equations $(\ref{soliton_function_F})$ and $(\ref{kahler_ricci_equation})$ we obtain that
\begin{equation*}
	i\partial\overline{\partial}\left(\log\frac{\omega_{\phi}^n}{\omega_g^n}+\frac{1}{2}X\phi - \lambda\phi\right) = i\partial\overline{\partial} F.
\end{equation*}
Hence, if $\phi$ satisfies the following complex Monge-Ampère equation 
\begin{equation}
	\label{soliton_monge_ampere}
\log\frac{\omega_{\phi}^n}{\omega_g^n}+\frac{1}{2}X\phi - \lambda\phi = F,
\end{equation}
then $\omega_{\phi}$ is automatically a solution of equation $(\ref{kahler_ricci_equation})$.\\
Then we ask the following question: does there exist a smooth function $\phi$ that solves the complex Monge-Ampère equation $(\ref{soliton_monge_ampere})$?\\
It turns out that if we choose our weighted spaces carefully, we should be able to solve it.
To show the existence of a solution we usually use the continuity method.  The openness follows from the fact that the linearization of the previous operator is exactly the operator $\mathcal{P}$ which is an isomorphism. The closedness is the more difficult part, since we are dealing with an operator with unbounded coefficients on a non-compact manifold. \\
As examples of Asymptotically conical Kähler-Ricci solitons, Conlon and Deruelle show that given any negative line bundle $L$ over a projective manifold $D$, the total space of $L^{\otimes p}$ admits an asymptotically conical expanding gradient Kähler-Ricci soliton for any $p$ such that $c_1\left(K_D\bigotimes\left(L^*\right)^{\otimes p}\right)>0$. This is actually the particular case of a more general construction described in Corollary B of \cite{conlon2016expanding}.

\subsection{Motivation}
Let $L$ be holomorphic line bundle over a compact complex orbifold $D$. The total space of $L$ has singularities going off to infinity. In this case, the crepant resolution of $L$ will introduce some topology at infinity, which make it harder to build $QAC$ solitons using the same technique as in the $AC$ case. For instance, if we take the canonical bundle $K_X$ over a complex orbifold $X$ with isolated singularities, then, the canonical bundle over the crepant resolution of $X$ is a crepant resolution of $K_X$.\\
 Conlon, Degeratu and Rochon solve this problem by using a natural compactification of $L$ into an \textbf{orbifold with fibred corners} $\widetilde{L}$. Then, given a background metric of the right type (in their case a Calabi-Yau conic orbifold metric) they proceed by gluing suitable local models near each singularity.\\
We think that it is possible to proceed in the same manner in order to build $QAC$ Kähler-Ricci soliton. In fact we could use the same technique as in the $AC$ case to solve the soliton equation away from the singular set, and use the technique in \cite{rochon_QAC} to glue suitable models near each singularity. \\
In this setting, the radial vector field with respect to the (background) conique metric on $L\backslash D$ is a $b-$vector field on $\widetilde{L}$.\\
This work focuses on solving one of the needed steps in order to build examples of QAC Kähler-Ricci soliton.

\subsubsection{Main results}

In section \ref{section_lunadi_theorem} we prove a version of Lunardi's theorem \cite{lunardi_estimate} which is a global Schauder estimate that is essential in the proof of our main result. We modify and expand the proof in \cite{smoothing_deruelle} to obtain a more general version of the theorem. See section \ref{section_isomorphism_theorem} (subsection \ref{section_functional_spaces}) for the definition of the functional spaces mentioned below.

\begin{theo}[Lunardi]
	Let $\left(M^n, g\right)$ be a complete Riemannian manifold with positive injectivity radius, and $V$ be a smooth vector field on $M$. Let $\mathcal{A}$ be an elliptic differential operator acting on tensors over $M$ such that:
	\begin{equation*}
		\mathcal{A} = \underbrace{\Delta +\nabla_V}_{\Delta_V}+ r(x)\:,\:r\in C^3(M).
	\end{equation*}
	Suppose that $sup_{x\in M} r(x)=r_0<\infty$ and that there exists a positive constant $C$ such that $\sum_{i=1}^{3}\vert\vert\nabla^i r\vert\vert<C$. Assume also that there exists a positive constant $K$ such:\\
	\begin{equation*}
		\vert\vert Rm(g)\vert\vert_{C^3(M,E)}+\vert\vert Rm(g)\ast V\vert\vert_{C^3(M,E)}+\vert\vert \nabla V\vert\vert_{C^{2}(M,E)}\le K,
	\end{equation*}
	where $Rm(g)\ast V=Rm(g)(V,.,.,.)$. Assume also that there exists a function $\phi\in C^2(M)$ and a constant $\lambda_0\ge r_0$ such that:
	\begin{equation*}
		\lim_{x\to\infty}\phi (x) = +\infty, \: sup_{x\in M}\left(\mathcal{A}(\phi)(x)-\lambda_0\phi\left(x\right)\right)<\infty.
	\end{equation*}
	Then:
	\begin{enumerate}
		\item For any  $\lambda>r_0$, there exists a positive constant $C$ such that for any $H\in C^0\left(M, E\right)$, there exists a unique tensor $h\in D^2_{\mathcal{A}}\left(M, E\right)$, satisfying:
		\begin{equation*}
			\mathcal{A}(h) -\lambda h= H,\: \vert\vert h\vert\vert_{D^2_{\mathcal{A}}\left(M,E\right)}\le C \vert\vert H\vert\vert_{C^0\left(M, E\right)}.
		\end{equation*}
		Moreover $D^2_{\mathcal{A}}\left(M, E\right)$ is continuously embedded in $C^{\theta}\left(M,E\right)$ for any $\theta\in\left(0,2\right)$, i.e. there exists a positive constant $C(\theta)$ such that for any $h\in D^2_{\mathcal{A}}\left(M, E\right)$, 
		\begin{equation*}
			\vert\vert h\vert\vert_{C^{\theta}(M, E)}\le C(\theta)\vert\vert h\vert\vert^{\frac{\theta}{2}}_{D^2_{\mathcal{A}}\left(M, E\right)} \vert\vert h\vert\vert^{1-\frac{\theta}{2}}_{C^0(M, E)}.
		\end{equation*}
		\item For any $\lambda>r_0$, there exists a positive constant $C$ such that for any $H\in C^{0,\theta}\left(M, E\right)$, $\theta\in \left(0,1\right)$ , there exists a unique tensor $h\in C^{2,\theta}\left(M, E\right)$ satisfying:
		\begin{equation*}
			\mathcal{A}(h) - \lambda h= H,\: \vert\vert h\vert\vert_{C^{2,\theta}\left(M, E\right)}\le C \vert\vert H\vert\vert_{C^{0,\theta}\left(M, E\right)}.
		\end{equation*}
	\end{enumerate}
\end{theo}

Then in chapter 3 we prove the main results. Given a $QAC-$manifold $(X, g)$ such that $g$ is a polyhomogeneous metric, and a $b-$vector field $V$ on $X$ (we will denote by $M = \mathring{X}$), we obtain the following Schauder estimates for the linear elliptic operator $\mathcal{P}_{\alpha}$ (described in equation $(\ref{conjugated_operator_A})$ and remark $\ref{remark_original_operator_weighted_operator}$):

\begin{theo}
	Let $\mathcal{C}^{k;j,\theta}(M, E)$ be the functional space defined by:
	\begin{equation*}
		\mathcal{C}^{k;j,\theta}(M, E) = \left\{h\in C^{k+j+\lfloor\theta\rfloor,\theta-\lfloor\theta\rfloor}_{loc}(M, E)\:\vert\: x_{max}^{-i}\nabla^ih\in C^{j+\lfloor\theta\rfloor,\theta-\lfloor\theta\rfloor}(M, E),\:\:\forall i=0,\dots,k\right\}.
	\end{equation*}
	such that $\theta\in(0,2)$, and endowed with the norm:
	\begin{equation*}
		\vert\vert h\vert\vert_{\mathcal{C}^{k;j,\theta}(M, E)} = \sum_{i=0}^{k}\vert\vert x_{max}^{-i}\nabla^ih\vert\vert_{C^{j+\lfloor\theta\rfloor,\theta-\lfloor\theta\rfloor}(M, E)}
	\end{equation*}
	Suppose also that:
	\begin{equation*}
		\vert\vert Rm(g)\ast V\vert\vert_{C^0(M,E)}+\vert\vert \nabla V\vert\vert_{C^{0}(M,E)}< \infty
	\end{equation*}
	Then, for any constant $\lambda\in \mathbb{R}$ such that:
	\begin{equation*}
		\lambda>\max\left(\sup\limits_{M}V\ln(v^{\alpha}x_{max}^k),\sup\limits_{M}V\ln(v^{\alpha}x_{max}^{k-1})\right)
	\end{equation*}
	we have that:
	\begin{itemize}
		\item There exists a positive constant $C$ such that for any $H\in C^{k,\theta}_{Qb}(M, E)$ there exists a unique $h\in D^{k+2,\theta}_{\mathcal{P}_{\alpha}}(M, E)$ satisfying:
		\begin{equation*}
			\mathcal{P}_{\alpha}(h) - \lambda h=H,\: \vert\vert h\vert\vert_{D^{k+2,\theta}_{\mathcal{P}_{\alpha}}(M, E)}\le C\vert\vert H\vert\vert_{C^{k,\theta}_{Qb}(M, E)};\: \theta\in [0,1)
		\end{equation*}
		i.e. the operator 
		\begin{equation*}
			\mathcal{P}_{\alpha}-\lambda: D^{k+2,\theta}_{\mathcal{P}_{\alpha}}(M, E)\to C^{k,\theta}_{Qb}(M, E)
		\end{equation*}
		is an isomorphism of Banach spaces. Moreover, $D^{k+2}_{\mathcal{P}_{\alpha}}(M, E)$ embeds continuously in $\mathcal{C}^{k;0,\theta}(M, E)$ for any $\theta\in(0,2)$, i.e there exists a positive constant C such that for any $h\in D^{k+2}_{\mathcal{P}_{\alpha}}(M, E)$,
		\begin{equation*}
			\vert\vert h\vert\vert_{\mathcal{C}^{k;0,\theta}(M, E)}\le C \vert\vert h\vert\vert^{\frac{\theta}{2}}_{D^{k+2}_{\mathcal{P}_{\alpha}}(M, E)}\vert\vert h\vert\vert^{1-\frac{\theta}{2}}_{C^k_{Qb}(M, E)}
		\end{equation*}
		\item There exists a positive constant $C$ such that, for $\theta\in(0,1)$
		\begin{equation*}
			\vert\vert h\vert\vert_{\mathcal{C}^{k;2,\theta}(M, E)} \le C\vert\vert H\vert\vert_{C^{k,\theta}_{Qb}(M, E)}
		\end{equation*}
	\end{itemize}
\end{theo}

As a consequence of the previous theorem, we prove the following result:

\begin{theo}[Isomorphism theorem]
Le $(X, g)$ be a $QAC-$manifold such that $g$ is polyhomogeneous. We will denote by $M = \mathring{X}$.   Let $V$ be a $b-$vector field on $X$ such that:
\begin{equation*}
	\vert\vert Rm(g)\ast V\vert\vert_{C^0(M,E)}+\vert\vert \nabla V\vert\vert_{C^{0}(M,E)}< \infty
\end{equation*}
Then, the operator $\Delta_V-\lambda:D^{2+k,\theta}_{\Delta_V,\alpha}(M, E)\to C^{k,\theta}_{Qb,\alpha}(M, E)$ is an isomorphism of Banach spaces, for any $\theta\in (0,1)$ and any constant $\lambda\in\mathbb{R}$ such that:
\begin{equation*}
	\lambda>\max\left(\sup\limits_{M}V\ln(x^{\alpha}x_{max}^k),\sup\limits_{M}V\ln(x^{\alpha}x_{max}^{k-1})\right)
\end{equation*}
\end{theo}

Note that the $b-$vector $V$ is unbounded with respect to the $QAC-$metric $g$, which was taken into consideration when defining the weighted Hölder spaces used in this result.\\
It is also worth mentioning that this result generalizes a previous result of Deruelle \cite{smoothing_deruelle} on asymptotically conical manifolds that was used to build expanding Kähler Ricci solitons in \cite{conlon2016expanding} and in the analogous problem of constructing \textbf{QAC}-expanders, proves openness.

	\begin{Acknowledgement}
	I am grateful to Fr\'ed\'eric Rochon, my Ph.D. supervisor, for introducing this problem and for his help. His assistance and insightful remarks were crucial to the success of this project.
	\end{Acknowledgement}
	\section{Analysis on Riemannian manifolds}
	
We will try to summarize in this section the material needed for the subsequent sections, the main source being \cite{jost2008riemannian} and the three first chapters of \cite{joyce2000compact}.

Let $\left(M, g\right)$ be a non compact, complete Riemannian manifold of dimension $n$ with positive injectivity radius.  We will denote by $\operatorname{dV}$ the volume element induced by the metric $g$. When using coordinate notation, we implicitly refer to some local frame $\left(e_1,\dots, e_n\right)$ on $TM$, and its dual frame $\left(e_1^*,\dots,e_n^*\right)$ on $T^*M$.

\begin{mydef}
	Let $E$ a be a vector bundle over $M$. A connection $\nabla$ on $E$ is a linear map $\nabla:\Gamma\left(E\right)\to \Gamma\left(E\right)\bigotimes\Gamma\left(T^*M\right)$ satisfying:
	\begin{itemize}
		\item $\nabla (fs) = f\nabla s +s\otimes df$, for all $s\in\Gamma\left(E\right)$ and $s\in C^{\infty}(M)$,
		\item  $\nabla_{\alpha v+w}s = \alpha\nabla_v s +\nabla_w s$, for all $v,w\in\Gamma(TM)$, $s\in\Gamma\left(E\right)$ and $\alpha\in C^{\infty}(M)$.
	\end{itemize}
\end{mydef}

\begin{remq}
	To prevent any confusion we use $\Gamma(E)$ to denote the space of \textbf{smooth} global sections of $E$, i.e elements of $C^{\infty}\left(X, E\right)$.
\end{remq}

\begin{prop}
	Let $E$ be a vector bundle over $M$ endowed with a connection $\nabla$. Then, there exists a unique section $R\left(\nabla\right)\in\Gamma(\operatorname{End}(E))\bigotimes\Lambda^2T^*M)$ called the curvature that satisfies the following equation:
	\begin{equation}
		R\left(\nabla\right)(X, Y) e = \left[\nabla_X,\nabla_Y\right] e - \nabla_{\left[X, Y\right]} e\:,\:\text{for all } X,Y\in\Gamma\left(TM\right)\text{ and } e\in\Gamma\left(E\right).
	\end{equation}
\end{prop}

\begin{remq}
	\label{connection_extension}
	Given two vector bundles $E$ and $F$ over $M$, endowed with connections $\nabla^E$ and $\nabla^F$ respectively, we can define a connection $\nabla$ for each of the following vector bundles:
	\begin{itemize}
		\item $E\bigoplus F$: $\nabla( e+f) = \nabla^E e + \nabla^F f$;
		\item  $E\bigotimes F$: $\nabla (e\otimes f) = \nabla^E e\otimes f + e\otimes\nabla^F f$;
		\item  $E^*$: $\left(\nabla L\right)(e) = d(L(e)) - L(\nabla^E e)$.
	\end{itemize}
\end{remq}

\begin{mydef}
	Let $\left(E, \left<.,.\right>\right)$ be a vector bundle over $g$ endowed with a bundle metric. A connection $\nabla$ on $E$ is compatible with the bundle metric (or just \textbf{metric}) if 
	\begin{equation*}
		X\left<S,T\right> = \left<\nabla_X S,T\right> +\left<S,\nabla_X T\right>\:,\:\forall S,T\in\Gamma(E)\:,\:\forall X\in\Gamma(TM).
	\end{equation*}
\end{mydef}

\begin{theo}[Fundamental Theorem of Riemannian Geometry]
	There exists a unique torsion free (i.e. $\nabla_X Y- \nabla_Y X = \left[X, Y\right]$), compatible connection on $TM$ (equipped with the bundle metric $g$), defined by the following \textbf{Koszul formula}:
	\begin{equation}
		2\left<\nabla_X Y, Z\right> = X\left<Y, Z\right> + Y\left<X, Z\right>  - Z\left<X, Y\right>  + \left<\left[X, Y\right], Z\right>  -\left<\left[X, Z\right], Y\right> -\left<\left[Y, Z\right], X\right>. 
	\end{equation}
	This connection is called the Levi-Civita connection of the metric $g$.
\end{theo}

Using the musical isomorphisms, a Riemannian metric $g$ on a manifold $M$ induces an Euclidean metric $g^{-1}$ on the vector bundle $T^*M$. Locally, we will use notation $g_{ij}$ and $g^{ij}$ to refer to $g(e_i, e_j)$ and  $g^{-1}(e_i^*, e_j^*)$ respectively. The Euclidean metrics $g$ and $g^{-1}$ induce a bundle metric \textbf{$\left<.,.\right>$} on each tensor bundle of the form $E=TM^{\otimes^r}\bigotimes T^*M^{\otimes^s}$ in the following manner:\\
Let $T, S\in\Gamma(E)$ be two sections of $E$, then
\begin{equation}
	\left<T,S\right> = g_{i_1 p_1}\dots g_{i_r p_r}g^{j_1 q_1}\dots g^{j_s q_s}\:T^{i_1\dots i_r}_{j_1\dots j_s} S^{p_1\dots p_r}_{q_1\dots q_s},
\end{equation} 
where
\begin{align*}
	&T = T^{i_1\dots i_r}_{j_1\dots j_s} \: e_{i_1}\otimes\dots\otimes e_{i_r}\otimes e^*_{j_1}\otimes\dots\otimes e^*_{j_s},\\
	&S = S^{p_1\dots p_r}_{q_1\dots q_s} \: e_{p_1}\otimes\dots\otimes e_{p_r}\otimes e^*_{q_1}\otimes\dots\otimes e^*_{q_s}.
\end{align*}

This allows us to define the norm of a tensor $T$ as $\vert T\vert = \left<T, T\right>^{\frac{1}{2}}$, which is a continuous function on $M$. We will use this notation when defining function spaces.\\
The bundle of $k-$forms $\Lambda^kT^*M$ is a sub-bundle of $T^*M^{\otimes^k}$ locally spanned by sections of the form
\begin{equation*}
	e^*_{i_1}\wedge\dots\wedge e^*_{i_k} = \frac{1}{k!}\sum_{\sigma\in S_k}\operatorname{sign}(\sigma)e^*_{i_{\sigma(1)}}\otimes\dots\otimes e^*_{i_{\sigma(k)}}.
\end{equation*}
Using this identification, we can compute product and norms of $k-$forms. We can also define the formal adjoint $d^*$ of the exterior derivative $d:\Gamma\left(\Lambda^kT^*M\right)\to \Gamma\left(\Lambda^{k+1}T^*M\right)$, in order to define the \textbf{de Rham Laplacian} $\Delta = dd^*+d^*d$.\\
The Levi-Civita connection $\nabla$ of $g$ can be extended to a connection on a vector bundle of the form $TM^{\otimes^r}\bigotimes T^*M^{\otimes^s}$ (using remark \ref{connection_extension}). In particular, the induced connection is compatible with the bundle metric previously introduced.

\subsection{Functional spaces}

Let $\left(E,\left<.,.\right>,\nabla\right)$ be a vector bundle over $M$ endowed with a bundle metric and a compatible connection. \\
The space of continuous sections of $E$ that have $k$ continuous bounded derivatives we denote by $C^k(M, E)$, and admits a norm:
\begin{equation}
	\vert\vert s\vert\vert_{C^k(M, E)} = \sum_{i=0}^{k}\sup\limits_{M}\vert\nabla^is\vert,
\end{equation}
making it a Banach space. 
\subsubsection{Hölder spaces}
We define the \textbf{Hölder} space of continuous sections of $E$ that have $k+\alpha$ ($\alpha\in\left(0,1\right)$) continuous bounded derivatives to be:
	\begin{equation}
		C^{k,\alpha}(M, E) = \left\{s\in C^k(M, E)\:\vert\: \vert\vert s\vert\vert_{C^{k,\alpha}(M, E)} = \vert\vert s\vert\vert_{C^{k}(M, E)} +\left[\nabla^{k}s\right]_{\alpha}<\infty\right\},
	\end{equation}
	where the \textbf{Hölder} semi-norm $\left[\nabla^ks\right]_{\alpha}$ is defined by
	\begin{equation*}
		\left[T\right]_{\alpha} = \sup\limits_{x\in M}\sup\limits_{\substack{y\in M \\ 0<d(x,y)<\delta}}\frac{\vert T(x)-P^*_{x,y}T(y)\vert}{d(x,y)^{\alpha}},
	\end{equation*}
	$P_{x,y}$ being the parallel transport along the unique minimizing geodesic from $x$ to $y$, and $\delta$ is the injectivity radius of $g$. Notice that $C^{k,\alpha}(M, E)$ is a Banach space with the norm $\vert\vert .\vert\vert_{C^{k,\alpha}(M,E)}$. When $E$ is a trivial line bundle with the trivial connection, we denote those spaces by $C^k(M)$ and $C^{k,\alpha}(M)$ respectively.\\
	Now, we list some useful results regarding Hölder spaces.
	\begin{prop}
		\label{prop_product_func_section_holderestimate}
		For $\alpha\in\left(0, 1\right)$, $u\in C^{k,\alpha}(M)$, and $T\in C^{k,\alpha}(M, E)$, $uT\in C^{k,\alpha}(M, E)$. More precisely, there exists a positive constant $C>0$ such that
		\begin{equation*}
			\vert\vert uT\vert\vert_{C^{k,\alpha}(M, E)}\le C\left(\sum_{p=0}^{k} \vert\vert u\vert\vert_{C^{p}(M)}\vert\vert T\vert\vert_{C^{k-p,\alpha}(M, E)}+\vert\vert u\vert\vert_{C^{p,\alpha}(M)}\vert\vert T\vert\vert_{C^{k-p}(M, E)}\right)
		\end{equation*} 
		\begin{proof}
			Notice that $u(x)T(x) - u(y) P^*_{x,y}T(y) = u(x)\left(T(x) - P^*_{x,y}T(y)\right) + P^*_{x,y}T(y)\left(u(x) - u(y)\right)$. Consequently, we obtain that:
			\begin{equation*}
				\left[uT\right]_{\alpha}\le \vert\vert u\vert\vert_{C^{0}(M)}\left[T\right]_{\alpha}+\vert\vert T\vert\vert_{C^{0}(M,E)}\left[u\right]_{\alpha}
			\end{equation*}
			which then implies that 
			\begin{align*}
				\vert\vert uT\vert\vert_{C^{0,\alpha}(M,E)}&\le \vert\vert u\vert\vert_{C^{0}(M)}\vert\vert T\vert\vert_{C^{0}(M,E)} + \left[uT\right]_{\alpha} \\
				&\le \vert\vert u\vert\vert_{C^{0}(M)}\vert\vert T\vert\vert_{C^{0}(M,E)} +\vert\vert u\vert\vert_{C^{0}(M)}\left[T\right]_{\alpha}+\vert\vert T\vert\vert_{C^{0}(M,E)}\left[u\right]_{\alpha}\\
				&\le \vert\vert u\vert\vert_{C^{0,\alpha}(M)}\vert\vert T\vert\vert_{C^{0}(M,E)} + \vert\vert u\vert\vert_{C^{0}(M)}\vert\vert T\vert\vert_{C^{0,\alpha}(M,E)}.
			\end{align*} 
			For $k=1$ we use the Leibniz rule to obtain that $\nabla(uT)(x)-\nabla(uT)(y) = \nabla u(x)\otimes\left(T(x)-T(y)\right)+T(y)\left(\nabla u(x)-\nabla u(y)\right)+\nabla T(x)\left(u(x)-u(y)\right)+u(y)\left(\nabla T(x)-\nabla T(y)\right)$. This implies that:
			\begin{equation}
				\label{inequality_0}
				\left[\nabla(uT)\right]_{\alpha}\le \vert\vert\nabla u\vert\vert_{C^{0}(M)}\left[T\right]_{\alpha} + \vert\vert T\vert\vert_{C^{0}(M,E)}\left[\nabla u\right]_{\alpha} +\vert\vert\nabla T\vert\vert_{C^{0}(M,E)}\left[u\right]_{\alpha}+\vert\vert u\vert\vert_{C^{0}(M)}\left[\nabla T\right]_{\alpha}
			\end{equation}
			Then we use the fact that $\vert\vert uT\vert\vert_{C^{1}(M,E)}\le \vert\vert u\vert\vert_{C^{0}(M)}\vert\vert T\vert\vert_{C^{1}(M,E)}+\vert\vert\nabla u\vert\vert_{C^0(M)}\vert\vert T\vert\vert_{C^{0}(M,E)}$ combined with equation $(\ref{inequality_0})$ to prove the case $k=1$. We proceed by induction to prove the result for $k>1$. 
			\end{proof}
	\end{prop}
	
	\begin{remq}
		For any non-negative real value $\theta$, we will denote by $C^{\theta}(M, E)$ the Hölder space $ C^{\lfloor\theta\rfloor, \theta-\lfloor\theta\rfloor}(M, E)$.
	\end{remq}
	
	\begin{theo}[The Mean Value Theorem]
		Let $V$ and $W$ be two normed vector spaces, $\Omega\subset V$ a convex subset of $V$ and $f\in C^1(\Omega, W)$. Let $a,b\in\Omega$, and suppose that there exists a positive constant $M$ such that 
		\begin{equation*}
			\vert\vert f'(ta+(t-1)b)\vert\vert \le M\:,\:\text{ for all }t\in\left[0,1\right].
		\end{equation*}
		Then we have:
		\begin{equation*}
			\vert\vert f(a)-f(b)\vert\vert\le M\vert\vert a-b\vert\vert.
		\end{equation*}
	\end{theo}
	As a consequence of the previous theorem, we have 
	\begin{prop}
		Let $E$ be a vector bundle over $M$ and $T\in C^1(M, E)$. Then there exist a positive constant $C$ depending only on the constant $\delta$ used in the definition of Hölder spaces, such that:
		\begin{equation}
			\vert\vert T\vert\vert_{C^{0,\theta}(M, E)}\le C\vert\vert T\vert\vert_{C^{1}(M,E)}, \forall\theta\in\left(0, 1\right)
		\end{equation}
	\end{prop}
	
	\subsubsection{Sobolev Spaces}
	
	Let $p\ge 1$, we define the $\textbf{Lebesgue space}$ $L^p(M, E)$ as the set of locally integrable sections (elements of $L^1_{loc}(M, E)$) of $E$ such that the norm
	\begin{equation*}
		\vert\vert s\vert\vert_{L^p(M, E)} = \left(\int_M \vert s\vert^p dV \right)^{\frac{1}{p}},
	\end{equation*}
	is finite. Given a non-negative integer $k$, we define the $\textbf{Sobolev space}$ 
	\begin{equation*}
		W^{k,p}(M, E)=\left\{s\in L^p(M, E)\:\vert\: \nabla^is\in L^p(M, E)\right\},
	\end{equation*}
	with the norm
	\begin{equation*}
		\vert\vert s\vert\vert_{W^{k,p}(M, E)}=\left(\sum_{i=0}^{k}\int_M\vert\nabla^is\vert^p dV\right)^{\frac{1}{p}}.
	\end{equation*}
	Notice that $W^{k,p}(M, E)$ is a Banach space with the norm $\vert\vert .\vert\vert_{W^{k,p}(M, E)}$. Note also that the derivatives in the previous definition are meant in the weak sense. The local Sobolev space $W^{k,p}_{loc}(M, E)$ consists of all locally integrable sections $s$ of $E$  whose restriction to any pre-compact $Q\Subset M$ lies on $W^{k,p}(Q, E)$.
	\begin{equation*}
		W^{k,p}_{loc}(M, E)=\left\{s\in L^1_{loc}(M, E)\:\vert\: \forall Q\Subset M\: :\: s_{\vert_Q}\in W^{k,p}(Q, E)\right\}.
	\end{equation*}
\subsection{Fredholm operators}
Let $\left(E,\left<.,.\right>_E,\nabla^E\right)$ and $\left(F,\left<.,.\right>_F, \nabla^F\right)$ be two vector  bundles over $M$ of ranks $s$ and $t$ respectively.\\
Before we proceed to the definition, let us recall that for any $x\in M$ there exists a neighborhood $U\subset M$ of $x$, $(e_i)_{1\le i\le s}\in \Gamma\left(E_{\vert_U}\right)$, and $(f_i)_{1\le i\le t}\in \Gamma\left(F_{\vert_U}\right)$ such that $(e_i(y))_{1\le i\le s}$ and $(f_i(y))_{1\le i\le t}$  are basis of $E_y$ and $F_y$ respectively for any $y\in U$. $(e_i)_{1\le i\le s}$ and $(f_i)_{1\le i\le t}$  are called local frames. These frames can be used to locally trivialize $E$ and $F$ respectively.
\begin{mydef}
 A linear map $P:\Gamma(E)\to \Gamma(F)$ is called a smooth linear differential operator of order $l$ if in local trivializing neighborhood it has the following form:
 \begin{equation}
 	P e = \sum_{i=1}^{l} P_{i}e
 \end{equation}
 where $s\in\Gamma(E)$, and such that $P_{i}$ is an $t\times s$ matrix whose components are of the form:
 \begin{equation*}
 \sum_{\vert\beta\vert = i}a_{\beta}(x)\nabla^{\beta}
 \end{equation*}
  $\beta$ being a multi-index with indices ranging from $1\dots n$, and $a_{\beta}$ are locally defined smooth functions. \\
  Let $\xi=\left(\xi^1,\dots,\xi^n\right)\in\mathbb{R}^n$ and $\sigma_{\xi}(P, x)$ the $t\times s$ matrix obtained from $P_l$ by replacing $\nabla^{\beta}$ by $\xi^{\beta}$ and evaluated at $x$. $\sigma(P, x)$ is called the \textbf{principal symbol} of $P$. $P$ is a linear elliptic differential operator if $\sigma_{\xi}(P, x)$ is an isomorphism for any $\xi\ne 0$.
\end{mydef}

\begin{remq}
	Note that if a linear differential operator $P:\Gamma(E)\to\Gamma(F)$ is elliptic, then $\operatorname{rank}(E)=\operatorname{rank}(F)$.
\end{remq}

\begin{mydef}
	The formal adjoint of a linear differential operator $P:\Gamma(E)\to\Gamma(F)$ is a smooth linear operator $P^*:\Gamma(F)\to\Gamma(E)$ satisfying:
	\begin{equation}
		\label{l2_product_vectorbundle}
		\int_M \left<P e,f\right>_F\:\operatorname{dV} = \int_M \left<e,P^*f\right>_F\:\operatorname{dV} \text{ for all }e\in C_c^{\infty}(M, E)\text{, and } f\in C_c^{\infty}(M, F).
	\end{equation}
	$P$ is elliptic if and only if $P^*$ is. The formal adjoint $P^*$ depends on the choice of the bundles metrics on $E$ and $F$ as well as the Riemannian metric $g$ on $M$.
\end{mydef}

\begin{theo}[Schauder estimates]
	\label{schauder_estimate}
	Let $E$ and $F$ be vector bundles over $M$, $\Omega\Subset M$ be a bounded domain, $K$ a compact subset of $\Omega$, and $L:\Gamma(E)\to\Gamma(F)$ a linear elliptic differential operator of order $k$ with coefficients in $C^{k,\theta}(\Omega)$, where $k\ge 0$ and $\theta\in\left(0, 1\right)$.\\
	 Then, there exists a positive constant $C(K, \Omega, g, \theta, l, \text{coefficients of L})$ such that for any $u\in C^{k+l,\theta}(U, E)$ we have that:
	\begin{equation}
		\vert\vert u\vert\vert_{C^{k+l,\theta}(K, E)}\le C\left(\vert\vert L(u)\vert\vert_{C^{l,\theta}(\Omega, F)}+\vert\vert u\vert\vert_{C^{0}(\Omega,E)}\right).
	\end{equation} 
\end{theo}

\begin{remq}
	Let $f\in C^2(M)$, and assume that $f$ attains its maximum (minimum) at a point $p\in M$. Then
	\begin{equation*}
		\Delta f(p) \le 0\:\left(\Delta f(p)\ge 0\right)\:,\:\text{ and } df(p)=0
	\end{equation*}
	As a consequence, let $\mathcal{A} = \Delta + X$ where $X$ is a smooth vector field, and $f\in C^2(M)$. If $f$ attains its maximum (minimum) at a point $p\in M$, then $\mathcal{A}f(p)\le0$ ($\mathcal{A}f(p)\ge0$).
\end{remq}

\begin{mydef}
	Let $V$ and $W$ be two Banach spaces and $P\in\mathcal{L}(V, W)$, $\mathcal{L}(V, W)$ being the set of continuous linear maps from $V$ to $W$. Then, $P$ is a \textbf{Fredholm} operator if:
	\begin{itemize}
		\item $\operatorname{dim(ker(P))}$ is finite.
		\item $\operatorname{ran(P)}$ is a closed sub-space of $W$ with finite codimension. Actually, finite codimension implies closedness.
	\end{itemize}
	In this case, we define the index of $P$ by the equation:
	\begin{equation}
		\operatorname{index} P = \operatorname{dim(ker(L))}- \operatorname{dim(coker(L))} .
	\end{equation}
	In particular, an isomorphism between $V$ and $W$ is Fredholm of index zero.
\end{mydef}
\begin{mydef}
	Let $V$ and $W$ be two Banach spaces and $P\in\mathcal{L}(V, W)$. Then, $P$ is a \textbf{compact} operator if the image under $P$ of any bounded sequence in $V$ contains a convergent sub-sequence in $W$. 
\end{mydef}
\begin{remq}
	The index of a Fredholm operator does not change under perturbation by a compact operator. In other words, if $P$ and $K$ are a Fredholm operator and a compact operator respectively, then, $P+K$ is also Fredholm with the same index as $P$.
\end{remq}
The following theorem is partly a consequence of theorem \ref{schauder_estimate} , which in particular states that linear elliptic operators over a bounded domain are Fredholm. 
\begin{theo}[Theorem 1.5.4 \cite{joyce2000compact}]
	\label{theorem_elliptic_operator}
	Let $k>0$ and $l\ge k$ be integers, and $\theta\in\left(0, 1\right)$. Suppose that $E$ and $F$ are vector bundles over a compact manifold $M$, equipped with bundle metrics. Suppose also that $P:\Gamma(E)\to\Gamma(F)$ is a linear elliptic operator of order $k$ with $C^{l,\theta}$ coefficients.
	Then
	\begin{itemize}
		\item $P^*$ is elliptic with $C^{l-k,\theta}$ coefficients, and both $ker P$, $ker P^*$ are finite sub-spaces of $C^{k+l,\theta}(M, E)$ and $C^{l,\theta}(M, F)$ respectively.
		\item  If $f\in C^{l,\theta}(M, F)$ then there exists $u\in C^{k+l,\theta}(M, E)$ such that $P u = f$ if and only if $f\perp ker P^*$ ($f$ is in the subspace $F/ ker P^*$), and if one requires that $u\perp ker P$ then $u$ is unique.
	\end{itemize}
\end{theo}

The notation $\perp$ refers to the $L^2$ inner product defined in equation $(\ref{l2_product_vectorbundle})$. The previous theorem is a statement of the Fredholm alternative.

	\section[Quasi-Asymptotically conical manifolds]{Quasi-Asymptotically conical manifolds}
\label{section_QAC}
	
	Manifolds with fibred corners are a powerful tool to encode the asymptotic behavior of Riemannian metrics in term of the Lie algebra of vector fields. This section is a rather bare-bones introduction to the subject, the main source of which are the work of \text{Richard Melrose}, \cite{albin_witt}, \cite{debord2011pseudodifferential}, \cite{rochon_QAC} and \cite{kottke2021quasi}.
	
\subsection{Manifold with corners}
	
\begin{mydef}
	\label{def_stratefied_space}
	A $\textbf{smoothly stratified}$ space $X$ of dimension $n$ is a metrizable, locally compact, second countable space which decomposes into a locally finite union of locally closed $\textbf{strata}$ $\mathcal{S}=\left\{S_{\alpha}\right\}$, where each $S_{\alpha}$ is a smooth manifold of dimension $\text{dim S}_{\alpha} \le n$. The set of strata $\mathcal{S}$ obeys the following properties:
\begin{itemize}
	\item[$\left(i\right)$] Each strata $S$ is endowed with a tubular neighborhood $T_{S}$ and a radial function in the tubular neighborhood $\rho_{S}: T_{S}\to \left[0, 1\right)$ such that $\rho_{S}^{-1}(0)= S$, together with a continuous retraction $\pi_{S}:T_{S}\to S$.
	\item[$\left(ii\right)$] If $S_{\alpha}, S_{\beta}\in\mathcal{S}$, then $T_{S_{\alpha}}\cap S_{\beta}\ne\emptyset \Leftrightarrow S_{\alpha}\cap\overline{S_{\beta}}\ne\emptyset\Leftrightarrow S_{\alpha}\subset \overline{S_{\beta}}$. In this case we write $S_{\alpha}\le S_{\beta}$. If moreover $S_{\alpha}\ne S_{\beta}$ then we write $S_{\alpha}< S_{\beta}$. This induces a partial order on the set of strata $\mathcal{S}$.
	\item[$\left(iii \right)$] The retraction $\pi_{S}:T_{S}\to S$ is a locally trivial fibration with fibre the cone $C\left(L_S\right)$ over some compact stratified space $L_S$.
	\item[$\left(iv \right)$] If we let $X_i$ be the union of strata of dimension less than or equal to $i$, then we obtain a filtration $\emptyset\subset X_1\subset\cdots\subset X_n=X$, $X_{n-1}$ being the singular set and $X\backslash X_{n-1}$ the regular set.
\end{itemize}
\end{mydef}

\begin{remq}
	Although we don't specify any restriction on the codimension of the singular set, in some cases (complex algebraic varieties) the singular set is at least of real codimension 2.
\end{remq}

\begin{mydef}
	Let $\left(X,\mathcal{S}\right)$ be a smoothly stratified space. The $\textbf{depth}$ of $X$ is the largest $k$ such that $S_1<S_2<\cdots<S_{k+1}$ is a totally ordered chain in $\mathcal{S}$.\\ The $\textbf{relative depth}$ of a stratum $S$ is the largest $k$ such that $S<S_1<\cdots<S_{k}$ is a totally ordered chain in $\mathcal{S}$. The $\textbf{relative depth}$ of a point $x\in X$ is the relative depth of the unique stratum that contains it.
\end{mydef}

\begin{myexmp}
	Orbifolds are a good example of stratified spaces. Indeed, let $X$ be a complex orbifold. A point $x\in X$ is singular if in an orbifold chart, its local isotropy subgroup $H_x$ is non trivial, and regular otherwise. We will denote by $\mathcal{S}=\left\{S_{\Sigma_{\alpha}}\right\}$ the canonical stratification of $X$ such that each $S_{\Sigma_{\alpha}}$ is the union of singular points with isotropy subgroups in the isomorphism class $\Sigma_{\alpha}$, $S_{\Sigma_{e}}$ being the regular stratum.
\end{myexmp}

Let $n$ be a positive integer and $k$ an integer such that $0\le k\le n$. Let us define $\mathbb{R}^n_k=\left[0,\infty\right)^k\times\mathbb{R}^{n-k}$.   we define the set
\begin{equation}
	\partial_l\mathbb{R}^n_k=\left\{x\in\mathbb{R}^n_k\:\vert\: x_i=0 \text{ for exactly $l$ of the first $k$ indices }\right\}.
\end{equation}
An open subset of $\mathbb{R}^n_k$ is a set $\Omega = \widetilde{\Omega}\cap \mathbb{R}^n_k$ for some open set $\widetilde{\Omega}\subset \mathbb{R}^n$. We will denote by $\partial_l\Omega = \widetilde{\Omega}\cap\partial_l \mathbb{R}^n_k$ the boundary of codimension $l$ of $\Omega$. Given two open sets $\Omega_1$ and $\Omega_2$ of $\mathbb{R}^n_k$, a map $\phi:\Omega_1\to \Omega_2$ is a diffeomorphism if there exists two open sets $\widetilde{\Omega_1}$ and $\widetilde{\Omega_2}$ such that $\Omega_i=\widetilde{\Omega_i}\cap\mathbb{R}^n_k$ for $i=1,2$ and $\phi$ extends to a diffeomorphism $\phi:\widetilde{\Omega_1}\to \widetilde{\Omega_2}$ in the usual sense. Such a diffeomorphism restricts to a bijective map between boundaries of the same codimension. 

\begin{mydef}
	Let $X$ be a paracompact Hausdorff topological space. A chart with corners $\left(U, \phi\right)$ on $X$ is a homeomorphism $\phi:U\to V\subset \mathbb{R}^n_{k_{\phi}}$ for some integer $k_{\phi}$, such that $U$ and $V$ are open sets of $X$ and $ \mathbb{R}^n_{k_{\phi}}$ respectively. Two charts with corners $\left(U, \phi\right)$ and $\left(W,\psi\right)$ are compatible if $U\cap W=\emptyset$ or 
	\begin{equation*}
		\psi\circ\phi^{-1}:\phi(U\cap W)\to\psi(U\cap W)
	\end{equation*}
	is a diffeomorphism in the sense described earlier. A maximal set of compatible charts that covers $X$ is called a $\mathcal{C}^{\infty}$ structure with corners on $X$ of dimension $n$.  
	A $t-$manifold is a pair  $\left(X,\mathcal{F}= C^{\infty}(X)\right)$ such that $C^{\infty}(X)$ is the set of smooth functions on $X$ induced by some $\mathcal{C}^{\infty}$ structure with corners.\\
	We denote by $\partial_l X$ the set of boundaries of $X$ of codimension $l$, defined by:
	\begin{equation*}
		\partial_l X=\left\{p\in X\:\vert\: \text{charts around $p$ maps $p$ to $\partial_l \mathbb{R}^n_k$}\right\}.
	\end{equation*} 
	The boundary hypersurfaces of $X$ are the closure of connected components of $\partial_1 X$, the set of which will be denoted by $M_1(X)$.
\end{mydef}
\begin{mydef}[Melrose]
	\label{definition_manifold_with_corners}
	A manifold with corners $X$ of dimension $n$ and $\textbf{depth}$ at most $k$, is a $t-$manifold of dimension $n$ such that the boundary hypersurfaces of $X$ (corners of codimension 1) are embedded sub-manifolds (with corners) and such that $\partial_l X=\emptyset$ for any integer $l>k$. 
\end{mydef}
Some definitions of manifold with corners drops the second part of the previous definition. For instance, Joyce's definition (see definition 2.2 of \cite{joyce_corners}) doesn't impose any requirements on hypersurfaces. For example, the \textbf{teardrop} $T=\left\{\left(x,y\right)\in\mathbb{R}^2\:\vert\: x\ge 0, y^2\le x^2-x^4\right\}$ fits the definition of Joyce of a manifold with corners of dimension $2$, but does not qualify as such according to definition \ref{definition_manifold_with_corners} because $\partial T$ self intersects.
\begin{remq}
	We will assume that each boundary hypersurface $H_i$ of $X$ has a defining function $x_i\in \mathcal{C}^{\infty}(X)$ such that:
	\begin{itemize}
		\item[$\left(1\right)$] $H_i=x_i^{-1}(0)$;
		\item[$\left(2\right)$] $x_i$ is positive on $X\backslash H_i$;
		\item[$\left(3\right)$] $dx_i$ is nowhere vanishing on $H_i$;
		\item[$\left(4\right)$] Each point $p\in H_i$ has a local coordinate system with $x_i$ as one of its elements.
	\end{itemize}
\end{remq}

\begin{myexmp}
	As an example of a manifold with corners, we can take the product $X=X_1\times X_2$ of two connected manifolds with boundaries. In this case, $X$ has two hypersurfaces $H_1=\partial X_1\times X_2$ and $H_2=\partial X_2\times X_1$, the corner (of codimension 2) being $\partial X_1\times\partial X_2$.
\end{myexmp}

\subsubsection{QFB-metric}

\label{fibration_structure_section}
The notion of iterated fibration structure was introduced by Melrose in the context of the resolution (blowup) of smoothly stratified spaces. In fact, there is a one to one correspondence between smoothly stratified spaces and manifolds with fibred corners; see for instance propositions 2.5 and 2.6 in \cite{albin_witt}.

\begin{mydef}[Melrose]
	\label{melrose_iterated_fibration}
	Let $X$ be a manifold with corners and $\left(H_i\right)_{1\le i\le l}$ the list of boundary hypersurfaces of $X$. An $\textbf{iterated fibration structure}$ on $X$ is a collection of fibrations $\pi=\left(\pi_i\right)_{1\le i\le l}$ such that:
	\begin{itemize}
		\item[$\left(i\right)$] Each $\pi_i: H_i\to S_i$ is a fiber bundle with fiber $F_i$ where both $F_i$ and $S_i$ are manifolds with corners.
		\item[$\left(ii\right)$] If $H_{ij}=H_i\cap H_j\ne\emptyset$ then $\text{dim F}_i\ne\text{dim F}_j$.
		\item[$\left(iii\right)$] We write $H_i<H_j$ if $H_{ij}\ne\emptyset$ and $\text{dim F}_i<\text{dim F}_j$. In this case, $\pi_i:H_{ij}\to S_i$ is a surjective submersion. 
		\item[$\left(iv\right)$] The boundary hypersurfaces of $S_j$ are exactly the $S_{ij}=\pi_j\left(H_{ij}\right)$ with $H_i<H_j$. Moreover, there exists a surjective submersion $\pi_{ij}:S_{ij}\to S_i$ such that when restricted to $H_{ij}$ we have $\pi_{ij}\circ\pi_j=\pi_i$.
	\end{itemize}
\end{mydef}

The iterated fibration structure induces a partial order on the set of boundary hypersurfaces. Thus, we define the $\textbf{relative depth}$ of a boundary hypersurface $H$ as the largest $k$ such that $H<H_1<H_2<\cdots<H_{k-1}$ for some $k-1$ hypersurfaces $H_i$. Notice that if $H_i$ and $H_j$ are respectively minimal and maximal hypersurfaces , then both $F_j$ and $S_i$ are closed manifolds.
 
 Let $H_1,\dots, H_l$ be an exhaustive list of boundary hypersurfaces, and let $x_1,\dots, x_l$ be the corresponding boundary defining functions.  In what follows, we will denote by $v=\prod\limits_{k=1}^l x_k$ a total boundary defining function.

\begin{mydef}
	A $\textbf{tube system}$ for a hypersurface $H$ is a triplet $\left(\mathcal{N}_H,r_h, x_h\right)$ with $\mathcal{N}_H$ an open neighborhood of $H$ in $X$, $r_h:\mathcal{N}_H\to H$ a smooth retraction, and $\left(r_h, x_h\right):\mathcal{N}_H\to H\times[0,\infty)$ a diffeomorphism onto its image.
\end{mydef}

\begin{mydef}
	\label{def_manifold_fibered_corners}
	A $\textbf{manifold with fibred corners}$ is a manifold with corners endowed with an iterated fibration structure $\left(X,\pi\right)$. We say that the boundary defining functions are $\textbf{compatible}$ with the iterated fibration structure, if for each boundary hypersurfaces $H_i<H_j$, the restriction of $x_i$ to $H_j$ is constant along the fibers of $\pi_j:H_j\to S_j$.
\end{mydef}

We will always assume that the boundary defining functions are compatible with the iterated fibration structure in sense of definition \ref{def_manifold_fibered_corners}, and such that $x_i$ is identically equal to $1$ outside of a tubular neighborhood of $H_i$.

\begin{mydef}
	An $\textbf{iterated fibred tube system}$ of a manifold with fibred corners  $X$, is a family of tube systems $\left(\mathcal{N}_i, r_i, x_i\right)$ for $H_i\in M_1(X)$ such that for any hypersurfaces $H_i<H_j$ we have:
	\begin{equation}
		r_j\left(\mathcal{N}_i\cap\mathcal{N}_j\right)\subset\mathcal{N}_i\:,\:x_i\circ r_j = x_i\:,\:\pi_i\circ r_i\circ r_j = \pi_i\circ r_i \text{ on }\mathcal{N}_i\cap\mathcal{N}_j
	\end{equation}
	and the restriction of $x_i$ to $H_j$ is constant along the fibers of $\pi_j:H_j\to S_j$.
\end{mydef}

The existence of an iterated fibred tube system on a manifold with fibred corners is proved in lemma 1.4 of \cite{debord2011pseudodifferential}.

We will review the notion of Quasi fibred boundary metrics introduced in \cite{rochon_QAC}. Let $\left(X, \pi\right)$ be a manifold with fibred corners. We denote by:
\begin{equation}
	\mathcal{V}_b=\left\{\xi\in C^{\infty}(X; TX)\;\vert \;\xi \text{ is tangent to the hypersurfaces of $X$}\right\},
\end{equation}
the Lie algebra of $\textbf{b-vector fields}$. This intrinsic definition is equivalent the following one 
\begin{equation}
	\label{b_vector_field}
	\mathcal{V}_b=\left\{\xi\in C^{\infty}(X; TX)\;\vert \;\xi\left(x_i\right)\in x_i\mathcal{C}^{\infty}\left(X\right)\right\},
\end{equation}
which is easier to use.
\begin{mydef}
	\label{qfb_vector_field}
	A $\textbf{quasi fibred boundary vector field}$ (or $\textbf{QFB-vector field}$) is a b-vector field $\xi$ such that:
	\begin{itemize}
		\item[$\left(i\right)$] $\xi\vert_{H_i}$ is tangent to the fibers of $\pi_i$;
		\item[$\left(ii\right)$]  $\xi\left(v\right)\in v^2 \mathcal{C}^{\infty}\left(X\right)$, where $v$ is a total boundary defining function.
	\end{itemize}
	These conditions are clearly still satisfied for the Lie bracket of two such vector fields. Thus, the set of $\textbf{quasi fibred boundary}$ vector fields is a Lie algebra, which will be denoted by $\mathcal{V}_{QFB}(X)$.
\end{mydef}

\begin{remq}
	The definition of \textbf{QFB-vector fields} depends on the choice of a total boundary defining function $v\in C^{\infty}(X)$ (see lemma 1.1 of \cite{kottke2021quasi}).
\end{remq}

Using definition \ref{melrose_iterated_fibration} we can give an explicit description of QFB-vector fields. Indeed, let $H_1<H_2<\cdots<H_k$ be a totally ordered chain and $\left(x_1, y_1,x_2,y_2,\cdots, x_k, y_k, z\right)$ a local coordinate system around a point $p\in H_1\cap H_2\cap\cdots\cap H_k$ that straightens out the fibrations $\pi_i:H_i\to S_i$ such that:
\begin{itemize}
	\item  $x_i$ is a boundary defining function of $H_i$;
	\item  $y_i=\left(y^1_i,\cdots,y^{k_i}_i\right)$ for $i\in\left\{1,\cdots, k\right\}$ and $z=\left(z_1,\cdots, z_q\right)$;
	\item Each fibration $\pi_i$ corresponds to the map
	\begin{equation}
		\label{fibration_in_local_coordinate}
		\left(x_1,y_1,\cdots, \hat{x_i},y_i,\cdots, x_k,y_k,z \right)\mapsto \left(x_1,y_1,\cdots, x_{i-1},y_{i-1},y_i\right).
	\end{equation}
\end{itemize}
Using equation $(\ref{fibration_in_local_coordinate})$ we see that $\left(x_{i+1}, y_{i+1},\cdots, z\right)$ are coordinates on the fibers of $\pi_i$. Thus, the space of b-vector fields tangent to the fibers of the fibrations $\pi_i$ is locally spanned by:
\begin{equation}
	\label{local-desc_b_vector}
	\frac{\partial}{\partial z_j},\; \frac{\partial}{\partial y^l_j}, \;x_j \frac{\partial}{\partial x_j} \text{ for $j>i$} .
\end{equation}
Now, using the second part of definition \ref{qfb_vector_field} we deduce that QFB-vector fields are spanned by:
\begin{equation}
\label{local-desc_qfb_vector}
\frac{\partial}{\partial z_j},\; v_i \frac{\partial}{\partial y^j_i},\;v_{1} x_{1} \frac{\partial}{\partial x_{1} } , \;v_{i+1}\left(x_{i+1} \frac{\partial}{\partial x_{i+1} } - x_{i}\frac{\partial}{\partial x_{i}} \right), i=1\dots k-1,,
\end{equation}
where $v_i=\prod\limits_{j=i}^k x_j$.
\begin{remq}
	\label{aman_lie_structure}
	$\mathcal{V}_{QFB}\left(X\right)$ is called a $\textbf{structural Lie algebra}$ in the sense of definition 1.4 in \cite{Ammann_2004}. In particular, structural Lie algebras are finitely generated protective $C^{\infty}(X)$-modules. Thus, using the \textbf{Serre-Swan} theorem, there exists a smooth vector bundle (the $QFB-$tangent bundle) ${}^{\pi}TX\to X$ such that $\mathcal{V}_{QFB}\left(X\right) \simeq\Gamma\left({}^{\pi}TX\right)$. This vector bundle is actually a \textbf{boundary tangential Lie algebroid} (see definition 1.14 of \cite{Ammann_2004}). The same thing goes for $\mathcal{V}_{b}\left(X\right)$ and $\mathcal{V}_{}\left(X\right)$ (definition \ref{qb-vector-field} below).
\end{remq}
We will denote by  $i_{\pi}:{}^{\pi}TX\to TX$ the natural bundle map that restricts to an isomorphism over $\mathring{X}$, such that:
\begin{equation}
	\mathcal{V}_{QFB}\left(X\right)={i_{\pi}}_*\mathcal{C}^{\infty}\left(X; {}^{\pi}TX\right).
\end{equation}
The $\textbf{QFB-cotangent bundle}$ ${}^{\pi}T^*X$  is then defined as the vector bundle dual to the QFB-tangent bundle ${}^{\pi}TX$, and is locally spanned by:
\begin{equation}
	\label{qfb_forms}
	dz_j,\; \frac{dv_i}{v_i^2},\; \frac{dy^j_i}{v_i}.
\end{equation}

\begin{mydef}
 A $\textbf{quasi fibred boundary metric}$ (or \textbf{QFB-metric}) is a positive-definite tensor \\
 $g_{\pi}\in\mathcal{C}^{\infty}(\mathring{X}; \operatorname{Sym}^2\left({}^{\pi}T^*X\right))$ that restricts to a Riemannian metric on $\mathring{X}$ via the map $i_{\pi}:{}^{\pi}TX\to TX$. We say that $g_{\pi}$ is a smooth $QFB-$metric if it is smooth up to the boundary. We say that $\left(X, g_{\pi}\right)$ is a $QFB-$manifold.
\end{mydef}

\begin{mydef}
	If a manifold with fibred corners $\left(X,\pi\right)$ is such that $H_i= S_i$ and $\pi_i=Id$ for each maximal boundary hypersurface $H_i$, then a QFB-vector field is called $\textbf{quasi-asymptotically conical vector field}$ (or $\textbf{QAC-vector field}$) and in the same manner, a QFB-metric is called a $\textbf{quasi-asymptotically conical metric}$ (or $\textbf{QAC-metric}$). If $g_{QAC}$ is a $QAC-$metric on $X$, then $\left(X, g_{QAC}\right)$ is called a \textbf{QAC-manifold}.
\end{mydef}


\subsubsection{Asymptotically Conical manifolds}

As we said, manifolds with fibred corners can be used to encode asymptotic conditions on certain complete manifolds. Let us for instance consider a non-compact Riemannian manifolds $\left(M, g\right)$ of dimension $n+1$, and a compact subset $K\subset M$. Suppose that $M\backslash K$ is diffeomorphic to the non-compact ends of the Riemannian cone
$\left(\left(1,\infty\right)\times Y, g_c=dr^2+r^2h\right)$ with $\left(Y, h\right)$ a compact  Riemannian manifold (called the link of the cone). Suppose also that under such an identification we have that:
\begin{equation*}
	\vert\vert \nabla^k\left(g-g_c\right)\vert\vert = O(r^{-\epsilon-k})\text{ for all }k\in\mathbb{N}_0\text{, and some }\epsilon >0.
\end{equation*}
Then, $\left(M, g\right)$ is called an \textbf{Asymptotically Conical} (or \text{AC}) manifold. In particular,  $\left(M, g\right)$ is called  \textbf{Asymptotically Euclidean} (or \textbf{AE}) manifold if $\left(Y, h\right)$ is $\mathbb{S}^n$ equipped with the round metric, and \textbf{Asymptotically Locally Euclidean} (or \textbf{ALE}) if $Y=\mathbb{S}^n\backslash \Gamma$ where the finite subgroup $\Gamma\subset O(n)$ acts freely on $\mathbb{S}^n$. So \textbf{AC} manifolds can be seen as a generalization of \textbf{AE} and \textbf{ALE} manifolds.\\
Let $\left(C, g\right)$ be a manifold with boundary, and $g$ a $QAC-$metric on it. Then in a neighborhood of the boundary, the $\textbf{QAC-cotangent bundle}$ ${}^{\pi}T^*C$ is generated locally by
\begin{equation*}
	\frac{d\rho}{\rho^2},\; \frac{dy^j}{\rho},
\end{equation*}
 with $\rho$ is a boundary defining function, and such that the metric $g$ is of the form 
\begin{equation*}
	g= \underbrace{\frac{d\rho^2}{\rho^4} + \frac{h}{\rho^2}}_{g_0} +\eta,
\end{equation*}
where $\eta$ is some mixed terms tensor, and $h$ is a Riemannian metric on $\partial C$. If we suppose that $\vert\vert \eta\vert\vert_{g_0}=O(\rho^{\epsilon})$ (for some $\epsilon>0$), then it becomes clear that $AC-$metrics are a particular case of $QAC-$metrics on manifold with boundary. In fact, in manifold with boundary, QAC-metrics corresponds the \textbf{scattering metrics} of \cite{melrose1995geometric}.

\subsubsection{Quasi- Asymptotically Conical manifolds}

Let us start with the case of \textbf{Quasi-Asymptotically Locally Euclidean} (or \textbf{QALE}) manifolds. These were introduced by Joyce \cite{joyceQALE} to study the existence of \textbf{Kähler metrics} on the resolution of $\mathbb{C}^n\backslash \Gamma$ where $\Gamma\subset U(n)$ is a finite subgroup that does not act freely on $\mathbb{C}^n\backslash\left\{0\right\}$. In this case, fixed points are subspaces of $\mathbb{C}^n$ with potentially different isotropy subgroups of $\Gamma$. The main source of examples of $QALE-$metrics are crepant resolutions of $\mathbb{C}^n\backslash \Gamma$ (see theorem 3.3 of \cite{joyceQALE}). Although, \cite{Carron_2011} showed that the Nakajima metric \cite{nakajima1999lectures} is a $QALE-$metric in the sens of Joyce.\\
Mazzeo gave a description of these singular sets in terms of \textbf{iterated cone-edge} spaces \cite{Rafe_QAC}, a sub class of stratified spaces and together with Degeratu \cite{Degeratu_2017} introduced \textbf{Quasi-Asymptotically Conical manifolds} as resolution blow-ups of these manifolds (into a manifold with fibred corners). An alternative description of these metrics was given in section $1$ of \cite{rochon_QAC}. In some sense, $QAC-$manifolds generalize $QALE-$manifolds the way $AC-$manifolds generalize $ALE-$manifolds.\\
In their work \cite{rochon_QAC}, Conlon, Degeratu and Rochon built $Calabi-Yau$  $QAC-$metrics that are neither $QALE-$metrics nor Cartesian products of $AC-$metrics. Concrete examples of such metrics can be built using the following theorem
\begin{theo}[Corollary 4.10 of \cite{rochon_QAC}]
	Let $\left(D, g_D\right)$ be a Kähler-Einstein Fano orbifold with isolated singularities of complex codimension at least two with each locally admitting a Kähler crepant resolution, then $D$ admits a Kähler crepant resolution $\widehat{D}$ and the $QAC-$compactification $\widehat{X}_{QAC}$ of $K_{\widehat{D}}$ admits a Kähler $QAC-$metric asymptotic to $g_C$ (a quasi-regular Calabi-Yau cone metric on $K\backslash D$) with rate $\delta$ for any $\delta>0$.
\end{theo} 

\subsubsection{Qb manifolds}
\begin{mydef}
	\label{qb-vector-field}
	Let $\left(X,\pi\right)$ be a QAC-manifold and ${x_{max}}$ the product of boundary defining functions of the maximal hypersurfaces of $X$. A smooth \textbf{quasi b-metric} (\textbf{Qb-metric}) is a metric $g_{Qb}$ of the form:
	\begin{equation}
		\label{qb-metric-definition}
		g_{Qb}=x^2_{max}g_{QAC}.
	\end{equation}
	for some smooth QAC-metric $g_{QAC}$. The Lie algebra of Qb-vector fields is defined as
	\begin{equation*}
		\mathcal{V}_{Qb}=\{\xi\in C^{\infty}\left(X, TX\right)\:\vert\: \sup\limits_{\mathring{X}}g_{Qb}\left(\xi,\xi\right)<\infty\}.
	\end{equation*}
	or equivalently as b-vector fields such that for each i
	\begin{itemize}
		\item $\xi\vert_{H_i}$ is tangent to the fibers of $\pi_i$ if $H_i$ is not a maximal hypersurface;
		\item  \label{qb-vector-property} $\xi v\in\frac{v^2}{x_{max}}C^{\infty}(X)$.
	\end{itemize}
\end{mydef}

\begin{remq}
	\label{remark_QAC_QB_relation}
	From the previous definition, it is easy to see that $\mathcal{V}_{QAC}\left(X\right)=x_{max}\mathcal{V}_{Qb}\left(X\right)$. 
\end{remq}

\begin{prop}
	\label{lie_brackets_of_QAC_vectors}
	Given two QAC-vector fields $\tilde{V}$ and $\tilde{W}$ we have that $\left[\tilde{V},\tilde{W}\right]\in x_{max}\mathcal{V}_{QAC}(X)$.\\
	In addition, $X(f)\in x_{max}C^{\infty}(X)$ for any $QAC-$vector field $X$ and function $f\in C^{\infty}(X)$.\\
	\begin{proof}
		Using remark \ref{remark_QAC_QB_relation}, any QAC-vector field $\tilde{V}$ is of the form $x_{max}V$ for some some  Qb-vector field $V$. Then,
		$\left[\tilde{V},\tilde{W}\right]=\left[x_{max}V,x_{max}W\right]$ for some Qb-vector fields $V$ and $W$. A straightforward computation shows that:
		\begin{equation*}
			\left[x_{max}V,x_{max}W\right] = x^2_{max}\left(\left[V, W\right]+\frac{V(x_{max})}{x_{max}}W - \frac{W(x_{max})}{x_{max}}V\right).
		\end{equation*}
		By definition \ref{qb-vector-field}, both $\frac{W(x_{max})}{x_{max}}$ and $\frac{W(x_{max})}{x_{max}}$ are in $C^{\infty}(X)$, which implies that $\left[x_{max}V,x_{max}W\right]\in x_{max}\mathcal{V}_{QAC}(X)$. The second assertion follows from the fact that $\xi(f)\in C^{\infty}(X)$ for any $Qb-$vector field $\xi$.
	\end{proof}
\end{prop}

\subsection{Polyhomogeneity}
Note that when defining a $QAC-$metric we only require that the given tensor is defined on $\mathring{X}$. In some cases, such tensors can be extended smoothly up to the boundary, but most of the time requiring a metric to be smooth up to the boundary is restrictive. In this section we introduce a class of tensors that, while not smooth up to the boundary, have a Taylor like asymptotic expansion near the boundary. These are called polyhomogeneous tensors.
\begin{mydef}
	An index set $K$ is a subset of $\mathbb{C}\times\mathbb{N}_0$ such that:
	\begin{align}
		&\left(z, k\right)\in K\:,\:\space \vert\left(z, k\right)\vert\to\infty\implies \text{Re } z\to\infty,\\
		&\left(z, k\right)\in K\:,\space p\in\mathbb{N}\implies\left(z+p, k\right)\in K,\\
		&\left(z, k\right)\in K\implies\left(z,p\right)\in K\:\forall\: 0\le p\le k.
	\end{align}
   An index set $K$ is a non-negative index set if it also satisfies the following conditions:
   	\begin{align}
   	&\mathbb{N}_0\times\left\{0\right\}\subset K,\\
   	&\left(z, k\right)\in K\implies Im z=0\text{ and } Rez\ge 0,\\
   	&\left(0, k\right)\in K\implies k=0.
   \end{align}
\end{mydef}

Given two index sets $G$ and $K$, we define the index set $G+K$ as follows:
\begin{equation}
	G+K =\left\{\left(z_1+z_2, k_1+k_2\right)\:\vert\:\left(z_1, k_1\right)\in G,\left(z_2, k_2\right)\in K\right\}.
\end{equation}
Note that, if $K$ and $G$ are non-negative index sets, then $K\cup G\subset K+G$. In particular, given a non-negative index set $H$, we define the index set
\begin{equation}
	H_{\infty} = \sum_{i=1}^{\infty} H = \bigcup_{i=1}^{\infty}\sum_{j=1}^{i} H.
\end{equation}
It is easy to see that $H_{\infty} + H= H_{\infty}$.
Before we define polyhomogeneous functions on a manifold with corners, let's start with the simpler case of a manifold with boundary.
\begin{mydef}
	Let $M$ be a compact manifold with boundary, and $\rho$ the boundary defining function of $\partial M$. The set of polyhomogeneous functions on $M$ with respect to an index set $K$, denoted by $\mathcal{A}^K_{phg}(M)$, is the the set of functions $f\in C^{\infty}(\mathring{M})$ such that:
	\begin{equation}
		f\sim\sum_{\left(z,k\right)\in K} a_{(z,k)}\rho^z \left(\log\rho\right)^k\:,\: a_{(z,k)}\in C^{\infty}(\partial M),
	\end{equation}
	where $\sim$ means that for any $N\in\mathbb{N}$ we have that:
	\begin{equation}
		f-\sum\limits_{\substack{\left(z,k\right)\in K \\ Re z\le N}}{} a_{(z,k)}\rho^z \left(\log\rho\right)^k\in \dot{C}^N(M),
	\end{equation}
	where $\dot{C}^N(M)$ is the set of $N$ differentiable functions on $M$ that restrict to zero on $\partial M$ together with all their derivatives up to order $N$. We define $$\dot{C}^{\infty}(M)=\bigcap_{N\in\mathbb{N}} \dot{C}^N(M).$$
\end{mydef}
\begin{remq}
	\label{inverse_of_polyhomogeneous_is_polyhomogeneous}
	Note that if $K=\mathbb{N}_0\times\left\{0\right\}$ then $\mathcal{A}^K_{phg}(M) = C^{\infty}(M)$. It is also easy to see that $\mathcal{A}^{\emptyset}_{phg}(M) = \dot{C}^{\infty}(M)$. Another important remark is that the multiplicative inverse of a positive polyhomogeneous function $f$ that is bounded away from zero, is also polyhomogeneous. This is a direct consequence of theorems $B.1$ and $B.6$ of \cite{sher_inverse_poly}.
\end{remq}
Now we are ready to define  polyhomogeneous functions on a manifold with corners $X$. 
\begin{mydef}
	An index family $\mathcal{K}$ on a manifold with corner $X$, is an assignment of an index set $\mathcal{K}(H)$ to each hypersurface $H\in M_1(X)$.
	If $F$ is a boundary surface of $X$, then we will denote by $\mathcal{K}_{\vert_F}$ the family index that assigns $\mathcal{K}(H)$ to the boundary hypersurface $F\cap H$ of $F$ (such that $H\in M_1(X)$). $\mathcal{K}$ is a non-negative family index if $\mathcal{K}(H)$ is a non-negative index set for each $H\in M_1(X)$. \\
	Given two family indices $\mathcal{G}$ and $\mathcal{K}$, we define family indices $\mathcal{G}+\mathcal{K}$ and $\mathcal{G}_{\infty}$ as follows:
	\begin{align*}
		&\mathcal{G}_{\infty}(H) = \left(\mathcal{G}(H)\right)_{\infty},\\
		&\left(\mathcal{G}+\mathcal{K}\right)(H) = \mathcal{G}(H)+\mathcal{K}(H) \text{ for each } H\in M_1(X).
	\end{align*}
	We will denote by $\mathcal{A}^{\mathcal{K}}_{phg}(X)$ the space of polyhomogeneous functions on $X$ with index family $\mathcal{K}$.
\end{mydef}

\begin{mydef}
	$\mathcal{A}^{\mathcal{K}}_{phg}(X)$ is the set of functions $f\in C^{\infty}(\mathring{X})$ such that near each boundary hypersurface $H$:
	\begin{equation}
		f\sim\sum_{\left(z,k\right)\in \mathcal{K}(H)} a_{(z,k)}\rho_H^z \left(\log\rho_H\right)^k\:,\: a_{(z,k)}\in \mathcal{A}^{\mathcal{K}_{\vert_H}}_{phg}(H),
	\end{equation}
	$\rho_H$ being the defining function of $H$.
\end{mydef}

Note that in the previous definition, the $a_{(z,k)}$ coefficients are well defined since the induction will ends when reaching a corner of maximal codimension which consists of a closed manifold.
\begin{mydef}
	Let $\left(X,\pi\right)$ be a QAC-manifold, and $E$ a vector bundle over $X$. The set of polyhomogeneous sections of $E$ with family index $\mathcal{K}$ is defined by:
	\begin{equation}
		\mathcal{A}^{\mathcal{K}}_{phg}(X, E) = \mathcal{A}^{\mathcal{K}}_{phg}(X)\bigotimes_{C^{\infty}(X)}\Gamma\left(E\right).
	\end{equation}
\end{mydef}

\begin{mydef}
	We define a \textbf{polyhomogeneous QFB-metric } as an euclidean metric $g\in \mathcal{A}^{\mathcal{K}}_{phg}(X,\operatorname{Sym}^2\left( {}^{\pi}T^*X\right) )$ such that $\mathcal{K}$ is a non-negative family index. We define \textbf{polyhomogeneous QAC-metric } in the same manner.\\
	As a direct consequence of this definition, we have that $\vert\vert \xi\vert\vert_g$ is uniformly bounded on $\mathring{X}$ for any $QFB-$vector field $\xi$.
\end{mydef}

\begin{prop}
	\label{inverse_of_polyhomogeneous_metric}
	The inverse of a \textbf{polyhomogeneous QFB-metric } $g$ is also a \textbf{polyhomogeneous QFB-metric}. \\
	Thus, $g$ induces a \textbf{polyhomogeneous euclidean metric }  on the vector bundles $E={}^{\pi}TX^{\otimes^r}\otimes{}^{\pi}T^*X^{\otimes^s}$.\\
	\begin{proof}
		Use remark \ref{inverse_of_polyhomogeneous_is_polyhomogeneous}
	\end{proof}
\end{prop}
\begin{eg}
	\label{b_vectors_as_polyhomogeneous}
	Using equations $(\ref{local-desc_b_vector})$ and $(\ref{local-desc_qfb_vector})$ we can see that $\mathcal{V}_b\subset \mathcal{A}^{\mathcal{G}}_{phg}(X, {}^{\pi}TX)$ such that $\mathcal{G}$ is an index family satisfying $\mathcal{G}(H) \subset \mathbb{Z}_{\left\{\ge -1\right\}}\times\left\{0\right\}$ for every $H\in M_1(X)$.
\end{eg}

\subsubsection{Some results on \textbf{polyhomogeneous QAC-metrics}}

Some of the work done in \cite{Ammann_2004} can actually be extended to \textbf{polyhomogeneous QAC-metric}. For instance, let $\left(X,\mathcal{V}_{QAC}, g_{QAC}\right)$ be a \textbf{QAC-manifold} such that $g_{QAC}$ is polyhomogeneous with respect to a non-negative family index $\mathcal{G}$. We will denote by $\operatorname{Diff}_{\mathcal{V}_{QAC}}(X)$ the algebra of differential operators generated by vectors in $\mathcal{V}_{QAC}$ and with coefficients in $C^{\infty}(X)$.

\begin{mydef}
	\label{def_differential_operators}
	Given two vector bundles $E_1$ and $E_2$ over $X$, we will denote by 
	\begin{equation}
		\operatorname{Diff}_{\mathcal{V}_{QAC}}(X, E_1, E_2) = \operatorname{Diff}_{\mathcal{V}_{QAC}}(X)\bigotimes_{{C^{\infty}(X)}}\Gamma\left(E_1^*\otimes E_2\right)
	\end{equation}
	the algebra of differential operators taking sections of $E_1$ to sections of $E_2$ generated by vectors in $\mathcal{V}_{QAC}$ and with coefficients in $C^{\infty}(X)$.  To clarify this definition, let $\mathcal{U}$ be a local trivializing neighborhood of both $E_1$ and $E_2$. Then, $\Gamma(E_{i_{\vert_{\mathcal{U}}}}) \simeq C^{\infty}(\mathcal{U})\bigotimes\mathbb{C}^{N_i}$ for $i=1,2$. So locally, the elements of  $\operatorname{Diff}_{\mathcal{V}_{QAC}}(X, E_1, E_2)$ are a linear combination of the composition of operators of the form $X\otimes A$, such that $X$ in $\mathcal{V}_{QAC}(X)$, and $A$ a smooth family of linear mappings in $\mathcal{L}\left(\mathbb{C}^{N_1},\mathbb{C}^{N_2}\right)$. Note that this definition doesn't depend on the local trivialization since two local trivializations only differ by a smooth family of linear mappings $\in\mathcal{L}\left(\mathbb{C}^{N_1},\mathbb{C}^{N_2}\right)$.\\ 
	 In the same manner, we define
	\begin{equation}
		\operatorname{Diff}_{\mathcal{V}_{QAC},\mathcal{G}}(X, E_1, E_2) = \mathcal{A}^{\mathcal{G}}_{phg}(X)\bigotimes_{{C^{\infty}(X)}}\operatorname{Diff}_{\mathcal{V}_{QAC}}(X, E_1, E_2), 
	\end{equation}
	the algebra of differential operators taking sections of $E_1$ to sections of $E_2$ generated by vectors in $\mathcal{V}_{QAC}$ and with polyhomogeneous coefficients in $\mathcal{A}^{\mathcal{G}}_{phg}(X)$. When $E_1=E_2$ we will use the simpler notations $\operatorname{Diff}_{\mathcal{V}_{QAC}}(X, E)$ and  $\operatorname{Diff}_{\mathcal{V}_{QAC},\mathcal{G}}(X, E)$.
\end{mydef}

\begin{prop}
	\label{connection_polyhomogeneous_covariant}
	Let $\nabla$ be the Levi-Civita connection of $g_{QAC}$. Then, $\nabla$ can be extended to a differential operator in $x_{max}\operatorname{Diff}_{\mathcal{V}_{QAC},\mathcal{G}}(X, {}^{\pi}TX,{}^{\pi}TX^*\otimes {}^{\pi}TX)$. Consequently, the Riemannian curvature tensor $R$ is an element of $x^2_{max}\mathcal{A}^{\mathcal{G}}_{phg}(X, \Lambda^2\left({}^{\pi}T^*X\right)\bigotimes End({}^{\pi}TX))$.\\
	\begin{proof}
	Let $X, Y$ and $Z$ be three $QAC$-vector fields. Then, using the Koszul  identity, we have that
	\begin{equation}
		\label{Koszul_identity}
		2\left<\nabla_X Y, Z\right> = X\left(\left<Y, Z\right> \right) + Y\left(\left<X, Z\right> \right) - Z\left(\left<X, Y\right> \right) + \left<\left[X, Y\right], Z\right> - \left<\left[Y, Z\right], X\right> - \left<\left[X, Z\right], Y\right>.
	\end{equation}
	Then, using proposition \ref{lie_brackets_of_QAC_vectors} we deduce that each term in the right side of equation $(\ref{Koszul_identity})$ is in $x_{max}\mathcal{A}^{\mathcal{G}}_{phg}(X)$. Thus, $\left<\nabla_X Y, Z\right>\in x_{max}\mathcal{A}^{\mathcal{G}}_{phg}(X)$, which implies that $\nabla_X Y\in x_{max}\mathcal{A}^{\mathcal{G}}_{phg}(X, {}^{\pi}TX)$. Since $R(X, Y) = \left[\nabla_X,\nabla_Y\right]-\nabla_{\left[X, Y\right]}$, the remaining statements are direct consequences of the previous one.
	\end{proof}
\end{prop}
As a consequence of proposition \ref{connection_polyhomogeneous_covariant} , we have the following corollary.
\begin{cor}
	\label{covariant_polyhomogeneous_tensor}
	Let $k\in\mathbb{N}$. Then $\nabla^kR\in x^{2+k}_{max}\mathcal{A}^{\mathcal{G}_{\infty}}_{phg}(X, {}^{\pi}T^*X^{\otimes^k}\otimes \Lambda^2\left({}^{\pi}T^*X\right)\bigotimes End({}^{\pi}TX))$. More generally, if $T\in\mathcal{A}^{\mathcal{K}}_{phg}\left(X, {}^{\pi}TX^{\otimes^r}\otimes{}^{\pi}T^*X^{\otimes^s}\right)$, then
	 $\nabla^kT\in x_{max}^k\mathcal{A}^{\mathcal{K}+\mathcal{G}_{\infty}}_{phg}\left(X, {}^{\pi}T^*X^{\otimes^k}\otimes {}^{\pi}TX^{\otimes^r}\otimes{}^{\pi}T^*X^{\otimes^s}\right)$.
\end{cor}

\begin{prop}
	\label{bounded_polyhomogeneous_tensor}
	Let $T\in \mathcal{A}^{\mathcal{K}}_{phg}\left(X, {}^{\pi}TX^{\otimes^r}\otimes{}^{\pi}T^*X^{\otimes^s}\right)$ be a polyhomogeneous tensor with respect to some index family $\mathcal{K}$ such that $\mathcal{K}(H)\subset \mathbb{Z}\times\mathbb{N}_0$ for any $H\in M_1(X)$. Then, $\vert\vert T\vert\vert_{g_{QAC}}<\infty$ if and only if $\mathcal{K}$ can be chosen to be a non-negative index family.\\
	\begin{proof}
		Let us choose a local frame near the boundary of $X$ that diagonalizes the metric $g_{QAC}$. Then, we have that:
		\begin{equation*}
			\vert\vert T\vert\vert^2_{g_{QAC}} = g_{QAC,i_1 i_1}\dots g_{QAC,i_r i_r}g_{QAC}^{j_1j_1}\dots g_{QAC}^{j_s j_s} {T^{i_1\dots i_r}_{j_1\dots j_s}}^2.
		\end{equation*}
		The coefficients $g_{QAC,i_1 i_1}\dots g_{QAC,i_r i_r}g_{QAC}^{j_1j_1}\dots g_{QAC}^{j_s j_s}$ are positive and bounded on $\mathring{X}$. Thus, $\vert\vert T\vert\vert_{g_{QAC}}<\infty$ implies that the coefficients $T^{i_1\dots i_r}_{j_1\dots j_s}$ are also uniformly bounded on $\mathring{X}$. Consequently, $\mathcal{K}$ can be chosen to be a non-negative index family. To prove the opposite direction we only use that polyhomogeneous functions with respect to a non-negative index family are bounded. 
	\end{proof}
\end{prop}

As a consequence of proposition \ref{bounded_polyhomogeneous_tensor} we have the following important corollary that is going to be essential in the proof of the isomorphism theorem in the next section.
\begin{cor}
	\label{polyho_lunardi_condition}
	Let $V$ be a $b-$vector field such that:
	\begin{equation}
		\label{bounded_curvature_tensor}
		\vert\vert Rm(g_{QAC})\ast V\vert\vert_{C^0(\mathring{X})} + \vert\vert  \nabla V\vert\vert_{C^0(\mathring{X})} <\infty
	\end{equation}
	Then both $\left(Rm(g_{QAC})\ast V\right)$ and $\nabla V$ are polyhomogeneous tensors with respect to non-negative family indices. Consequently, we have that:
	\begin{equation}
		\label{decaying_curvature_tensor}
		\vert\vert x_{max}^{-\left(k+1\right)}\nabla^k \left(Rm(g_{QAC})\ast V\right)\vert\vert_{C^0(\mathring{X})} + \vert\vert  x_{max}^{-k}\nabla^{k+1} V\vert\vert_{C^0(\mathring{X})} <\infty
	\end{equation}
	for any positive integer $k$. In particular, equation $(\ref{bounded_curvature_tensor})$ is satisfied for any $QAC-$vector field $V$.\\
	\begin{proof}
		Let us note that both $Rm(g_{QAC})\ast V$ and $\nabla V$ are polyhomogeneous tensors. Then using proposition \ref{bounded_polyhomogeneous_tensor} , the equation $(\ref{bounded_curvature_tensor})$ implies that both $\left(Rm(g_{QAC})\ast V\right)$ and $\nabla V$ are polyhomogeneous tensors with respect to a non-negative family indices. Equation $(\ref{decaying_curvature_tensor})$ is a combination of propositions \ref{covariant_polyhomogeneous_tensor} and \ref{bounded_polyhomogeneous_tensor}.\\
		The last assertion follows from the fact that $QAC-$vector fields are smooth up to the boundary.
	\end{proof}
\end{cor}

\subsubsection{Adjoints of differential operators}

Let $\left(E, \left<,\right>_E\right)$ be a hermitian vector bundle over $\mathring{X}$, and $C^{\infty}_{\mathcal{C}}(\mathring{X}, E)$ the space of compactly supported sections of $E$. We consider the hermitian inner product on $C^{\infty}_{\mathcal{C}}(\mathring{X}, E)$ defined by:
\begin{equation}
	\label{hermitian_inner_product}
	\left<\alpha, \beta\right> :=\int_{\mathring{X}}\left<\alpha,\beta \right>_E\:dV
\end{equation}
such that $\operatorname{dV}$ is the volume element induced by $g_{QAC}$ over $\mathring{X}$.
\begin{lemm}
	\label{adjoint_of_operator}
	Let $D\in\operatorname{Diff}_{{\mathcal{V}_{QAC}},\mathcal{G}}(X, E)$ be a differential operator and $D^{\#}$ its formal adjoint. Then, $D^{\#}\in\operatorname{Diff}_{{\mathcal{V}_{QAC}},\mathcal{K}}(X, E)$ for some non-negative index family $\mathcal{K}$.\\
	\begin{proof}
	As we noted in definition \ref{def_differential_operators}, elements of $\operatorname{Diff}_{{\mathcal{V}_{QAC}},\mathcal{G}}(X, E)$ are locally linear combination of the composition of operators of the form $X\otimes A$ such that $X\in\mathcal{V}_{QAC}(X)$ and $A$ a smooth family of endomorphisms of $\mathbb{C}^N$. We know that the adjoint of an endomorphism is again an endomorphism. It only remains to prove that the adjoint of an operator of the form $X\otimes 1$ is in $\operatorname{Diff}_{{\mathcal{V}_{QAC}},\mathcal{K}}(X, E)$.\\
	Let $\alpha,\beta\in C^{\infty}_{\mathcal{C}}(\mathring{X}, E)$, and $U$ a trivializing neighborhood of $E$ with respect to a unitary frame. Then 
	\begin{align}
		&\alpha_{\vert_U} = f_i\otimes z_i, \\
		&\beta_{\vert_U} = h_j\otimes w_i\:,\: f_i,h_j\in C^{\infty}(U), \text{and }z_i,w_j\in\mathbb{C}^N.
	\end{align}
	For the sake of simplicity we will suppose that $\alpha_{\vert_U} = f\otimes z$ and $\beta_{\vert_U} = h\otimes w$. Then
	\begin{align*}
		&\int_{U}\left<X\otimes 1 \alpha,\beta \right>_E\:\operatorname{dV} = \int_{U}\left<\alpha,X^{\#}\otimes 1\beta \right>_{\mathbb{C}^N}\:\operatorname{dV} \\
		=& \int_{U}X(f)\overline{h} \left<z,w\right>_{\mathbb{C}^N}\:\operatorname{dV} =  \int_{U}f\overline{X^{\#}(h)} \left<z,w\right>_{\mathbb{C}^N}\:\operatorname{dV}
	\end{align*}
	where $\left<,\right>_{\mathbb{C}^N}$ is the standard hermitian product on $\mathbb{C}^N$. So $X^{\#}$ is such that
	\begin{equation*}
		\int_{\mathring{X}}X(f)\overline{h} \:\operatorname{dV} =  \int_{\mathring{X}}f\overline{X^{\#}(h)}\: \operatorname{dV}
	\end{equation*} 
	for any $f,h\in C^{\infty}_{\mathcal{C}}(\mathring{X})$. Since $\int_{\mathring{X}}X(f)\overline{h} \:\operatorname{dV} = \int_{\mathring{X}}X(f\overline{h}) \:dV-\int_{\mathring{X}}f\overline{X(h)} \:\operatorname{dV}$ and that\\ $0=\int_{\mathring{X}}\operatorname{div}(f\overline{h}X) \:\operatorname{dV} =  \int_{\mathring{X}}f\overline{h}\operatorname{div}(X) \:\operatorname{dV}+\int_{\mathring{X}}X(f\overline{h}) \:\operatorname{dV}$, we obtain :
	\begin{equation*}
		\int_{\mathring{X}}X(f)\overline{h} \:\operatorname{dV} =  -\int_{\mathring{X}}f\overline{h}\operatorname{div}(X) \:\operatorname{dV} -\int_{\mathring{X}}f\overline{X(h)} \:\operatorname{dV} = \int_{\mathring{X}}f\overline{-\left(\operatorname{div X}+X(h)\right)} \:\operatorname{dV}
	\end{equation*}
	This implies that $X^{\#} = -\operatorname{div X} - X$. We recall that $\operatorname{div X} = \operatorname{tr}_{g_{QAC}}\nabla X$, which implies that $\operatorname{div X}\in x_{max}\mathcal{A}^{\mathcal{K}}_{phg}(X)$ (by proposition \ref{connection_polyhomogeneous_covariant}) for some non-negative index family $\mathcal{K}$.
\end{proof}
	
	\begin{cor}
		\label{corollary_adjoint_operator}
		Let $E_1$ and $E_2$ be two Hermitian vector bundles over $X$ and $D\in\operatorname{Diff}_{{\mathcal{V}_{QAC}},\mathcal{G}}(X, E_1, E_2)$. Then the formal adjoint of $D$ is in $\operatorname{Diff}_{{\mathcal{V}_{QAC}},\mathcal{K}}(X, E_2, E_1)$ for some non-negative index family.\\
		\begin{proof}
			Let define $E= E_1\bigoplus E_2$ and use the natural matrix (block) notation to describe  $\operatorname{Diff}_{{\mathcal{V}_{QAC}},\mathcal{G}}(X, E_1, E_2)$ as a subset of $\operatorname{Diff}_{{\mathcal{V}_{QAC}},\mathcal{G}}(X, E)$, then apply lemma \ref{adjoint_of_operator}.
		\end{proof}
	\end{cor}
	
	\begin{prop}
		\label{proposition_differential_operator}
		The exterior derivative $d$ is a differential operator in $\operatorname{Diff}_{{\mathcal{V}_{QAC}}}(X, \Lambda^p{}^{\pi}T^*X, \Lambda^{p+1}{}^{\pi}T^*X)$. In particular, $d\in x_{max}\operatorname{Diff}_{{\mathcal{V}_{Qb}}}(X, \Lambda^p{}^{\pi}T^*X, \Lambda^{p+1}{}^{\pi}T^*X)$.\\
		\begin{proof}
			Let $\omega\in\Gamma\left(\Lambda^p{}^{\pi}T^*X\right)$, and $\left\{X_i\right\}_{1\le i\le n}$ a local frame of $QAC-$vector fields. Then, we have :
			\begin{align*}
				\left(d\omega\right)\left(X_{i_0},\dots, X_{i_p}\right) = &\sum_{j=0}^{p} \left(-1\right)^j X_{i_j}\left(\omega\left(X_{i_0},\dots, \widehat{X_{i_j}},\dots,X_{i_p}\right)\right)-\\
				&\sum_{0\le s<t\le p}^{p} \left(-1\right)^{s+t} \omega\left(\left[X_{i_s}, X_{i_t}\right],X_{i_0},\dots,\widehat{X_{i_s}},\dots,\widehat{X_{i_t}},\dots,X_{i_p}\right)
			\end{align*}
			Since $\left[X_i, X_j\right] = c_{ij}^s X_s$ where $c_{ij}^s\in x_{max}C^{\infty}(X)$, we deduce that $d\in \operatorname{Diff}_{{\mathcal{V}_{QAC}}}(X, \Lambda^p{}^{\pi}T^*X, \Lambda^{p+1}{}^{\pi}T^*X)$. Actually we can see that $d\in x_{max}\operatorname{Diff}_{{\mathcal{V}_{Qb}}}(X, \Lambda^p{}^{\pi}T^*X, \Lambda^{p+1}{}^{\pi}T^*X)$
		\end{proof}
	\end{prop}
	
	\begin{cor}
			\label{corol_laplace_polyhomogene}
		The Hodge-Laplace operator on $\left(\mathring{X}, g_{QAC}\right)$ defined by $\Delta_{g_{QAC}}=\left(d+d^*\right)^2$ is a differential operator in  $\operatorname{Diff}_{{\mathcal{V}_{QAC}},\mathcal{K}}(X, \Lambda^p{}^{\pi}T^*X)$ for some non-negative index family $\mathcal{K}$.  In particular, $\Delta_{g_{QAC}}\in x_{max}^2\operatorname{Diff}_{{\mathcal{V}_{Qb}},\mathcal{K}}(X, \Lambda^p{}^{\pi}T^*X)$.\\
		\begin{proof}
			The proof is a combination of corollary \ref{corollary_adjoint_operator} and proposition \ref{proposition_differential_operator}.
		\end{proof}
	\end{cor}
\end{lemm}

	\section[The isomorphism theorem]{The isomorphism theorem}
\label{section_isomorphism_theorem}

Let $\left(X, g_{QAC}\right)$ be a QAC-manifold such that $g_{QAC}$ is a polyhomogeneous metric. We will denote by $(H_i)_{1\le i\le k}$ the boundary hypersurfaces of $X$, and  $(x_i)_{1\le i\le k}$ their respective defining functions. We define $x_{max}$ as the product of boundary defining functions of maximal hypersurfaces, and $v=\prod_{i=1}^{k}x_i$. \\The Riemannian manifold $\left(M=X\backslash\partial X, g_{QAC}\right)$ is a complete manifold of bounded geometry and positive injectivity radius (proposition 1.3 in \cite{rochon_QAC}).\\ Let $E=\left(T^*M^{\otimes^r}\otimes TM^{\otimes^s}\right)$ be a tensor bundle over $M$, and $\mathcal{A}$ the elliptic operator defined by:
\begin{equation}
	\label{soliton_elleptic_operator}
	\mathcal{A}=\underbrace{\Delta +\nabla_V}_{\Delta_V}-\lambda
\end{equation}
acting on sections of $E$, such that:
\begin{itemize}
	\item $V$ is a b-vector field on $X$.
	\item $\lambda$ is a positive constant.
	\item $\Delta$ and $\nabla$ are the Laplacian and the Levi-Civita connection of $g_{QAC}$ respectively.
\end{itemize}

In this section, we will prove that $\mathcal{A}:D^{k+2,\theta}_{\Delta_V, f}\left(M, E\right)\to C^{k,\theta}_{Qb, f}\left(M, E\right)$ is an isomorphism of Banach spaces for some positive function $f$ to be defined later and such that $D^{k+2,\theta}_{\Delta_V, f}\left(M, E\right)$ and $C^{k,\theta}_{Qb, f}\left(M, E\right)$ are as defined below.

\subsection{Function spaces}
\label{section_functional_spaces}
In the following functional spaces, we will consider the norm with respect to $g_{QAC}$ and the euclidean structure on $E$. Covariant derivatives will be taken with respect to the Levi-Civita connection of $g_{QAC}$ and the connection on $E$. Motivated by the work of \cite{Siepmann} and \cite{smoothing_deruelle}, we define the following weighted holder spaces:

\begin{itemize}
	\item $C_{Qb}^{k,\theta}\left(M, E\right):=\left\{h\in C^{k,\theta}_{loc}\left(M, E\right)\;\vert\; \vert\vert h\vert\vert_{C_{Qb}^{k,\theta}\left(M, E\right)}<\infty\right\}$, where
	\begin{align*}
		\vert\vert h\vert\vert_{C^{k}_{Qb}\left(M, E\right)} &:= \sum_{i=0}^{k}\:\sup\limits_M\vert x_{max}^{-i} \nabla^ih\vert,  \\
		\vert\vert h\vert\vert_{C^{k,\theta}_{Qb}\left(M, E\right)} &:= \vert\vert h\vert\vert_{C^{k}_{Qb}\left(M, E\right)} + \left[x_{max}^{-k} \nabla^kh\right]_{\theta},
	\end{align*}
	and
	\begin{equation}
		\label{seminorm_qb}
		\left[T\right]_{\theta} := \sup\limits_{x\in M}\:\sup\limits_{y\in B\left(x,\delta\right)\backslash\left\{x\right\}}\frac{\vert T(x)-P^*_{x,y}T(y)\vert}{d(x,y)^{\theta}},
	\end{equation}
	$P_{x,y}$ being the parallel transport along the unique minimizing geodesic from $x$ to $y$, and $\delta$ the injectivity radius of $g_{QAC}$.
	\begin{remq}
		The space $C_{Qb}^{k,\theta}\left(M, E\right)$ defined here is different from the one considered in \cite{rochon_QAC}, since in $(\ref{seminorm_qb})$ it is the distance of the QAC-metric which is used instead of the distance of the Qb-metric.
	\end{remq}
	
	Given an elliptic differential operator $\mathcal{P}$ acting on sections of $E$, we also define the following spaces:
	\item  $D^{2+k}_{\mathcal{P}}(M, E):= \left\{h\in \bigcap_{p\ge 1}  W_{loc}^{2+k,p}(M, E)\;\vert\; h\in C^k_{Qb}(M,E)\;;\; \mathcal{P}\left(h\right)\in C^k_{Qb}(M,E)\right\}$, with the norm
	\begin{equation*}
		\vert\vert h\vert\vert_{D^{2+k}_{\mathcal{P}}(M, E)} := \vert\vert h\vert\vert_{C^k_{Qb}(M, E)}+\vert\vert\mathcal{P}\left(h\right)\vert\vert_{C^k_{Qb}(M, E)}.
	\end{equation*}
	\item  $D^{{2+k},\theta}_{\mathcal{P}}(M, E):= \left\{h\in C^{k+2,\theta}_{loc}(M,E)\;\vert\; h\in C^{k,\theta}_{Qb}(M,E)\;;\; \mathcal{P}\left(h\right)\in C^{k,\theta}_{Qb}(M,E)\right\}$, with the norm
	\begin{equation*}
		\vert\vert h\vert\vert_{D^{{2+k},\theta}_{\mathcal{P}}(M, E)} := \vert\vert h\vert\vert_{C^{k,\theta}_{Qb}(M, E)}+\vert\vert\mathcal{P}\left(h\right)\vert\vert_{C^{k,\theta}_{Qb}(M, E)}.
	\end{equation*}
\end{itemize}

The following weighted spaces are defined using a positive function $f$ to be defined later:
\begin{itemize}
	\item $C^{k,\theta}_{Qb, f}\left(M, E\right):= f^{-1}C^{k,\theta}_{Qb}\left(M, E\right)$ with the norm $\vert\vert h\vert\vert_{C^{k,\theta}_{Qb, f}\left(M, E\right)}$:=$\vert\vert fh\vert\vert_{C^{k,\theta}_{Qb}\left(M, E\right)}$.
	\item $D^{k+2,\theta}_{\mathcal{P}, f}\left(M, E\right):= f^{-1}D^{k+2,\theta}_{\mathcal{P}}\left(M, E\right)$ with the norm $\vert\vert h\vert\vert_{D^{k+2,\theta}_{ \mathcal{P}, f}\left(M, E\right)}$:=$\vert\vert fh\vert\vert_{D^{k+2,\theta}_{\mathcal{P}}\left(M, E\right)}$.\\
\end{itemize}

\begin{remq}
	It is worth mentioning that weighted holder spaces were introduced to study the behavior of the Laplacian on non compact manifolds. For instance, \cite{ChaljubSimon1979ProblmesED} introduced the following spaces:
	\begin{equation*}
		C_{\beta,\rho}^{k,\theta}\left(M, E\right):=\left\{h\in C^{k,\theta}_{loc}\left(M, E\right)\;\vert\; \vert\vert h\vert\vert_{C_{\beta,\rho}^{k,\theta}\left(M, E\right)}:=\vert\vert h\vert\vert_{C_{\beta,\rho}^{k}\left(M, E\right)}+\left[\nabla^k h\right]_{\theta,\beta-k-\theta,\rho}<\infty\right\},
	\end{equation*}
	such that:
	\begin{align*}
	&\vert\vert h\vert\vert_{C_{\beta,\rho}^{k}\left(M, E\right)}= \sum_{i=0}^{k}\vert\vert\rho^{i-\beta} \nabla^{j}h\vert\vert_{C^0\left(M, E\right)},\\
	&\left[T\right]_{\theta,\gamma,\rho} := \sup\limits_{x\in M}\:\sup\limits_{y\in B\left(x,\delta\right)\backslash\left\{x\right\}}\inf(\rho(x),\rho(y))^{-\gamma}\frac{\vert T(x)-P^*_{x,y}T(y)\vert}{d(x,y)^{\theta}},
	\end{align*}
	to study linear elliptic operators on asymptotically euclidean manifolds ($\rho$ being some distance function). These spaces were adapted by Joyce in \cite{joyceALE} and \cite{joyceQALE} to study asymptotically locally euclidean manifolds and quasi-asymptotically euclidean manifolds. Note that $C_{0,x_{max}^{-1}}^{k,\theta}\left(M, E\right)\subset C_{Qb}^{k,\theta}\left(M, E\right)$.
\end{remq}

	\begin{remq}
	It is also important to note that the space $C_{Qb}^{k,\theta}\left(M, E\right)$ as defined here is not equal to the interpolation space $\left(C_{Qb}^{k}\left(M, E\right), C_{Qb}^{k+1}\left(M, E\right)\right)_{\theta,\infty}$ which can be identified as follows:
	\begin{equation*}
		\left(C_{Qb}^{k}\left(M, E\right), C_{Qb}^{k+1}\left(M, E\right)\right)_{\theta,\infty} = \left\{h\in C^k_{Qb}(M, E)\:\vert\: \left[x_{max}^{-k}\nabla^kh\right]_{Qb,\theta}<+\infty\right\}
	\end{equation*}
	such that:
	\begin{equation*}
		\left[T\right]_{Qb,\theta} = \sup\limits_{x\in M}\sup\limits_{y\in B(x, \frac{\delta}{x_{max}})\backslash\left\{x\right\}}\min\left\{x_{max}^{-\theta}(x),x_{max}^{-\theta}(y)\right\}\frac{\vert T(x)-P^*_{x,y}T(y)\vert}{d(x,y)^{\theta}}
	\end{equation*}
	where $\delta$ is a positive constant depending on the lower bound of $\inf\limits_{x\in M}inj(x,g_{QAC})x_{max}(x)$.
\end{remq}

\begin{prop}
	\label{proposition_D_banach_space}
	Let $\mathcal{P}$ be an elliptic differential operator and $\lambda$ a constant such that \\
	$\mathcal{P}-\lambda: D^{k+2,\theta}_{\mathcal{P}}(M, E)\to C^{k,\theta}_{Qb}(M, E)$ is an isomorphism. Then, the space $D^{{2+k},\theta}_{\mathcal{P}}(M, E)$ is a Banach space. \\
	\begin{proof}
		We are going to use the fact that $C_{Qb}^{k,\theta}\left(M, E\right)$ is a Banach space. Let $\left(h_k\right)_{k\in\mathbb{N}}$ be a Cauchy sequence in $D^{{2+k},\theta}_{\mathcal{P}}(M, E)$. Then, there exists $h\in C_{Qb}^{k,\theta}\left(M, E\right)$ such that $h_k$ converges to $h$ in $C_{Qb}^{k,\theta}\left(M, E\right)$.\\
		Since $((\mathcal{P}-\lambda)(h_k))$ is also a Cauchy sequence in $C_{Qb}^{k,\theta}\left(M, E\right)$, there exists $\tilde{h}\in C_{Qb}^{k,\theta}\left(M, E\right)$ such that $(\mathcal{P}-\lambda)(h_k)$ converges to $(\mathcal{P}-\lambda)(\tilde{h})$ . Note also that $(\mathcal{P}-\lambda)(h_k)$ converges to $(\mathcal{P}-\lambda)(h)$ in the $C^{k,\theta}_{loc}(M,E)$ topology which implies that $(\mathcal{P}-\lambda)(\tilde{h}) = (\mathcal{P}-\lambda)(h)$. Consequently, $h\in D^{{2+k},\theta}_{\mathcal{P}}(M, E)$.
	\end{proof}
\end{prop}

\subsection{Lunardi's theorem on QAC manifolds}

In order to prove the isomorphism theorem, we are going to use the following theorem:

\begin{theo}[Lunardi]
	\label{theoreme_lunardi_2}
	Let $\left(M^n, g\right)$ be a complete Riemannian manifold with positive injectivity radius, and $V$ be a smooth vector field on $M$. Let $\mathcal{P}$ be an elliptic differential operator acting on tensors over $M$ such that:
	\begin{equation}
		\mathcal{P} = \underbrace{\Delta +\nabla_V}_{\Delta_V}+ r(x)\:,\:r\in C^3(M).
	\end{equation}
	Suppose that $sup_{x\in M} r(x)=r_0<\infty$ and that there exists a positive constant $C$ such that $\sum_{i=1}^{3}\vert\vert\nabla^i r\vert\vert<C$. Assume also that there exists a positive constant $K$ such:\\
	\begin{equation}
		\label{lunardi_requirement}
		\vert\vert Rm(g)\vert\vert_{C^3(M,E)}+\vert\vert Rm(g)\ast V\vert\vert_{C^3(M,E)}+\vert\vert \nabla V\vert\vert_{C^{2}(M,E)}\le K,
	\end{equation}
	where $Rm(g)\ast V=Rm(g)(V,.,.,.)$. Assume also that there exists a function $\phi\in C^2(M)$ and a constant $\lambda_0\ge r_0$ such that:
	\begin{equation}
		\label{lunardi_condition_phi_2}
		\lim_{x\to\infty}\phi (x) = +\infty, \: sup_{x\in M}\left(\mathcal{P}(\phi)(x)-\lambda_0\phi\left(x\right)\right)<\infty.
	\end{equation}
	Then:
	\begin{enumerate}
		\item For any  $\lambda>r_0$, there exists a positive constant $C$ such that for any $H\in C^0\left(M, E\right)$, there exists a unique tensor $h\in D^2_{\mathcal{P}}\left(M, E\right)$, satisfying:
		\begin{equation*}
			\mathcal{P}(h) -\lambda h= H,\: \vert\vert h\vert\vert_{D^2_{\mathcal{P}}\left(M,E\right)}\le C \vert\vert H\vert\vert_{C^0\left(M, E\right)}.
		\end{equation*}
		Moreover $D^2_{\mathcal{P}}\left(M, E\right)$is continuously embedded in $C^{\theta}\left(M,E\right)$ for any $\theta\in\left(0,2\right)$, i.e there exists a positive constant $C(\theta)$ such that for any $h\in D^2_{\mathcal{P}}\left(M, E\right)$, 
		\begin{equation*}
			\vert\vert h\vert\vert_{C^{\theta}(M, E)}\le C(\theta)\vert\vert h\vert\vert^{\frac{\theta}{2}}_{D^2_{\mathcal{P}}\left(M, E\right)} \vert\vert h\vert\vert^{1-\frac{\theta}{2}}_{C^0(M, E)}.
		\end{equation*}
		\item For any $\lambda>r_0$, there exists a positive constant $C$ such that for any $H\in C^{0,\theta}\left(M, E\right)$, $\theta\in \left(0,1\right)$ , there exists a unique tensor $h\in C^{2,\theta}\left(M, E\right)$ satisfying:
		\begin{equation*}
			\mathcal{P}(h) - \lambda h= H,\: \vert\vert h\vert\vert_{C^{2,\theta}\left(M, E\right)}\le C \vert\vert H\vert\vert_{C^{0,\theta}\left(M, E\right)}.
		\end{equation*}
	\end{enumerate}
	\end{theo}
We will prove theorem \ref{theoreme_lunardi_2} in the next section.
\begin{remq}
	\label{remark_lunardi_condition}
	Note that in order to satisfy condition $(\ref{lunardi_requirement})$ for the $b-$vector field $V$, it is sufficient to have
	\begin{equation}
		\label{lunardi_condition_0}
		\vert\vert Rm(g)\ast V\vert\vert_{C^0(M,E)}+\vert\vert \nabla V\vert\vert_{C^{0}(M,E)}<\infty.
	\end{equation}
	This is a direct consequence of corollaries \ref{covariant_polyhomogeneous_tensor} and \ref{polyho_lunardi_condition}.  Note also that condition $(\ref{lunardi_condition_0})$ is satisfied if the vector field $V$ is a $QAC-$vector field.
\end{remq}

Let's start by showing that condition $(\ref{lunardi_condition_phi_2})$ is easily satisfied.

\begin{prop}
	\label{exaustion_function_qac_manifold}
	There exists a smooth function $\phi:M\to\mathbb{R}_+$ such that:
	\begin{equation*}
		\lim_{p\to\infty}\phi (p) = +\infty, \: \sup\limits_{p\in M}\left(\Delta_X\right)\left(\phi\right)<\infty,
	\end{equation*}
	for any b-vector field $X$.\\
	\begin{proof}
		Let us set $\phi(p)=-\ln(v(p))$. Using the definition of a b-vector field and the fact that the Laplacian of $g_{QAC}$ can be expressed as a polynomial of degree at most 2 (without terms of order $0$) in QAC-vector fields \cite{Ammann_2004} we see that both $\Delta \phi$ and $\nabla_X \phi$ are bounded on $M$.
	\end{proof}
\end{prop}
Let $\mathcal{A}_{f}$ be the differential operator defined by $\mathcal{A}_f(h)=f\mathcal{A}(f^{-1}h)$. Then, since:
\begin{align*}
	f\Delta(f^{-1}h) &= f\left(f^{-1}\Delta h + h\Delta f^{-1} + 2<\nabla f^{-1}, \nabla h>\right)\\
	&=\Delta h + h f\Delta f^{-1}-2<\nabla \ln(f), \nabla h>\\
	&=\Delta h + h\left(\vert\vert\nabla\ln(f)\vert\vert^2-\Delta \ln(f)\right)-2<\nabla \ln(f), \nabla h>,\\
	&\text{and}\\
	f\nabla_V(f^{-1}h) &= f\left(f^{-1}\nabla_V h + h\nabla_V f^{-1}\right)f = \nabla_V h - h\nabla_V\ln(f),
\end{align*}
we have that:
\begin{equation}
	\label{conjugated_operator_A}
	\mathcal{A}_f(h)=\left(\underbrace{\Delta +\nabla_{V-2\nabla\ln(f)}-V\ln(f)}_{\mathcal{P}_f}-\lambda\right)(h)+\underbrace{\left(\vert\vert\nabla\ln(f)\vert\vert^2-\Delta\ln(f)\right)}_{\mathcal{K}_f}h.
\end{equation}

\begin{remq}
	\label{remark_original_operator_weighted_operator}
	The operator $\mathcal{A}:D^{k+2,\theta}_{\Delta_V, f}\left(M, E\right)\to C^{k,\theta}_{Qb, f}\left(M, E\right)$ is an isomorphism of Banach spaces if and only if $\mathcal{A}_f:D^{k+2,\theta}_{\Delta_V}\left(M, E\right)\to C^{k,\theta}_{Qb}\left(M, E\right)$ is. In what follows, we are going to set $f = v^{-\alpha}$ for some positive real value $\alpha$. \\We are also going to use the following notation:
	\begin{itemize}
		\item  $r_{\alpha}:= \alpha V\ln(v)$;
		\item  $V_{\alpha}:= V+2\alpha\nabla\ln(v)$;
		\item $\mathcal{P}_{\alpha} :=\Delta_{V_{\alpha}}+r_{\alpha}$;
		\item $\mathcal{K}_{\alpha}:=\alpha^2\vert\vert\nabla\ln(v)\vert\vert^2-\alpha\Delta\ln(v)$;
		\item $C^{k,\theta}_{Qb,\alpha}(M, E) := C^{k,\theta}_{Qb, f}(M, E)$;
		\item $D^{k+2,\theta}_{\Delta_V,\alpha}(M, E) := D^{k+2,\theta}_{\Delta_V,f}(M, E)$.
	\end{itemize}
\end{remq}

\begin{prop}
	\label{condition_lunardi_remark}
	Given a function $f$ as defined in the previous remark, we have that:
	\begin{itemize}
		\item[$(i)$] $r_{\alpha}$ is bounded on $M$.
		\item[$(ii)$] $\nabla \ln(v)$ is a QAC-vector field. Thus, both $\vert\vert Rm(g)*\nabla\ln(v)\vert\vert_{C^k(M, E)}$ and $\vert\vert \nabla \nabla\ln(v)\vert\vert_{C^{k-1}(M, E)}$ are bounded for any integer $k\ge 1$.
		\item[$(iii)$] $\lim\limits_{x\to\infty}\vert\vert \nabla^i\mathcal{K}_{\alpha}\vert\vert = 0$ for $i\ge 0$.
	\end{itemize}
	\begin{proof}
		$(i):$ This follows from the proof of proposition \ref{exaustion_function_qac_manifold}.\\
		$(ii):$ $d\ln(f) = -\alpha v\frac{dv}{v^2}$ which is clearly a QAC-covector (see equation $(\ref{qfb_forms})$) that tends to $0$ near the boundary. The rest follows from the fact that both $Rm(g_{QAC})$ and $\nabla\ln(f)$ are tensors over the QAC-vector bundle over $X$, thus bounded with respect to the QAC-metric together with its derivatives.\\
		$(iii):$ Since the Laplacian is a polynomial on $QAC-$vector fields without a constant term, we have that $\Delta\ln(v)\in vC^{\infty}(X)$. Taking covariant derivatives will increase the decay towards maximal hypersurfaces.
	\end{proof}
\end{prop}

\begin{prop}
	For any $h\in C^{k,\theta}_{Qb}(M, E)$ we have that $\mathcal{K}_{\alpha} h\in C^{k,\theta}_{Qb}(M, E)$, we also have that
	\begin{equation}
		\label{equivalence_of_space}
		\Delta_V(h)\in C^{k,\theta}_{Qb}(M, E) \Longleftrightarrow \mathcal{P}_{\alpha} (h)\in C^{k,\theta}_{Qb}(M, E) 
	\end{equation}
	As a consequence, we have that $D^{k+2,\theta}_{\mathcal{P}_{\alpha}+\mathcal{K}_{\alpha} }(M, E) = D^{k+2,\theta}_{\mathcal{P}_{\alpha}}(M, E) = D^{k+2,\theta}_{\Delta_V}(M, E)$.\\
	\begin{proof}
		$\mathcal{K}_{\alpha} h\in C^{k,\theta}_{Qb}(M, E)$ follows from $(iii)$ of proposition  \ref{condition_lunardi_remark} and proposition \ref{prop_product_func_section_holderestimate}. Since we have that
		\begin{align*}
			&\vert \vert\vert \mathcal{P}_{\alpha}(h)+ \mathcal{K}_{\alpha}h\vert\vert_{C^{k,\theta}_{Qb}(M, E)}\le \vert \vert\vert \mathcal{P}_{\alpha}(h)\vert\vert_{C^{k,\theta}_{Qb}(M, E)}+ \vert\vert\mathcal{K}_{\alpha}h\vert\vert_{C^{k,\theta}_{Qb}(M, E)}, \\
			& \text{and} \\
			&\vert \:\: \vert\vert \mathcal{P}_{\alpha}(h)\vert\vert_{C^{k,\theta}_{Qb}(M, E)}- \vert\vert\mathcal{K}_{\alpha}h\vert\vert_{C^{k,\theta}_{Qb}(M, E)} \vert \le \vert \vert\vert \mathcal{P}_{\alpha}(h)+ \mathcal{K}_{\alpha}h\vert\vert_{C^{k,\theta}_{Qb}(M, E)},
		\end{align*}
		which implies that $\left(\mathcal{P}_{\alpha}+\mathcal{K}_{\alpha}\right)(h)\in C^{k,\theta}_{Qb}(M, E) \Longleftrightarrow \mathcal{P}_{\alpha} (h)\in C^{k,\theta}_{Qb}(M, E)$.  We proceed in the same manner using $(i)$ and $(iii)$ of proposition \ref{condition_lunardi_remark} to prove $(\ref{equivalence_of_space})$.
	\end{proof}
\end{prop}

\begin{remq}
	Going back to remark \ref{remark_original_operator_weighted_operator}, in order to prove that $\mathcal{A}_f:D^{k+2,\theta}_{\Delta_V}\left(M, E\right)\to C^{k,\theta}_{Qb}\left(M, E\right)$ is an isomorphism, it's suffice to prove that $\mathcal{P}_{\alpha}-\lambda:D^{k+2,\theta}_{\mathcal{P}_{\alpha}}\left(M, E\right)\to C^{k,\theta}_{Qb}\left(M, E\right)$ is an isomorphism and that $\mathcal{K}_{\alpha}$ is a compact operator.
\end{remq}

\subsection{The isomorphism theorem}

\begin{theo}[Isomorphism theorem]
	\label{isomorphism_theorem}
	Let $\mathcal{C}^{k;j,\theta}(M, E)$ be the functional space defined by:
	\begin{equation*}
		\mathcal{C}^{k;j,\theta}(M, E) = \left\{h\in C^{k+j+\lfloor\theta\rfloor,\theta-\lfloor\theta\rfloor}_{loc}(M, E)\:\vert\: x_{max}^{-i}\nabla^ih\in C^{j+\lfloor\theta\rfloor,\theta-\lfloor\theta\rfloor}(M, E),\:\:\forall i=0,\dots,k\right\}.
	\end{equation*}
	such that $\theta\in(0,2)$, and endowed with the norm:
	\begin{equation*}
		\vert\vert h\vert\vert_{\mathcal{C}^{k;j,\theta}(M, E)} = \sum_{i=0}^{k}\vert\vert x_{max}^{-i}\nabla^ih\vert\vert_{C^{j+\lfloor\theta\rfloor,\theta-\lfloor\theta\rfloor}(M, E)}
	\end{equation*}
	Suppose also that:
	\begin{equation}
		\label{condition_lunardi_k_superieur}
			\vert\vert Rm(g)\ast V\vert\vert_{C^0(M,E)}+\vert\vert \nabla V\vert\vert_{C^{0}(M,E)}< \infty
	\end{equation}
	Then, for any constant $\lambda\in \mathbb{R}$ such that:
	\begin{equation}
		\label{condition_on_lambda_k_superior}
	    \lambda>\max\left(\sup\limits_{M}V\ln(v^{\alpha}x_{max}^k),\sup\limits_{M}V\ln(v^{\alpha}x_{max}^{k-1})\right)
	\end{equation}
	we have that:
	\begin{itemize}
		\item There exists a positive constant $C$ such that for any $H\in C^{k,\theta}_{Qb}(M, E)$ there exists a unique $h\in D^{k+2,\theta}_{\mathcal{P}_{\alpha}}(M, E)$ satisfying:
		\begin{equation}
			\mathcal{P}_{\alpha}(h) - \lambda h=H,\: \vert\vert h\vert\vert_{D^{k+2,\theta}_{\mathcal{P}_{\alpha}}(M, E)}\le C\vert\vert H\vert\vert_{C^{k,\theta}_{Qb}(M, E)};\: \theta\in [0,1)
		\end{equation}
		i.e. the operator 
		\begin{equation}
			\label{iso_first_estimate_of_h}
			\mathcal{P}_{\alpha}-\lambda: D^{k+2,\theta}_{\mathcal{P}_{\alpha}}(M, E)\to C^{k,\theta}_{Qb}(M, E)
		\end{equation}
		is an isomorphism of Banach spaces. Moreover, $D^{k+2}_{\mathcal{P}_{\alpha}}(M, E)$ embeds continuously in $\mathcal{C}^{k;0,\theta}(M, E)$ for any $\theta\in(0,2)$, i.e there exists a positive constant C such that for any $h\in D^{k+2}_{\mathcal{P}_{\alpha}}(M, E)$,
		\begin{equation*}
			\vert\vert h\vert\vert_{\mathcal{C}^{k;0,\theta}(M, E)}\le C \vert\vert h\vert\vert^{\frac{\theta}{2}}_{D^{k+2}_{\mathcal{P}_{\alpha}}(M, E)}\vert\vert h\vert\vert^{1-\frac{\theta}{2}}_{C^k_{Qb}(M, E)}
		\end{equation*}
		\item There exists a positive constant $C$ such that, for $\theta\in(0,1)$
		\begin{equation}
			\label{iso_second_estimate_of_h}
			\vert\vert h\vert\vert_{\mathcal{C}^{k;2,\theta}(M, E)} \le C\vert\vert H\vert\vert_{C^{k,\theta}_{Qb}(M, E)}
		\end{equation}
	\end{itemize}
\end{theo}

In order to prove the previous theorem, we are going to proceed by induction on $k$. Let us consider the case $k=0$.

\begin{theo}[Isomorphism theorem (k=0)]
	\label{isomorphism_theorem_0}
	Suppose that:
	\begin{equation}
	 \label{condition_lunardi_k_0}
		\vert\vert Rm(g)\ast V\vert\vert_{C^0(M,E)}+\vert\vert \nabla V\vert\vert_{C^{0}(M,E)}<\infty.
	\end{equation}
	Then, for any constant $\lambda\in\mathbb{R}$ such that:
	\begin{equation}
		\label{condition_on_lambda_k_0}
		\lambda>\sup\limits_{M}\left(V\ln(v^{\alpha})\right),
	\end{equation}
	we have that:
	\begin{itemize}
		\item There exists a positive constant $C$ such that, for any $H\in C^0(M, E)$, there exists a unique tensor $h\in D^2_{\mathcal{P}_{\alpha}}(M, E)$ satisfying
		\begin{equation*}
			\mathcal{P}_{\alpha}(h)-\lambda h=H ,\: \vert\vert h\vert\vert_{D^2_{\mathcal{P}_{\alpha}}}\le C\vert\vert H\vert\vert_{C^0(M, E)}.
		\end{equation*}
		Moreover, $D^2_{\mathcal{P}_{\alpha}}(M, E)$ is continuously embedded in $C^{\theta}(M, E):= C^{\lfloor\theta\rfloor, \theta-\lfloor\theta\rfloor}(M, E)$ for any $\theta\in\left(0,2\right)$, i.e there exists a positive constant $C$ such that for any $h\in D^2_{\mathcal{P}_{\alpha}}(M, E)$,
		\begin{equation*}
			\vert\vert h\vert\vert_{C^{\theta}(M, E)}\le C\vert\vert h\vert\vert^{\frac{\theta}{2}}_{D^2_{\mathcal{P}_{\alpha}}(M, E)} \vert\vert h\vert\vert^{1-\frac{\theta}{2}}_{C^0(M, E)}.
		\end{equation*}
		\item There exists a positive constant $C$ such that, for any $H\in C^{0,\theta}(M, E)$, with $\theta\in(0,1)$, there exists a unique tensor $h\in C^{2,\theta}(M, E)$ satisfying
		\begin{equation*}
			\mathcal{P}_{\alpha}(h)-\lambda h=H ,\: \vert\vert h\vert\vert_{C^{2,\theta}(M, E)}\le C\vert\vert H\vert\vert_{C^{0,\theta}(M, E)}.
		\end{equation*}
		Moreover, the operator $\mathcal{A}:D^{2,\theta}_{\Delta_V,\alpha}(M, E)\to C^{0,\theta}_{Qb,\alpha}(M, E)$ is an isomorphism of Banach spaces.
	\end{itemize}
	\begin{proof}
		Condition \ref{condition_lunardi_k_0} together with remark \ref{remark_lunardi_condition} implies condition $(\ref{lunardi_requirement})$ of theorem \ref{theoreme_lunardi_2}. We also have that condition \ref{lunardi_condition_phi_2} is satisfied by proposition \ref{exaustion_function_qac_manifold}. This proves the first point and the first part of the second point. In order to prove that $\mathcal{A}$ is an isomorphism of Banach spaces we are going to use the fact that the index of a Fredholm operator remains unchanged under a perturbation by a compact operator.
		Thus, if $\mathcal{P}_{\alpha}-\lambda$ is an isomorphism and $\mathcal{K}_{\alpha}$ is compact, then $\mathcal{P}_{\alpha}+\mathcal{K}_{\alpha}-\lambda$ is of index $0$. Then, injectivity of $\mathcal{A}$ implies surjectivity.
		
		 $(1)$ $\mathcal{A}:D^{2,\theta}_{\Delta_V, \alpha}(M, E)\to C^{0,\theta}_{\alpha}(M, E)$ is injective.\\
		 Let $h\in D^{2,\theta}_{\Delta_V,\alpha}(M, E)$ be such that $\mathcal{A}(h)=0$ and $h_k = h - \frac{\phi}{k}$  with $\phi$ the smooth function of proposition \ref{exaustion_function_qac_manifold}. Then, $\sup\limits_{p\in M}h_k= h_k(p_k)$ for some $p_k\in M$. Moreover $\lim_{k\to\infty}\sup\limits_{p\in M}h_k=\sup\limits_{p\in M}h$.\\
		 Since $\mathcal{A}(h_k) = -\frac{\mathcal{A}(\phi)}{k}$, we have that $\mathcal{A}(h_k) \ge -\frac{\sup\limits_{p\in M}\mathcal{A}(\phi)}{k}$. Evaluating the last inequality at point $p_k$ we get $(\lambda - \sup\limits_{p\in M}V\ln(v^{\alpha})) h_k(p_k)\le \frac{\sup\limits_{p\in M}\mathcal{A}(\phi)}{k}$. By taking the limit $k\to\infty$ we find that $\sup\limits_{p\in M}h\le 0$. By applying the same method to $-h$ we deduce that $h=0$. Hence, $\mathcal{A}$ is injective.
		 
		 $(2)$ $\mathcal{K}_{\alpha}:D^{2,\theta}_{\mathcal{P}_{\alpha}}(M, E)\to C^{0,\theta}(M, E)$ is a compact operator.\\
		 Let $\left(h_k\right)_{k\in\mathbb{N}}$ be a bounded sequence in $D^{2,\theta}_{\mathcal{P}_{\alpha}}(M, E)$ and $\left(U_k\right)_{k\in\mathbb{N}}$ be a sequence of precompact open sets of $M$ such that $\overline{U_k}\subset U_{k+1}$ and $M=\cup_k U_k$.
		 By Schauder estimates \ref{schauder_estimate}, the sequence $\left(h_k\right)_k$ is bounded in $C^{2,\theta}(U_i, E\vert_{U_i})$ for any $i$.\\ 
		 Since $C^{2,\theta}(\overline{U_i}, E\vert_{\overline{U_i}})$ is compactly embedded into $C^2(\overline{U_i}, E\vert_{\overline{U_i}})$, there exists a sub-sequence $(h_k^i)_k$ that converges uniformly in $C^2(\overline{U_i}, E\vert_{\overline{U_i}})$. Let $(g_k)_k$ be a sub-sequence such that $g_k= h^k_k$. Then, $(g_k)_k$ converges in the topology of $C^2_{loc}(M, E)$ to $h\in D^{2,\theta}_{\mathcal{P}_{\alpha}}(M, E)$ (since it converges uniformly on every compact of $M$).\\
		 Before we finish the proof, we need the following lemma:
		 \begin{lemm}
		 	Let $h\in C^{0,\theta}(M, E)$ and $f\in C^{1}(M)$. Then, for any compact set $K\subset M$ there exists a precompact set $Q$ containing $K$ and a positive constant $C$ such that:
		 	\begin{equation*}
		 		\vert\vert f h\vert\vert_{C^{0,\theta}(M, E)}\le C \left(\vert\vert f\vert\vert_{C^1(Q)}  \vert\vert h\vert\vert_{C^{0,\theta}(Q, E\vert_{Q})} +\vert\vert f\vert\vert_{C^1(M\backslash K)}  \vert\vert h\vert\vert_{C^{0,\theta}(M\backslash K, E\vert_{M\backslash K})}\right).
		 	\end{equation*}
		 	\begin{proof}
		 		Let $Q$ be a precompact set containing $K$ such that $\forall x\in K$ we have that $B(x,\delta)\subset Q$ ($\delta$ being the injectivity radius). Then:
		 		\begin{align*}
		 			\vert\vert f h\vert\vert_{C^{0,\theta}(M, E)}&\le \left(	\vert\vert f h\vert\vert_{C^{0,\theta}(Q, E\vert_Q)} + \vert\vert f h\vert\vert_{C^{0,\theta}(M\backslash K, E\vert_{M\backslash K})} \right) \\
		 			&\le \left(	\vert\vert f \vert\vert_{C^{0,\theta}(Q)} \vert\vert h\vert\vert_{C^{0,\theta}(Q, E\vert_Q)} + \vert\vert f \vert\vert_{C^{0,\theta}(M\backslash K)} \vert\vert h\vert\vert_{C^{0,\theta}(M\backslash K, E\vert_{M\backslash K})} \right).
		 		\end{align*}
		 		Now using a local version of the mean value theorem, it is easy to see that there exists a positive constant $C$ (that depends only on the injectivity radius) such that:
		 		\begin{align*}
		 			\vert\vert f \vert\vert_{C^{0,\theta}(Q)} &\le C \vert\vert f \vert\vert_{C^{1}(Q)}, \\
		 			\vert\vert f \vert\vert_{C^{0,\theta}(M\backslash K)} &\le C \vert\vert f \vert\vert_{C^{1}(M\backslash K)} .
		 		\end{align*}
		 	 Consequently, we have that 
		 	 \begin{equation*}
		 	 	\vert\vert f h\vert\vert_{C^{0,\theta}(M, E)}\le C \left(\vert\vert f\vert\vert_{C^1(Q)}  \vert\vert h\vert\vert_{C^{0,\theta}(Q, E\vert_{Q})} +\vert\vert f\vert\vert_{C^1(M\backslash K)}  \vert\vert h\vert\vert_{C^{0,\theta}(M\backslash K, E\vert_{M\backslash K})}\right).
		 	 \end{equation*}
		 	\end{proof}
		 \end{lemm}
		 Since $\lim_{p\to\infty}\vert\nabla^i\mathcal{K}_{\alpha}\vert(p) = 0$ for $i=0, 1$ (by proposition \ref{condition_lunardi_remark}); it follows that for all $\epsilon>0$ there exists a compact set $K\subset M$ such that $\vert\vert \mathcal{K}_{\alpha}\vert\vert_{C^1(M\backslash K)}<\epsilon$.\\ Using the previous lemma, we have that
		 \begin{multline*}
		 	\vert\vert \mathcal{K}_{\alpha}\left(g_n-h\right)\vert\vert_{C^{0,\theta}(M, E)}\le C(\vert\vert \mathcal{K}_{\alpha}\vert\vert_{C^1(Q)}  \vert\vert \left(g_n-h\right)\vert\vert_{C^{0,\theta}(Q, E\vert_{Q})} +\\
		 	\vert\vert \mathcal{K}_{\alpha}\vert\vert_{C^1(M\backslash K)}  \vert\vert \left(g_n-h\right)\vert\vert_{C^{0,\theta}(M\backslash K, E\vert_{M\backslash K})}) < C\epsilon
		 \end{multline*}
		 for some precompact set $Q$ containing $K$ and $n$ large enough. This proves that $\left(\mathcal{K}_{\alpha}g_n\right)_{n\in\mathbb{N}}$ converges to $\mathcal{K}_{\alpha}h$ in the $C^{0,\theta}(M, E)$ topology, which proves that $\mathcal{K}_{\alpha}$ is a compact operator. 
	\end{proof}
\end{theo}

\begin{prop}
	\label{prop_non_bounded_function_maximum_principal}
	Let $f$ be a $C^2_{loc}(M)$ function such that:
	\begin{equation}
		\label{prop_p_of_f_equal_0}
		\mathcal{T}(f)= (\Delta_{V+2\alpha \nabla\ln(v)}+V\ln(v^{\alpha}x_{max}^k)-\lambda) (f) \ge 0\text{ and } f=O(x_{max}^{-1}),
	\end{equation}
	$\lambda$ being a constant such that:
	\begin{equation}
		\label{lambda_constant_lemma}
		\lambda>\max\left(\sup\limits_{M}V\ln(v^{\alpha}x_{max}^k),\sup\limits_{M}V\ln(v^{\alpha}x_{max}^{k-1})\right).
	\end{equation}
	Then, $\sup\limits_{M}f\le 0$. \\
	\begin{proof}
		   First of all, let us note that $f$ is only potentially unbounded near the maximal hypersurfaces (since $f=O(x_{max}^{-1})$). If $f$ is bounded above then we will use the exhaustion function of proposition \ref{exaustion_function_qac_manifold} to prove the proposition. Otherwise, $f$ is unbounded from above near maximal hypersurfaces, and we will use $x_{max}^{\theta}$ with $\theta<-1$ as a barrier function. \\
		   Before we proceed, let us note that inequality \ref{prop_p_of_f_equal_0} implies that:
		   \begin{equation}
		   	\label{prop_p_of_f_equal_0_first_inequality}
		   	\Delta_{V+2\alpha \nabla\ln(v)} (f)\ge (\lambda - V\ln(v^{\alpha}x_{max}^k)) f.
		   \end{equation}
		   In addition, given a function $f_s$ that attains its supremum at a point $p_s\in M$ we have that:
		   \begin{equation}
		   	\label{prop_p_of_f_equal_0_second_inequality}
		   	\Delta_{V+2\alpha \nabla\ln(v)}(f_s)(p_s)\le 0.
		   \end{equation}
		   \textbf{First case}: $f$ is bounded above on $M$:\\
		   Let us define $f_s = f - \frac{\phi}{s}$, $\phi$ being the function in proposition \ref{exaustion_function_qac_manifold}. Since $f$ is bounded above, $f_s$ attains its supremum at some point $p_s\in M$.  Using inequality $(\ref{prop_p_of_f_equal_0_first_inequality})$ we deduce that:
		   \begin{align*}
		   	\Delta_{V+2\alpha \nabla\ln(v)}(f_s) = \Delta_{V+2\alpha \nabla\ln(v)}(f)-\frac{1}{s}\Delta_{V+2\alpha \nabla\ln(v)}(\phi)\ge &(\lambda - V\ln(v^{\alpha}x_{max}^k))  f\\&-\frac{1}{s}\Delta_{V+2\alpha \nabla\ln(v)}(\phi).
		   \end{align*}
		   Combining this inequality with the fact that 
		   \begin{equation*}
		  (\lambda - V\ln(v^{\alpha}x_{max}^k))  f-\frac{1}{s}\Delta_{V+2\alpha \nabla\ln(x)}(\phi) = (\lambda - V\ln(v^{\alpha}x_{max}^k)) f_s-\frac{1}{s}\mathcal{T}(\phi).
		   \end{equation*}
		   we obtain that:
		   \begin{equation*}
		   	\Delta_{V+2\alpha \nabla\ln(v)}(f_s)\ge(\lambda - V\ln(v^{\alpha}x_{max}^k)) f_s-\frac{1}{s}\mathcal{T}(\phi).
		   \end{equation*}
		   Finally, evaluating this inequality at $p_s$ and using inequality $(\ref{prop_p_of_f_equal_0_second_inequality})$ we obtain that:
		   \begin{equation*}
		  (\lambda - V\ln(v^{\alpha}x_{max}^k)) f_s(p_s)\le \frac{1}{s}\mathcal{T}(\phi)(p_s).
		   \end{equation*}
		   By letting $s\to\infty$ and using the fact that $\mathcal{T}(\phi)$ is bounded above (proposition \ref{exaustion_function_qac_manifold}), we obtain that \\$\sup\limits_{M}f\le 0$.\\
		   \textbf{Second case}: $f$ is unbounded above near the maximal hypersurfaces:\\
		    Let $\theta\in(-2, -1)$ be a constant such that (see lemma \ref{remark_on_existence_of_theta} below for a proof of existence)
		    \begin{equation}
		    	\label{existence_of_theta}
		    	\lambda>\max\left(\sup\limits_{M}V\ln(v^{\alpha}x_{max}^k),\sup\limits_{M}V\ln(v^{\alpha}x_{max}^{k+\theta})\right)
		    \end{equation}
		    and let us set $f_s = f - \frac{x_{max}^{\theta}}{s}$. Since $f$ is unbounded above near maximal hypersurfaces, there exists $s_0\in\mathbb{N}_0$ such that for $s\ge s_0$ there exists a point $p_s\in M$ such that $\sup\limits_{p\in M}f_s(p) = f_s(p_s)$.\\
			Using inequality $(\ref{prop_p_of_f_equal_0_first_inequality})$, we have that:
			\begin{align*}
				\Delta_{V+2\alpha \nabla\ln(v)}(f_s) &= \Delta_{V+2\alpha \nabla\ln(v)}( f) - \frac{1}{s}\Delta_{V+2\alpha \nabla\ln(v)}(x_{max}^{\theta})\\
				&\ge \left(\lambda-V\ln(v^{\alpha}x_{max}^k)\right)f-\frac{1}{s}\Delta_{V+2\alpha \nabla\ln(v)}(x_{max}^{\theta})\\
				&\ge \left(\lambda-V\ln(v^{\alpha}x_{max}^k)\right)f_s- \frac{1}{s}\mathcal{P}(x_{max}^{\theta}).
			\end{align*}
			When evaluating the previous inequality at $p_s$, we have that:
			\begin{equation*}
				\left(\lambda-V\ln(v^{\alpha}x_{max}^k)\right)(f_s)(p_s) \le \frac{1}{s}\mathcal{T}(x_{max}^{\theta})(p_s).
			\end{equation*}
			A simple computation shows that 
			\begin{align*}
				\Delta_{V+2\alpha \nabla\ln(v)}(x_{max}^{\theta}) &= \Delta x_{max}^{\theta} + V(x_{max}^{\theta}) +2\alpha \nabla\ln(v)(x_{max}^{\theta})\\
				&= \Delta x_{max}^{\theta}  + x_{max}^{\theta} V\ln(x_{max}^{\theta}) + 2\alpha\theta v x_{max}^{\theta+1}\left<\frac{dv}{v^2},\frac{dx_{max}}{x_{max}^2}\right> \\
				&\le C + x_{max}^{\theta} V\ln(x_{max}^{\theta})
			\end{align*}
			since both $\Delta x_{max}^{\theta}$ and $v x_{max}^{\theta+1}\left<\frac{dv}{v^2},\frac{dx_{max}}{x_{max}^2}\right> $ are bounded by corollary \ref{corol_laplace_polyhomogene}.\\
			Consequently, using inequality $(\ref{existence_of_theta})$ we deduce that
			\begin{align*}
				\left(\lambda-V\ln(v^{\alpha}x_{max}^k)\right)(f_s)(p_s) &\le \frac{1}{s} \left(x_{max}^{\theta}\left(V\ln(x_{max}^{\theta})+V\ln(v^{\alpha}x_{max}^k) -\lambda\right)+ C \right) \\
				&\le \frac{1}{s} \left(x_{max}^{\theta}\left(V\ln(v^{\alpha}x_{max}^{k+\theta}) -\lambda\right)+ C \right) \\
				&\le \frac{C}{s}.
			\end{align*}
			This implies that when $s\to\infty$ we have $\sup\limits_{p\in M}f(p)\le 0$ which contradicts the hypothesis of the \textbf{second case}. 
	\end{proof}
	
	\begin{lemm}
		\label{remark_on_existence_of_theta}
	    Let $\lambda\in\mathbb{R}$ be such that 
	    \begin{equation}
	    	\lambda>\max\left(\sup\limits_{M}V\ln(v^{\alpha}x_{max}^k),\sup\limits_{M}V\ln(v^{\alpha}x_{max}^{k-1}) \right).
	    \end{equation}
	    Then, there exists $\theta\in (-2, -1)$ such that 
	    \begin{equation}
	    	\label{inequality_lambda_theta}
	    	\lambda>\max\left(\sup\limits_{M}V\ln(v^{\alpha}x_{max}^k),\sup\limits_{M}V\ln(v^{\alpha}x_{max}^{k+\theta}) \right).
	    \end{equation}
	    \begin{proof}
	    	Since $V$ is a b-vector field, we have that $V\ln(x_{max})\in C^{\infty}(X)$, thus is bounded on $M$. \\
	    	Consequently, the difference $V\ln(x_{max}^{k-1})-V\ln(x_{max}^{k+\theta})=V\ln(v_{max}^{-(1+\theta)})$ can be made arbitrary close to zero by choosing the constant $\theta\in (-2, -1)$ close enough to $-1$. This implies that $\sup\limits_{M}V\ln(v^{\alpha}x_{max}^{k+\theta})$ can be made arbitrarily close to $\sup\limits_{M}V\ln(v^{\alpha}x_{max}^{k-1})$ by a choice of a constant $\theta$ as described previously, which then preserves inequality $(\ref{inequality_lambda_theta})$.
	    \end{proof}
	\end{lemm}
	\begin{cor}
		\label{unbounded_tensor_maximum_principal}
		Let $h$ be a tensor such that $h\in\cap_{p\ge1}W^{2,p}_{loc}(M,E)$ and
		\begin{equation*}
			\mathcal{T}(h)= (\Delta_{V+2\alpha \nabla\ln(v)}+V\ln(v^{\alpha}x_{max}^k)-\lambda) (h) =0\text{ and } h=O(x_{max}^{-1}),
		\end{equation*}
		such that $\lambda>\max\left(\sup\limits_{M}V\ln(v^{\alpha}x_{max}^k),\sup\limits_{M}V\ln(v^{\alpha}x_{max}^{k-1}) \right)$.\\
		Then, $h\equiv 0$. \\
		\begin{proof}
			Let us define $f = \vert\vert h\vert\vert^2$. Then we have :
			\begin{align*}
				\mathcal{T}(f) = 2\left<\Delta_{V+2\alpha \nabla\ln(v)} h, h\right> - \left(\lambda - V\ln(v^{\alpha}x_{max}^k) \right)f +2\vert\vert\nabla h\vert\vert^2.
			\end{align*}
			Since $\mathcal{T}(h)=0$ we have that $\Delta_{V+2\alpha \nabla\ln(v)} h = \left(\lambda - V\ln(v^{\alpha}x_{max}^k)\right) h$. \\
			Consequently,
			\begin{equation*}
				\mathcal{T}(f) = \left(\lambda - V\ln(v^{\alpha}x_{max}^k) \right)f + 2\vert\vert\nabla h\vert\vert^2
			\end{equation*}
			which implies that 
			\begin{equation*}
				\mathcal{T}(f)\ge 0.
			\end{equation*}
			By applying proposition \ref{prop_non_bounded_function_maximum_principal} we get that $\sup\limits_{p\in M}f\le 0$. Thus ,$f\equiv 0$.
		\end{proof}
	\end{cor}
\end{prop}

\subsubsection{Proof of theorem \ref{isomorphism_theorem}}
\begin{proof}
	Uniqueness follows from theorem $\ref{isomorphism_theorem_0}$. In order to prove the existence of a solution, we are going to proceed by induction on $k$. The case $k=0$ is exactly theorem $\ref{isomorphism_theorem_0}$. Let $k$ be a positive integer and $H\in C^{k,\theta}_{Qb}(M, E)$. Using the induction hypothesis and the fact that $H\in C^{k-1,\theta}_{Qb}(M, E)$, there exists $h\in D^{2+k-1}_{\mathcal{P}_{\alpha}}(M, E)$ such that $\mathcal{P}_{\alpha}(h)-\lambda h = H$.\\
	Let us define $h_i= x_{max}^{-i}\nabla^i h$ for $i=0,\dots, k$. We want to prove that $h\in D^{2+k}_{\mathcal{P}_{\alpha}}(M, E)$. This amounts to proving that $h_k\in D^{2}_{\mathcal{P}_{\alpha}}(M, E)$. In order to do that, we are going to compute the evolution equation of $h_k$. \\
	Let us recall that $\mathcal{P}_{\alpha}= \Delta + \nabla_{\underbrace{V-2\nabla\ln(v^{-\alpha})}_{V_{\alpha}}}\underbrace{-V\ln(v^{-\alpha})}_{r_{\alpha}}$, so that
	\begin{equation*}
	\mathcal{P}_{\alpha}(h_k) = \Delta_{V_{\alpha}} (h_k) + r_{\alpha} h_k
	\end{equation*}
	In order to compute $\mathcal{P}_{\alpha}(h_k)$ we are going to use lemma \ref{technical_lemma_about_commuting_convariant_derivative} below. Consequently, we have that:
	\begin{multline}
		\label{nabla_hk}
		\nabla_{V_{\alpha}}(h_k) = V_{\alpha}(x_{max}^{-k})\nabla^kh + x_{max}^{-k}\nabla_{V_{\alpha}}\nabla^kh \\
		= (-V\ln(x_{max}^k) + a )h_k + x_{max}^{-k}\nabla^{k}\nabla_{V_{\alpha}}h + x_{max}^{-k}\sum_{j=0}^{k-1}\nabla^{k-j}V_{\alpha}\ast\nabla^{j+1}h+\nabla^{k-1-j}\left(Rm(g_{QAC})\ast V_{\alpha}\right)\ast\nabla^j h\\
		= (-V\ln(x_{max}^k)  + a )h_k  + x_{max}^{-k}\nabla^{k}\nabla_{V_{\alpha}}h + \sum_{j=0}^{k-1}x_{max}^{-k+j+1}\nabla^{k-j} V_{\alpha}\ast h_{j+1}+x_{max}^{-k+j}\nabla^{k-1-j}\left(Rm(g_{QAC})\ast V_{\alpha}\right)\ast h_j
	\end{multline}
	with $a\in x_{max}C^{\infty}(X)$. We also have that:
	\begin{multline}
		\label{delta_hk}
		\Delta(h_k) = \Delta(x_{max}^{-k})\nabla^kh + b x_{max}^{-k}\nabla^{k+1}h +x_{max}^{-k} \nabla^k\Delta h + x_{max}^{-k}\sum_{j=0}^{k}\nabla^{k-j}Rm(g_{QAC})\ast\nabla^j h\\
		= b x_{max}h_{k+1}+ \left(Rm(g_{QAC})+ c \right)\ast h_k +x_{max}^{-k} \nabla^k\Delta h + \sum_{j=0}^{k-1}x_{max}^{-k+j}\nabla^{k-j}Rm(g_{QAC})\ast h_j 
	\end{multline}
	with $b,c\in x_{max}C^{\infty}(X)$. \\
	Now, using equations $(\ref{nabla_hk})$ and $(\ref{delta_hk})$ we deduce that
	\begin{align*}
		(\mathcal{P}_{\alpha} + V\ln(x_{max}^k))(h_k) - \lambda h_k = &b x_{max} h_{k+1} + H_k + \left(a + c +Rm(g_{QAC})\right)\ast h_k \\
		&+\sum_{j=0}^{k-1}x_{max}^{-k+j+1}\nabla^{k-j}V_{\alpha}\ast h_{j+1}+x_{max}^{-k+j}\nabla^{k-1-j}\left(Rm(g_{QAC})\ast V_{\alpha}\right)\ast h_j\\
		&+\sum_{j=0}^{k-1}x_{max}^{-k+j}\nabla^{k-j}Rm(g_{QAC})\ast h_j 
	\end{align*}
	such that $H_k = x_{max}^{-k}\nabla^k H$. Thus
	\begin{equation}
		\label{proof_iso_terms}
		\begin{aligned}
			\vert\vert (\mathcal{P}_{\alpha} + V\ln(x_{max}^k))(h_k) - \lambda h_k\vert\vert_{C^{0,\theta}(M, E)} \le &\vert\vert b x_{max} h_{k+1}\vert\vert_{C^{0,\theta}(M, E)} + \vert\vert H_k\vert\vert_{C^{0,\theta}(M, E)} \\
			&+ \vert\vert \left(a + c +Rm(g_{QAC})\right)\ast h_k\vert\vert_{C^{0,\theta}(M, E)} \\
			+\sum_{j=0}^{k-1}&\vert\vert x_{max}^{-k+j+1}\nabla^{k-j} V_{\alpha}\ast h_{j+1}\vert\vert_{C^{0,\theta}(M, E)}+\\
			&\vert\vert x_{max}^{-k+j}\nabla^{k-1-j}\left(Rm(g_{QAC})\ast V_{\alpha}\right)\ast h_j\vert\vert_{C^{0,\theta}(M, E)}\\
			+\sum_{j=0}^{k-1}&\vert\vert x_{max}^{-k+j}\nabla^{k-j}Rm(g_{QAC})\ast h_j \vert\vert_{C^{0,\theta}(M, E)}.
		\end{aligned}
	\end{equation}
	From the induction hypothesis, there exists a positive constant $C$ such that:
	\begin{align}
		\label{estimation_h_theorem_isomorphism}
		&\vert\vert h\vert\vert_{D^{2+k-1}_{\mathcal{P}_{\alpha}}(M, E)}\le C\vert\vert H\vert\vert_{C^{k-1}_{Qb}(M, E)}\\
		&\vert\vert h\vert\vert_{\mathcal{C}^{k-1;0,\theta}(M, E)}\le C\vert\vert h\vert\vert^{\frac{\theta}{2}}_{D^{2+k-1}_{\mathcal{P}_{\alpha}}(M, E)}\vert\vert h\vert\vert^{1-\frac{\theta}{2}}_{C^{k-1}_{Qb}(M, E)}\text{ for any }\theta\in (0, 2).
	\end{align}
	From inequality \ref{iso_second_estimate_of_h} and the induction hypothesis we have that
	\begin{equation}
		\label{estimation_h_2_theorem_isomorphism_2}
		\vert\vert x_{max}^{-(k-1)}\nabla^{k-1}h\vert\vert_{C^{2,\theta}(M, E)}\le C\vert\vert H\vert\vert_{C^{k-1}_{Qb}(M, E)},
	\end{equation}
	which then implies that 
	\begin{align*}
		&\vert\vert x_{max}^{-(k-1)}\nabla^k h\vert\vert_{C^{0,\theta}(M, E)}\le C\vert\vert H\vert\vert_{C^{k,\theta}_{Qb}(M, E)},\\
		&\vert\vert x_{max}^{-(k-1)}\nabla^{k+1} h\vert\vert_{C^{0,\theta}(M, E)}\le C\vert\vert H\vert\vert_{C^{k,\theta}_{Qb}(M, E)}.
	\end{align*}
	Using the previous inequalities we get that :
	\begin{align*}
		\vert\vert b x_{max} h_{k+1}\vert\vert_{C^{0,\theta}(M, E)}  &\le \tilde{C}\vert\vert \frac{b}{x_{max}}\vert\vert_{C^{1}(M, E)} 	\vert\vert x^{-(k-1)}_{max}\nabla^{k+1}\vert\vert_{C^{0,\theta}(M, E)},  \\
		&\le B \vert\vert H\vert\vert_{C^{k,\theta}_{Qb}(M, E)},
	\end{align*} 
	and 
	\begin{align*}
		\vert\vert \left(a + c +Rm(g_{QAC})\right)\ast h_k\vert\vert_{C^{0,\theta}(M, E)}  &\le \tilde{C}\vert\vert \frac{a + c +Rm(g_{QAC})}{x_{max}}\vert\vert_{C^{1}(M, E)} 	\vert\vert x^{-(k-1)}_{max}\nabla^{k}\vert\vert_{C^{0,\theta}(M, E)},  \\
		&\le B \vert\vert H\vert\vert_{C^{k,\theta}_{Qb}(M, E)}.
	\end{align*} 
	for some positive constant $B$. We proceed in the same manner for the other terms in (\ref{proof_iso_terms}) using the fact that $\vert\vert x_{max}^{-k+j+1}\nabla^{k-j} V_{\alpha}\vert\vert_{C^{1}(M, E)}$, $\vert\vert x_{max}^{-k+j}\nabla^{k-1-j}\left(Rm(g_{QAC})\ast V_{\alpha}\right)\vert\vert_{C^{1}(M, E)}$ and $\vert\vert x_{max}^{-k+j}\nabla^{k-j}Rm(g_{QAC})\vert\vert_{C^{1}(M, E)}$ are bounded (corollary \ref{polyho_lunardi_condition}).\\
	Thus, there exists a positive constant $B$ such that:
	\begin{equation*}
		\vert\vert (\mathcal{P}_{\alpha} + V\ln(x_{max}^k))(h_k) - \lambda h_k\vert\vert_{C^{0,\theta}(M, E)} \le B \vert\vert H\vert\vert_{C^{k,\theta}_{Qb}(M, E)}.
	\end{equation*}
	Therefore, by theorem \ref{theoreme_lunardi_2}, there exists a solution $\tilde{h}_k\in D^{2,\theta}_{\mathcal{P}_{\alpha}}(M, E)$ satisfying
	\begin{align*}
		(\mathcal{P}_{\alpha} + V\ln(x_{max}^k))(\tilde{h}_k) - \lambda \tilde{h}_k = &b x_{max} h_{k+1} + H_k + \left(a + c +Rm(g_{QAC})\right)\ast h_k \\
		&+\sum_{j=0}^{k-1}x_{max}^{-k+j+1}\nabla^{k-j}V_{\alpha}\ast h_{j+1}+x_{max}^{-k+j}\nabla^{k-1-j}\left(Rm(g_{QAC})\ast V_{\alpha}\right)\ast h_j\\
		&+\sum_{j=0}^{k-1}x_{max}^{-k+j}\nabla^{k-j}Rm(g_{QAC})\ast h_j ,
	\end{align*}
	and such that
	\begin{equation*}
		\vert\vert \tilde{h}_k\vert\vert_{C^{2,\theta}(M, E)} \le C \vert\vert H\vert\vert_{C^{k,\theta}_{Qb}(M, E)} \text{ with }\theta\in(0,1).
	\end{equation*}
	It remains to prove that $\tilde{h}_k\equiv h_k$. First of all, as $H\in C^{k,\theta}_{loc}(M, E)$ we have that $h\in C^{k+2,\theta}_{loc}(M, E)$ (by elliptic regularity). As a consequence, the difference $T = \tilde{h}_k- h_k$ satisfies:
	\begin{equation*}
		T\in\cap_{p\ge 1}W^{2,p}_{loc}(M, E)\text{ ; } \left((\mathcal{P}_{\alpha} + V\ln(x_{max}^k)) - \lambda \right)(T) = 0.
	\end{equation*}
	Near the maximal hypersurfaces, we only have that $x_{max}^{-(k-1)}\nabla^kh$ is bounded, so we deduce that $x_{max}T$ is bounded ($T=O(x_{max}^{-1})$).
	By corollary \ref{unbounded_tensor_maximum_principal} we have that $T\equiv 0$.\\
	Consequently, $h_k\in D^{2,\theta}_{\mathcal{P}_{\alpha}}(M, E)$ and
	\begin{equation*}
		\vert\vert h_k\vert\vert_{C^{2,\theta}(M,E)}\le C\vert\vert H\vert\vert_{C^{k,\theta}_{Qb}(M, E)}  \text{ with }\theta\in(0,1).
	\end{equation*}
	We also have that 
	\begin{equation*}
		\vert\vert h_k\vert\vert_{C^{\theta}(M,E)}\le C\vert\vert h\vert\vert^{\frac{\theta}{2}}_{ D^{2}_{\mathcal{P}_{\alpha}}(M, E)} \vert\vert h\vert\vert^{1-\frac{\theta}{2}}_{C^{0}(M, E)}  \text{ with }\theta\in(0,2).
	\end{equation*}
\end{proof}

As a consequence of theorem \ref{isomorphism_theorem} together with remark \ref{remark_original_operator_weighted_operator}, we have the following result:
\begin{cor}
	\label{cor_isomorphism}
	Suppose that:
	\begin{equation}
			\vert\vert Rm(g)\ast V\vert\vert_{C^0(M,E)}+\vert\vert \nabla V\vert\vert_{C^{0}(M,E)}< \infty
	\end{equation}
	Then, the operator $\mathcal{A}:D^{2+k,\theta}_{\Delta_V,\alpha}(M, E)\to C^{k,\theta}_{Qb,\alpha}(M, E)$ is an isomorphism of Banach spaces, for any $\theta\in (0,1)$ and any constant $\lambda\in\mathbb{R}$ such that:
	\begin{equation*}
		\lambda>\max\left(\sup\limits_{M}V\ln(x^{\alpha}x_{max}^k),\sup\limits_{M}V\ln(x^{\alpha}x_{max}^{k-1})\right)
	\end{equation*}
\end{cor}

\begin{cor}
	The spaces $D^{2+k,\theta}_{\Delta_V}(M, E)$ and $D^{2+k,\theta}_{\Delta_V,\alpha}(M, E)$ are Banach spaces.\\
	\begin{proof}
		We use proposition \ref{proposition_D_banach_space} and the previous corollary.
		\end{proof}
\end{cor}
	\section[Lunardi theorem]{Lunardi's theorem}
\label{section_lunadi_theorem}

In this section we study a class of linear elliptic operators of the form $\Delta +\nabla_V + r$ with unbounded coefficients. Such operators where studied by Alessandra Lunardi in \cite{lunardi_estimate} on $\mathbb{R}^n$ then a version was proven by \cite{smoothing_deruelle} in the context of Riemannian manifold and such that $r\equiv0$. 
	
\begin{theo}[Lunardi]
	\label{theoreme_lunardi}
	Let $\left(M^n, g\right)$ be a complete Riemannian manifold with positive injectivity radius, and $V$ be a smooth vector field on $M$. Let $\mathcal{A}$ be an elliptic differential operator acting on tensors over $M$ such that:
	\begin{equation}
		\label{lunardi_operator}
		\mathcal{A} = \underbrace{\Delta +\nabla_V}_{\Delta_V}+ r(x)\:,\:r\in C^3(M).
	\end{equation}
  Suppose that $sup_{x\in M} r(x)=r_0<\infty$ and that there exists a positive constant $C$ such that $\sum_{i=1}^{3}\vert\vert\nabla^i r\vert\vert<C$. Assume also that there exists a positive constant $K$ such:\\
	\begin{equation}
		\vert\vert Rm(g)\vert\vert_{C^3(M,E)}+\vert\vert Rm(g)\ast V\vert\vert_{C^3(M,E)}+\vert\vert \nabla V\vert\vert_{C^{2}(M,E)}\le K,
	\end{equation}
where $Rm(g)\ast V=Rm(g)(V,.,.,.)$. Assume also that there exists a function $\phi\in C^2(M)$ and a constant $\lambda_0\ge r_0$ such that:
\begin{equation}
	\label{lunardi_condition_phi}
	\lim_{x\to\infty}\phi (x) = +\infty, \: sup_{x\in M}\left(\mathcal{A}(\phi)(x)-\lambda_0\phi\left(x\right)\right)<\infty.
\end{equation}
Then:
\begin{enumerate}
	\item For any  $\lambda>r_0$, there exists a positive constant $C$ such that for any $H\in C^0\left(M, E\right)$, there exists a unique tensor $h\in D^2_{\mathcal{A}}\left(M, E\right)$, satisfying:
	\begin{equation*}
		\mathcal{A}(h) -\lambda h= H,\: \vert\vert h\vert\vert_{D^2_{\mathcal{A}}\left(M,E\right)}\le C \vert\vert H\vert\vert_{C^0\left(M, E\right)}.
	\end{equation*}
	Moreover $D^2_{\mathcal{A}}\left(M, E\right)$ is continuously embedded in $C^{\theta}\left(M,E\right)$ for any $\theta\in\left(0,2\right)$, i.e. there exists a positive constant $C(\theta)$ such that for any $h\in D^2_{\mathcal{A}}\left(M, E\right)$, 
	\begin{equation*}
		\vert\vert h\vert\vert_{C^{\theta}(M, E)}\le C(\theta)\vert\vert h\vert\vert^{\frac{\theta}{2}}_{D^2_{\mathcal{A}}\left(M, E\right)} \vert\vert h\vert\vert^{1-\frac{\theta}{2}}_{C^0(M, E)}.
	\end{equation*}
	\item For any $\lambda>r_0$, there exists a positive constant $C$ such that for any $H\in C^{0,\theta}\left(M, E\right)$, $\theta\in \left(0,1\right)$ , there exists a unique tensor $h\in C^{2,\theta}\left(M, E\right)$ satisfying:
	\begin{equation*}
		\mathcal{A}(h) - \lambda h= H,\: \vert\vert h\vert\vert_{C^{2,\theta}\left(M, E\right)}\le C \vert\vert H\vert\vert_{C^{0,\theta}\left(M, E\right)}.
	\end{equation*}
\end{enumerate}
\end{theo}

\subsection{Proof}

\subsubsection{Uniqueness of the solution}

	\begin{prop}[Injectivity]
		\label{lunardi_injectivite}
		Let $h\in \cap_{p\ge 1}W^{2,p}_{loc}\left(M, E\right)$ be a bounded tensor and $\lambda>r_0$ a constant such that $\mathcal{A}(h) - \lambda h=0$. Then $h\equiv 0$. \\
		\begin{proof}
			Let us define $h_\epsilon = \sqrt{\vert\vert h\vert\vert^2+\epsilon^2}$ (for some positive constant $\epsilon$). Then:
			\begin{equation*}
				\mathcal{A}(h_{\epsilon}) -\lambda h_{\epsilon}= \frac{1}{h_{\epsilon}}\left(\left<\Delta_V h, h\right> - \left(\lambda - r\right)h_{\epsilon}^2 +\vert\vert\nabla h\vert\vert^2 -\frac{\vert\vert\nabla\vert\vert h\vert\vert^2\vert\vert^2}{4 h_{\epsilon}^2} \right).
			\end{equation*}
		Since $\Delta_V h=\left(\lambda - r\right)h$ and $\vert\vert\nabla\vert\vert h\vert\vert^2\vert\vert^2=4\vert\left<\nabla h, h\right>\vert^2\le 4\vert\vert\nabla h\vert\vert^2\;\vert\vert h\vert\vert^2$, we have that:
		\begin{equation*}
		\mathcal{A}(h_{\epsilon})  -\lambda h_{\epsilon}= \frac{1}{h_{\epsilon}}\left(-\left(\lambda - r\right)\epsilon^2+\vert\vert\nabla h\vert\vert^2-\frac{\vert\left<\nabla h, h\right>\vert^2}{h_{\epsilon}^2}\right)\ge \frac{\epsilon^2}{h_{\epsilon}}\left(-\left(\lambda-r\right)+\frac{\vert\vert\nabla h\vert\vert^2}{h_{\epsilon}^2}\right)\ge -\left(\lambda-r\right)\epsilon.
		\end{equation*}
		Let us define $h_{\epsilon, k} = h_{\epsilon}-\frac{\phi}{k}$ for an integer $k\ge 1$. Then, $\displaystyle\lim_{k\to\infty}sup_M h_{\epsilon, k}=sup_M h_{\epsilon}$, and we also have:
		\begin{equation*}
			\mathcal{A}(h_{\epsilon, k}) -\lambda h_{\epsilon, k}\ge -(\lambda - r)\epsilon - \frac{sup_M\left(\mathcal{A}(\phi)-\lambda \phi\right)}{k}.
		\end{equation*}
		Since $\phi$ is an exhaustion function (that can be chosen to be positive), $h_{\epsilon, k}$ attains its maximum in a point $x_k\in M$. When evaluating previous inequality at $x_k$, we get that:
	\begin{equation*}
	sup_M h_{\epsilon, k}\le \epsilon + \frac{sup_M\left(\mathcal{A}(\phi)-\lambda \phi\right)}{k\left(\lambda - r _0\right) }.
	\end{equation*}
By letting $k\to\infty$ and $\epsilon\to 0$ we get $ sup_M h_{\epsilon}\le 0$, and consequently $h\equiv 0$. 
		\end{proof}
	\end{prop}
	
	\subsubsection{Existence of the solution}
In order to study the existence and the regularity of the solution of the equation
\begin{equation}
	\label{equation_elleptique_lunardi}
	\mathcal{A}(h) -\lambda h= H,
\end{equation} 
we need to study the semigroup $T(t)$ associated to the following Cauchy problem:
\begin{equation}
	\label{probleme_de_cauchy_lunardi}
	\begin{cases}
		u_t(t,x) = \mathcal{A}(u)(t,x),\\
		u(0, x) = u_0(x).
	\end{cases}
\end{equation}
Using the interpolation procedure in \cite{lunardi_interpolation_method} we will be able to characterize the domain of the generator of $T(t)$ and provide an optimal description of the regularity of the solution of $(\ref{equation_elleptique_lunardi})$.\\
As a first step, we are going to find an estimate of $\vert\vert T(t)\vert\vert_{\mathcal{L}\left(C^{\alpha}\left(M,E\right), C^{\theta}\left(M,E\right)\right)}$ such that $0\le \alpha\le \theta\le 3$ . In order to do that, we will be using the following version of the maximum principle:
\begin{prop}
	\label{principe_du_maximum}
	Let $\left(z(t,)\right)_{t\in[0, T]}$ be a classic bounded solution of the Cauchy problem
	\begin{equation}
		\label{probleme_de_cauchy_lunardi_maximum}
		\begin{cases}
			z_t(t,x) - \mathcal{A}(z)(t,x)=g(t,x),\\
			z(0, x) = z_0(x),
		\end{cases}
	\end{equation}
and $\lambda_0\ge r_0$. Then
\begin{enumerate}
	\item If $\sup\limits_{M}\: z>0$, and if $g(t,x)\le 0$ for all $t\in\left[0,T\right]$ and $x\in M$, then
	\begin{equation}
		\label{maximum_principal_first_case}
		\sup\limits_{M} z\le e^{\lambda_0 t}\sup\limits_{M} \:z_0.
	\end{equation}
\item If $\inf\limits_{M}\: z<0$, and if $g(t,x)\ge 0$ for all $t\in\left[0,T\right]$ and $x\in M$, then
\begin{equation}
	\label{maximum_principal_second_case}
	\inf\limits_{M}z\ge e^{\lambda_0 t}\:\inf\limits_{M} \:z_0.
\end{equation}
\item In particular, if $g\equiv 0$, then
\begin{equation}
	\vert\vert z\vert\vert_{\infty}\le e^{\lambda_0 t}\:\vert\vert z_0\vert\vert_{\infty}.
\end{equation}
\end{enumerate}
\begin{proof}
	In order to prove inequality $\ref{maximum_principal_first_case}$, we define $v(t,x)=e^{-\lambda t}\:z(t,x)$ for $\lambda> \lambda_0$. Then, 
	\begin{equation*}
		\begin{cases}
			v_t(t,x) - \mathcal{A}(v)(t,x)+\lambda v=e^{-\lambda t} g(t,x)\\
			v(0, x) = z_0(x)
		\end{cases}
	\end{equation*}
We also define $v_k(t,x) = v(t,x)-\frac{\phi(x)}{k}$.  For $k$ large enough, $v_k$ admits a positive maximum (since $sup_M z>0$) at $(t_k, x_k)$. If $t_k\equiv 0$ for all $k$, then
\begin{equation*}
	sup_{\left[0,T\right]\times M}\:v_k\le sup_M z_0 - inf_M\:\frac{\phi}{k}.
\end{equation*}
Consequently
\begin{equation*}
	sup_{\left[0,T\right]\times M}\:e^{-\lambda t}\:z\le sup_M z_0,
\end{equation*}
hence inequality $\ref{maximum_principal_first_case}$.  Now, suppose that $t_k>0$. Since $\partial_t v(t_k, x_k)\ge 0$ (because $\partial_t v(t_k, x_k)=\partial_t v_k(t_k, x_k)\ge 0$), we have that:
\begin{equation*}
	\mathcal{A}(v)(t_k, x_k) - \lambda v(t_k, x_k)\ge 0.
\end{equation*}
Adding $-\frac{\left(\mathcal{A}(\phi)-\lambda \phi\right)(t_k, x_k)}{k}$ to both sides of the previous inequality, we find that:
\begin{equation*}
	\left(\lambda-r_0\right)v_k(t_k, x_k)\le \frac{\left(\mathcal{A}(\phi)-\lambda \phi\right)(t_k, x_k)}{k},
\end{equation*}
which is impossible for $k$ large enough.\\
Inequality $\ref{maximum_principal_second_case}$ can be proved by replacing $z$ with $-z$ in inequality $\ref{maximum_principal_first_case}$ and using the fact that  $-sup_M \left(-z_0\right)\ge inf_M\:z_0$. The last inequality is a combination of the previous ones.
\end{proof}
\end{prop}

Before we proceed with the next theorem, we will need the following technical lemma.
\begin{lemm}
	\label{technical_lemma_about_commuting_convariant_derivative}
	Let $\left(M, g\right)$ be a Riemannian manifold and $\nabla$ and $\Delta$ the Levi-Civita connection and the Laplacian respectively associated to $g$. Then, we have that:
	\begin{equation}
		\label{commute_derivation_and_laplacian_with_covariant}
		\begin{aligned}
			&\left[\nabla^k, \Delta\right]=\sum_{j=0}^{k}\nabla^{k-j}Rm(g)\ast\nabla^j,\\
			&\left[\nabla^k, \nabla_V\right]=\sum_{j=0}^{k-1}\nabla^{k-j}V\ast\nabla^{j+1}+\nabla^{k-1-j}(Rm(g)\ast V)\ast\nabla^j,
		\end{aligned}
	\end{equation}
	where $Rm(g)$ is the Riemannian curvature tensor, and $Rm(g)\ast V=Rm(g)(V,.,.,.)$.\\
	\begin{proof}
		We will proceed by induction on $k$ to prove equality $(\ref{commute_derivation_and_laplacian_with_covariant})$. Let us prove the result for $k=1$ using normal coordinates. By definition of the curvature tensor , we have that
		\begin{align*}
			\nabla_i\nabla_V h = \nabla_V\nabla_i h + \left( Rm(g)\ast V\right)\ast h + \nabla_{\nabla V}h =  \nabla_V\nabla_i h + \left( Rm(g)\ast V\right)\ast h + \nabla V\ast\nabla h, 
		\end{align*}
		where $\left( Rm(g)\ast V\right)\ast h$ corresponds to the action of the curvature tensor on tensors over $M$.
		This proves the second relation for $k=1$. Regarding the first the equality, a simple computation shows that: 
		\begin{equation*}
			\nabla\nabla^2_{i,j} h = \nabla^2_{i,j} \nabla h + \nabla_i\left(Rm(g)\ast h\right) + Rm(g)\ast\nabla_j h,
		\end{equation*}
		which proves the first relation of $(\ref{commute_derivation_and_laplacian_with_covariant})$ for $k=1$.\\
		Now suppose that the result is true for $k\ge 1$. Then we have that:
		\begin{align*}
			\nabla^{k+1}\nabla_V h &= \nabla\left(\nabla_V\nabla^k h+\sum_{j=0}^{k-1}\nabla^{k-j}V\ast\nabla^{j+1}h+\nabla^{k-1-j}(Rm(g)\ast V)\ast\nabla^jh\right)\\
			&= \nabla \nabla_V\nabla^kh + \sum_{j=0}^{k-1}\nabla^{(k+1)-j}V\ast\nabla^{j+1}h+\nabla^{(k+1)-(j+1)}V\ast\nabla^{(j+1)+1}h\\
			&+\nabla^{(k+1)-(j+1)}(Rm(g)\ast V)\ast\nabla^jh+\nabla^{(k+1)-(j+2)}(Rm(g)\ast V)\ast\nabla^{(j+1)}h \\
			&= \nabla_V\nabla^{k+1} h + \left( Rm(g)\ast V\right)\ast \nabla^k h + \nabla V\ast\nabla^{k+1} h \\
			&+ \sum_{j=0}^{k}\nabla^{(k+1)-j}V\ast\nabla^{j+1}h+\nabla^{(k+1)-(j+1)}(Rm(g)\ast V)\ast\nabla^jh \\
			&= \sum_{j=0}^{k}\nabla^{(k+1)-j}V\ast\nabla^{j+1}h+\nabla^{(k+1)-(j+1)}(Rm(g)\ast V)\ast\nabla^jh
		\end{align*}
		This proves the second part of equality $(\ref{commute_derivation_and_laplacian_with_covariant})$. We proceed in the same manner to prove the first one.
	\end{proof}
\end{lemm}
\begin{theo}
	\label{theoreme_estimation_lunardi}
	Let $\left(u(t,.)\right)_{t\in[0, T)}$ be a bounded solution of the Cauchy problem $(\ref{probleme_de_cauchy_lunardi})$ with initial condition $u_0\in C^{\infty}(M,E)$. Assume that there exists an integer $1\le k\le 3$ and a positive constant $K(k)$ such that:
	\begin{equation}
		\label{condition_rm_lunardi}
		\vert\vert Rm(g)\vert\vert_{C^k(M,E)}+\vert\vert Rm(g)\ast V\vert\vert_{C^k(M,E)}+\vert\vert \nabla V\vert\vert_{C^{k-1}(M,E)}\le K(k),
	\end{equation}
where $Rm(g)\ast V=Rm(g)(V,.,.,.)$.\\
Then, for any $T>0$, there exists a constant $\omega=\omega(n,k,\lambda_0)$ such that :
\begin{equation}
	\label{estimation_cov_lunardi}
	\vert\vert u(t)\vert\vert^2_{C^0(M,E)}+\sum_{i=1}^{k}\frac{\left(\alpha t\right)^i}{i}\vert\vert\nabla^i  u(t)\vert\vert^2_{C^0(M, E)}\le e^{\omega t}\vert\vert u_0\vert\vert^2_{C^0(M,E)}\:\: \forall t\in\left[0, T\right],
\end{equation}
$\alpha$ being a positive constant that will be defined later in order to obtain the right estimates.\\
\begin{proof}
	We are going to derive the evolution of the heat equation of the following function:
	\begin{equation*}
		s(t,x) = \vert\vert u\vert\vert^2(t,x) + \sum_{i=1}^{i=k}\frac{\left(\alpha t\right)^i}{i}\vert\vert\nabla^i  u\vert\vert^2(t,x),
	\end{equation*}
	and then apply the maximum principle for some values of $\alpha$. We compute that\\
	\begin{equation*}
		s_t = 2\left<u_t, u\right> + \alpha \sum_{i=1}^{k}\left(\alpha t\right)^{i-1}\vert\vert\nabla^i  u\vert\vert^2 + 2\sum_{i=1}^{k}\frac{\left(\alpha t\right)^{i-1}}{i}\left<\nabla^i u_t, \nabla^i u\right>. 
	\end{equation*}
 Since $u$ is a solution of $(\ref{probleme_de_cauchy_lunardi})$, we obtain:
 \begin{equation*}
 	s_t = 2\left(\left<\Delta_V u, u\right> +r\vert\vert u\vert\vert^2 \right)+ \alpha \sum_{i=1}^{k}\left(\alpha t\right)^{i-1}\vert\vert\nabla^i  u\vert\vert^2 + 2\sum_{i=1}^{k}\frac{\left(\alpha t\right)^{i}}{i}\left<\nabla^i \Delta_Vu+\nabla^iru, \nabla^i u\right>. 
 \end{equation*}
On the other hand, we have that:
\begin{equation*}
	\mathcal{A}(s) = 2\left(\vert\vert\nabla u\vert\vert^2 +\left<\Delta_V u, u\right>+\sum_{i=1}^{k}\frac{\left(\alpha t\right)^{i}}{i}\left(\vert\vert\nabla^{i+1}u\vert\vert^2+\left<\Delta_V\nabla^iu,\nabla^i u\right>\right) \right)+r s.
\end{equation*}
Thus,
\begin{multline*}
	s_t-\mathcal{A}(s) = 2\sum_{i=1}^{k}\frac{\left(\alpha t\right)^i}{i}\left<\left[\nabla^i,\Delta_V\right]u, \nabla^i u\right>+\left(\alpha-2\right)\vert\vert\nabla u\vert\vert^2 - \frac{2\left(\alpha t\right)^k}{k}\vert\vert\nabla^{k+1}u\vert\vert^2+\\ \sum_{i=1}^{k-1}\left(\alpha t\right)^i\left(\alpha-\frac{2}{i}\right)\vert\vert\nabla^{i+1}u\vert\vert^2-r\left(s-2\vert\vert u\vert\vert^2\right)+2\sum_{i=1}^{k}\frac{\left(\alpha t\right)^i}{i}\left<\nabla^ir u, \nabla^i u\right>.
\end{multline*}
If we choose $\alpha\le\frac{2}{k-1}$ if $k\ge 2$, and $\alpha\le 2$ if $k=1 $, we obtain the following inequality:
\begin{multline*}
s_t-\mathcal{A}(s) \le 2\sum_{i=1}^{k}\frac{\left(\alpha t\right)^i}{i}\left<\left[\nabla^i,\Delta_V\right]u, \nabla^i u\right>-r\left(s-\vert\vert u\vert\vert^2-2\sum_{i=1}^{k}\frac{\left(\alpha t\right)^i}{i}\vert\vert\nabla^i u\vert\vert^2\right)+\\ 2\sum_{i=1}^{k}\frac{\left(\alpha t\right)^i}{i}\sum_{j=1}^{i}\frac{i!}{\left(i-j\right)! j!}\left<\nabla^j r\nabla^{i-j} u, \nabla^i u\right>.
\end{multline*}
Now, using the fact that:
\begin{equation*}
	-r\left(s-\vert\vert u\vert\vert^2-2\sum_{i=1}^{k}\frac{\left(\alpha t\right)^i}{i}\vert\vert\nabla^i u\vert\vert^2\right)=rs\le r_0 s,
\end{equation*}
and that there exists a positive constant $C_1$ (depending on $T$ and using the fact that $\alpha\le 2$) such that:
\begin{equation*}
	2\sum_{i=1}^{k}\frac{\left(\alpha t\right)^i}{i}\sum_{j=1}^{i}\frac{i!}{\left(i-j\right)! j!}\left<\nabla^j r\nabla^{i-j} u, \nabla^i u\right>\le C_1 \sum_{i=1}^{k}\sum_{j=1}^{i}\vert\vert\nabla^{i-j} u\vert\vert\; \vert\vert \nabla^i u\vert\vert\le kC_1s,
\end{equation*}
we obtain that:
\begin{equation*}
	s_t-\mathcal{A}(s) \le 2\sum_{i=1}^{k}\frac{\left(\alpha t\right)^i}{i}\left<\left[\nabla^i,\Delta_V\right]u, \nabla^i u\right>+(r_0+kC_1)s.
\end{equation*}
On the other hand, by lemma \ref{technical_lemma_about_commuting_convariant_derivative} we have:
\begin{equation*}
	\left[\nabla^i, \Delta\right]h=\sum_{j=0}^{i}\nabla^jh\ast\nabla^{i-j}Rm(g),
\end{equation*} 
 and
 \begin{equation*}
 	\left[\nabla^i, \nabla_V\right]h=\sum_{j=0}^{i-1}\nabla^{i-j}V\ast\nabla^{j+1}h+\nabla^{i-1-j}(Rm(g)\ast V)\ast\nabla^jh.
 \end{equation*}
Using condition $\ref{condition_rm_lunardi}$ we deduce that there exists a positive constant $C_2$ such that:
\begin{equation*}
	\vert\vert \left[\nabla^i, \Delta_V\right]h \vert\vert \le C_2\sum_{j=0}^{i}\vert\vert\nabla^j h\vert\vert\text{ for }i\le k.
\end{equation*}
Consequently, there exists a positive constant $\tilde{C}$ such that:
	\begin{equation*}
		\sum_{i=1}^{k}\frac{\left(\alpha t\right)^i}{i}\left<\left[\nabla^i,\Delta_V\right]u, \nabla^i u\right>\le C_2 \sum_{i=1}^{k}\frac{\left(\alpha t\right)^i}{i}\sum_{j=0}^{i}\vert\vert\nabla^j h\vert\vert \;\vert\vert\nabla^i h\vert\vert\le \tilde{C}s
	\end{equation*}
	Finally, using the fact that $s(0,t)=\vert\vert u_0\vert\vert^2$ and by applying proposition \ref{principe_du_maximum} to the operator $\tilde{\mathcal{A}}=\mathcal{A}+\left(2\tilde{C}+r_0kC_1\right)I$ ($u_t-\tilde{\mathcal{A}}(u)\le 0$), we deduce that:
	\begin{equation}
		s(t, x) \le e^{\omega t}\vert\vert u_0\vert\vert^2_{C^0(M,E)}
	\end{equation}
	such that $\omega=r_0+2\tilde{C}+r_0kC_1$.
\end{proof}
\end{theo}
\begin{remq}
	The initial condition of the Cauchy problem $(\ref{probleme_de_cauchy_lunardi})$  doesn't need to be smooth. In fact, using estimates $(\ref{estimation_cov_lunardi})$ and the fact that $C^{\infty}(M, E)$ is dense in $C^0(M, E)$ (in the strong topology), we can see that $u_0$ can be chosen in $C^0(M, E)$.
\end{remq}
\begin{remq}
	The Cauchy problem $\ref{probleme_de_cauchy_lunardi}$ defines a semigroup of linear operators $T(t)$ that acts on $C^0(M, E)$, such that
	\begin{equation*}
		\left(T(t)u_0\right)(x) = u(t,x) \;\;\;t\ge 0,\;\;x\in M,\;\; u_0\in C^0(M, E).
	\end{equation*} 
	The estimates of theorem $\ref{theoreme_estimation_lunardi}$ implies that for all integers $0\le s\le k\le 3$ and $t\in\left(0,1\right]$, we have that:
	\begin{equation}
		\label{estimation_operateur_lineaire_T_0_k}
		\vert\vert T(t)\vert\vert_{\mathcal{L}\left(C^s\left(M,E\right), C^k\left(M,E\right)\right)}\le \frac{C e^{\omega t}}{t^{\frac{k-s}{2}}},
	\end{equation}
	such that $C$ is a positive constant independent of $t$.  Moreover, using the maximum principle we get the following estimate:
	\begin{equation}
		\label{estimation_c0_operateur_T}
		\vert\vert T(t)u_0\vert\vert_{C^0\left(M, E\right)}\le e^{\lambda_0 t}\:\vert\vert u_0\vert\vert_{C^0\left(M,E\right)}.
	\end{equation}
\end{remq}
In order to prove the existence of a solution to the Cauchy problem $(\ref{probleme_de_cauchy_lunardi})$, we are going to approximate the problem using a sequence of elliptic operators with bounded coefficients. These operators have a unique solution on $M\times\left(0,\infty\right)$ (using the theory of parabolic equations with bounded coefficients).
\begin{theo}[Lunardi]
For any $u_0\in C^0(M, E)$, there exists a unique (bounded)  solution $\left(u(t)\right)_{t\in[0,\infty)}$ of the Cauchy problem $(\ref{probleme_de_cauchy_lunardi})$.\\
\begin{proof}
	Unicity is a direct consequence of the maximum principle (proposition $\ref{principe_du_maximum}$).
	In order to prove the existence we will proceed as follows. Let $F\in C^{\infty}(M)$ such that:
	\begin{equation}
		\label{properties_of_F}
		lim_{x\to\infty}F(x)=\infty,\; F(x)\le c(1+d_p(x)),\;\; \forall x\in M,\;\; \vert\vert \nabla F\vert\vert+\vert\vert\nabla^2F\vert\vert\le c,
	\end{equation}
where $d_p$ is the distance function with respect to a fixed point $p\in M$. The existence of such a function is proved in theorem 3.6 in \cite{shi97} and uses the fact that the curvature is bounded.\\
Let $\psi:\mathbb{R}_+\to\mathbb{R}_+$ be a function such that $\psi(t)=1$ if $0\le t\le 1$ and $\psi(t)=0$ if $t\ge 2$. We define $\psi_s(x)= \psi(\frac{F(x)}{s})$, $V_s=\psi_s V$ and $r_s=\psi_s r$.\\ Note that $V_s$ is bounded on $M$ and that condition \ref{condition_rm_lunardi} continues to be satisfied by $V_s$. Consequently, the following Cauchy problem
\begin{equation*}
	\begin{cases}
		\partial_t u_s(t,x) = \mathcal{A}_s(u_s)(t,x)\\
		u_s(0, x) = u_0(x)
	\end{cases}
\end{equation*}
has a unique (bounded) solution $\left(u_s(t)\right)_{t\in(0,\infty)}$, where $\mathcal{A}_s= \Delta_{V_s}+r_s$.\\
Using equation $(\ref{properties_of_F})$ and the fact that $\psi$ is a compactly supported function, there exists a positive constant $C$ independent of $s$ such that $\vert\vert\psi_s\vert\vert_{C^2(M,E)}\le C$. Thus, inequality $(\ref{condition_rm_lunardi})$ is satisfied for the vector fields $V_s$ with a constant $K(k)$ independent of $s$. The same thing applies to inequality $(\ref{lunardi_condition_phi})$ with the operator $\mathcal{A}_s$.\\
Consequently, the estimates of theorem $\ref{theoreme_estimation_lunardi}$ applies to $u_s$, such that $\omega$ is a positive constant independent of $s$. In particular, $\vert\vert u_s\vert\vert_{C^k\left(M,E\right)}$ is uniformly bounded. \\
Using Arzela-Ascoli, there exists a subsequence $\left(u_{k_i}\right)_{k_i\in\mathbb{N}}$ and a tensor $u\in C^{\infty}(M,E)$ such that $u_{k_i}$ converges uniformly to $u$ on any compact subset of $(0,\infty)\times M$.
\end{proof}

\subsubsection{Interpolation spaces}
	Before we proceed with that last step of the proof, let us recall some results on interpolation spaces (see \cite{lunardi_interpolation} for more details).
	\begin{mydef}
		Let $X$ and $Y$ be two Banach spaces such that $Y$ is continuously embedded into $X$.  An intermediate space between $X$ and $Y$ is a Banach space $E$ such that $Y\subset E\subset X$ with continuous inclusions.
	\end{mydef}
	\begin{mydef}
	The interpolation space $\left(X, Y\right)_{\theta,\infty}$ is an intermediate space between $X$ and $Y$ defined by:
	\begin{equation}
		\label{def_esapce_interpolation}
		\left(X, Y\right)_{\theta,\infty}=\left\{x\in X\:\vert\: \vert\vert x\vert\vert_{\theta,\infty}=sup_{t\in\left(0,1\right)} t^{-\theta}K\left(t,x,X,Y\right)<\infty    \right\},
	\end{equation}
such that:
\begin{equation}
	\label{def_interpolation_K}
	K(t,x, X, Y) = inf\left\{\vert\vert a\vert\vert_X +t\vert\vert b\vert\vert_Y\:\vert\: x=a+b,\:\left(a,b\right)\in X\times Y\right\}.
\end{equation}
We can see that $\left(\left(X, Y\right)_{\theta,\infty}, \vert\vert\:.\:\vert\vert_{\theta,\infty}\right)$ is a Banach space.
	\end{mydef}
\begin{theo}[Interpolation theorem]
	\label{theoreme_dinterpolation}
	Let $Y_1$ and $Y_2$ be two Banach space continuously embedded into respectively the Banach spaces $X_1$ and $X_2$. If $T\in\mathcal{L}(X_1, X_2)\cap\mathcal{L}(Y_1, Y_2)$, then $T\in\mathcal{L}\left(\left(X_1, Y_1\right)_{\theta,\infty}, \left(X_2, Y_2\right)_{\theta,\infty}\right)$ for any $\theta\in\left(0,1\right)$. Moreover
	\begin{equation}
		\label{interpolation_norm}
		\vert\vert T\vert\vert_{\mathcal{L}\left(\left(X_1, Y_1\right)_{\theta,\infty}, \left(X_2, Y_2\right)_{\theta,\infty}\right)}\le\left(\vert\vert T\vert\vert_{\mathcal{L}(X_1, X_2)}\right)^{1-\theta}\left(\vert\vert T\vert\vert_{\mathcal{L}(Y_1, Y_2)}\right)^{\theta}.
	\end{equation}
\end{theo}
\begin{mydef}
	Let $\theta\in\left[0,1\right]$ and $E$ be an intermediate space between $X$ and $Y$. Then, we say that :
	\begin{itemize}
		\item[$\left(i\right)$] $E$ belongs to the class $J_{\theta}(X, Y)$ if there exists a constant $c>0$ such that:
		\begin{equation}
			\vert\vert x\vert\vert_E\le c\vert\vert x\vert\vert_X^{1-\theta}\vert\vert x\vert\vert_Y^{\theta}\:,\forall x\in Y.
		\end{equation}
	\item[$\left(ii\right)$] $E$ belongs to the class $K_{\theta}(X, Y)$ if there exists a constant $c>0$ such that:
	\begin{equation}
		K\left(t,x, X, Y\right)\le ct^{\theta}\vert\vert x\vert\vert_E,\:\forall x\in E,\:\forall t>0.
	\end{equation}
	\end{itemize}
	The last inequality implies that a Banach space $E$ is of class $K_{\theta}(X, Y)$ if and only if $E$ is embedded continuously into $\left(X, Y\right)_{\theta,\infty}$. We also have that $\left(X, Y\right)_{\theta,\infty}\in J_{\theta}\left(X,Y\right)\cap K_{\theta}\left(X, Y\right)$.

	\begin{theo}[Reiteration theorem]
		Let $0\le \theta_0<\theta_1\le 1$ and $\theta\in\left(0,1\right)$. If $\omega=\left(1-\theta\right)\theta_0+\theta\theta_1$, then:
	\begin{itemize}
		\item[$\left(i\right)$] If $E_i$ is of class $K_{\theta_i}(X, Y)$ $\left(i=0,1\right)$, then
		\begin{equation}
			\left(E_0, E_1\right)_{\theta,\infty}\subset\left(X, Y\right)_{\omega,\infty}.
		\end{equation}
	\item[$\left(ii\right)$] If $E_i$ is of class $J_{\theta_i}(X, Y)$ $\left(i=0,1\right)$, then
	\begin{equation}
		\left(X, Y\right)_{\omega,\infty}\subset\left(E_0, E_1\right)_{\theta,\infty}.
	\end{equation}
	\end{itemize}
	Consequently, if $E_i\in K_{\theta_i}(X, Y)\cap J_{\theta_i}(X, Y)$, then 
	\begin{equation}
		\left(E_0, E_1\right)_{\left(\theta,\infty\right)}=\left(X, Y\right)_{\left(\omega,\infty\right)}.
	\end{equation}
	\end{theo}
\end{mydef}
\end{theo}
\begin{prop}[proposition 2.8 \cite{smoothing_deruelle}]
	\label{deruelle_espace_interpolation}
	Let $\left(M, g\right)$ be a complete Riemannian manifold with positive injectivity radius and bounded curvature together with its covariant derivatives. Then
	\begin{itemize}
		\item[$\left(i\right)$] For $\theta\in\left(0,1\right)$ and $k\in\mathbb{N}$,
		\begin{equation}
			\left(C^k\left(M, E\right), C^{k+1}\left(M, E\right)\right)_{\theta, \infty}=C^{k,\theta}\left(M, E\right).
		\end{equation}
	\item[$\left(ii\right)$] Let $0\le \theta_1\le\theta_2$ and $0\le\theta\le 1$. Then, if $\omega=\left(1-\theta\right)\theta_1+\theta\theta_2$ is not an integer, 
	\begin{equation}
		\left(C^{\theta_1}\left(M, E\right), C^{\theta_2}\left(M, E\right)\right)_{\theta, \infty}=C^{\omega}\left(M, E\right),
	\end{equation}
such that
\begin{equation*}
	C^{\omega}\left(M, E\right):= C^{\lfloor \omega \rfloor, \omega-\lfloor \omega \rfloor}
\end{equation*}
	$\lfloor \omega \rfloor$ being the integer part of $\omega$.
	\end{itemize}
\end{prop}

\subsubsection{Regularity of the solution}
\begin{cor}
	\label{corollaire_estimation_T}
	Using the previous notation (of $C^{\omega}$), we have that
	\begin{equation}
		\label{estimation_operateur_T}
		\vert\vert T(t)\vert\vert_{\mathcal{L}\left(C^{\theta}(\left(M, E\right)), C^{\alpha}\left(M, E\right)\right)}\le \frac{C e^{\omega t}}{t^{\frac{\alpha-\theta}{2}}}\:,\: 0\le\theta\le\alpha\le 3
	\end{equation}
\begin{proof}
	If $\theta$ and $\alpha$ are both integers then the previous inequality is a direct consequence of estimate $\ref{estimation_operateur_lineaire_T_0_k}$. On the other hand, if $\alpha$ is an integer and $\theta$ is not an integer, then: 
	\begin{equation*}
		\vert\vert T(t)\vert\vert_{\mathcal{L}\left(C^{\theta}(\left(M, E\right)), C^{\alpha}\left(M, E\right)\right)}\le \vert\vert T(t)\vert\vert_{\mathcal{L}\left(C^{\lfloor\theta\rfloor}(\left(M, E\right)), C^{\alpha}\left(M, E\right)\right)}
	\end{equation*}
	If $\alpha$ is not an integer, let $0\le k_1\le k_2\le 3$ be two integers and $s\in\left(0,1\right)$ such that $\alpha = \left(1-s\right)k_1 + sk_2$.
	Using proposition $\ref{deruelle_espace_interpolation}$,we have that $C^{\alpha}\left(M, E\right)=\left(C^{k_1}(M,E), C^{k_2}(M,E)\right)_{s,\infty}$ and  $C^{\lfloor\theta\rfloor}\left(M, E\right)=\left(C^{\lfloor\theta\rfloor}(M,E), C^{\lfloor\theta\rfloor}(M,E)\right)_{s,\infty}$. Now, using theorem $\ref{theoreme_dinterpolation}$, we deduce that
	\begin{equation*}
	\vert\vert T\vert\vert_{\mathcal{L}\left(C^{\theta}\left(M, E\right), C^{\alpha}\left(M, E\right)\right)}\le \left(\vert\vert T\vert\vert_{\mathcal{L}(C^{\lfloor\theta\rfloor}(M,E), C^{k_1}(M, E))}\right)^{1-s}\left(\vert\vert T\vert\vert_{\mathcal{L}(C^{\lfloor\theta\rfloor}(M,E), C^{k_2}(M, E))}\right)^s=\frac{Ce^{\omega t}}{t^{\frac{\alpha-\lfloor\theta\rfloor}{2}}}
	\end{equation*}
	Since we could restrict ourselves to $t\in\left(0,1\right)$ (using the semigroup law), this concludes the proof. 
\end{proof}
\end{cor}
In section $4$ of \cite{lunardi_estimate} , the author notices that even though the semigroup $T(t)$ is not strongly continuous in $C^0\left(M,E\right)$, we could still define a realization of the operator $\mathcal{A}$ in $C^0\left(M,E\right)$. Let $\lambda$>$\lambda_0$ and $R$ be the following linear operator
\begin{equation*}
	\left(R\left(\lambda\right)u\right)\left(x\right)=\int_{0}^{\infty}e^{-\lambda t}\left(T\left(t\right)u\right)\left(x\right)\: dt, \:\:x\in M
\end{equation*} 
$R$ is well defined (using estimate $\ref{estimation_c0_operateur_T}$). Moreover, $\vert\vert R\left(\lambda\right)\vert\vert_{\mathcal{L}\left(C^0\left(M,E\right)\right)}\le \frac{1}{\lambda-\lambda_0}$.\\ Since $R\left(\lambda\right)\left(u\right)\left(x\right)$ is the Laplace transform of the tensor $t\mapsto T\left(t\right)\left(u\right)\left(x\right)$,
$R$ is injective. Thus, there exists a closed linear operator $A:D\left(A\right)\to C^0\left(M, E\right)$ (infinitesimal generator of $\mathcal{A}$), such that $R\left(\lambda\right)$ is the resolvent of$A$, and $D\left(A\right)=Image\left(R\left(\lambda\right)\right)$. By proposition $4.1$ of \cite{lunardi_estimate}, we have that
\begin{prop}[Lunardi]
	\begin{align*}
		&D\left(A\right)= D^2_{\mathcal{A}}\left(M,E\right)\\
		&A h=\mathcal{A}h,\;\forall h\in D\left(A\right) 
	\end{align*}
Moreover,  for any $\theta\in\left(0,2\right)$, there exists a positive constant $C$ such that
\begin{equation}
	\label{plongement_lunardi}
	\vert\vert h\vert\vert_{C^{\theta}\left(M, E\right)}\le C\vert\vert h\vert\vert^{1-\frac{\theta}{2}}_{C^0\left(M, E\right)}\vert\vert h\vert\vert^{\frac{\theta}{2}}_{D\left(A\right)},\: \forall h\in D\left(A\right),
\end{equation}
where $\vert\vert h\vert\vert_{D\left(A\right)}=\vert\vert h\vert\vert_{C^0\left(M, E\right)}+\vert\vert \mathcal{A}\left(h\right)\vert\vert_{C^0\left(M, E\right)}$.
\end{prop}

\begin{proof}[Proof of theorem $\ref{theoreme_lunardi}$]
	Equation $(\ref{plongement_lunardi})$ shows that $D^2_{\mathcal{A}}\left(M, E\right)$ is continuously embedded into $C^{\theta}\left(M, E\right)$ for all $\theta\in\left(0,2\right)$ which proves the first part of theorem $\ref{theoreme_lunardi}$.\\
	To prove the second part of theorem $\ref{theoreme_lunardi}$ we proceed as follows. Let $H\in C^{\theta}\left(M, E\right)$ ($\theta\in\left(0,1\right)$) and $\lambda>\lambda_0$. Then
	\begin{equation}
		h\left(x\right):=\int_{0}^{\infty} e^{-\lambda t}\left(T\left(t\right)H\right)\left(x\right)\:dt,
	\end{equation}
	is well defined. Moreover, since $h\in D\left(A\right)$ (because $\left(\lambda-\mathcal{A}\right)\left(h\right)=H$), we have that $h\in C^{\theta}\left(M, E\right)$ (using equation $(\ref{plongement_lunardi})$),
	\begin{equation*}
		\vert\vert h\vert\vert_{C^{\theta}\left(M, E\right)}\le C\vert\vert h\vert\vert^{1-\frac{\theta}{2}}_{C^0\left(M,E\right)}\vert\vert h\vert\vert^{\frac{\theta}{2}}_{D\left(A\right)}\le \frac{C}{{\left(\lambda-\lambda_0\right)}^{1-\frac{\theta}{2}}}\vert\vert H\vert\vert_{C^{\theta}\left(M,E\right)}.
	\end{equation*}
	It remains to prove that $h\in C^{2,\theta}\left(M, E\right)$. By proposition $\ref{deruelle_espace_interpolation}$ we have that \\$C^{2,\theta}\left(M, E\right)=\left(C^{\alpha}\left(M, E\right), C^{2,\alpha}\left(M, E\right)\right)_{\gamma,\infty}$ for $\gamma=1-\frac{\left(\alpha-\theta\right)}{2}$ and $\alpha\in\left(\theta, 1\right)$.\\
	Let $\eta>\omega$ ($\omega$ of corollary $\ref{corollaire_estimation_T}$). Then, $h$ satisfies $\left(\eta-\mathcal{A}\right)h = H+\left(\eta-\lambda \right)h=\tilde{H}$. The latter satisfies:
	\begin{equation*}
		\vert\vert \tilde{H}\vert\vert_{C^{\theta}\left(M, E\right)}\le \left(1+\frac{C\vert \eta-\lambda\vert}{{\left(\lambda-\lambda_0\right)}^{1-\frac{\theta}{2}}}\right)\vert\vert H\vert\vert_{C^{\theta}\left(M,E\right)}.
	\end{equation*}
	Since $\eta>\lambda_0$, we have that
	\begin{equation*}
		h\left(x\right)=\int_{0}^{\infty} e^{-\eta t}\left(T\left(t\right)\tilde{H}\right)\left(x\right)\: dt.
	\end{equation*}
	For every $\epsilon>0$, set
	\begin{equation*}
		a(x)=\int_{0}^{\epsilon}e^{-\eta t}\left(T\left(t\right)\tilde{H}\right)\left(x\right)\: dt \;\text{;}\; b(x)=\int_{\epsilon}^{\infty}e^{-\eta t}\left(T\left(t\right)\tilde{H}\right)\left(x\right)\: dt.
	\end{equation*}
	Then, $h(x)= a(x)+b(x)$. Using estimate $\ref{estimation_operateur_T}$, there exists positive constants $C_1$ and $C_2$ such that
	\begin{align*}
		& \vert\vert a\vert\vert_{C^{\alpha}\left(M, E\right)}\le C_1\epsilon^{\gamma}\vert\vert\tilde{H}\vert\vert_{C^{\theta}\left(M, E\right)}\\
		& \vert\vert b\vert\vert_{C^{2, \alpha}\left(M, E\right)}\le C_2\epsilon^{\gamma-1}\vert\vert\tilde{H}\vert\vert_{C^{\theta}\left(M, E\right)}
	\end{align*}
	Consequently (using the definition of $\vert\vert h\vert\vert_{\gamma,\infty}$), we obtain that: 
	\begin{align*}
		\vert\vert h\vert\vert_{\gamma,\infty} & \le sup_{\epsilon\in\left(0,1\right)}\epsilon^{-\gamma}\left(\vert\vert a\vert\vert_{C^{\alpha}\left(M,E\right)} +\epsilon\vert\vert b\vert\vert_{C^{2, \alpha}\left(M,E\right)}\right)=\left(C_1+C_2\right) \vert\vert\tilde{H}\vert\vert_{C^{\theta}\left(M,E\right))} \\
		& \le  \left(C_1+C_2\right) \left(1+\frac{C\vert \eta-\lambda\vert}{{\left(\lambda-\lambda_0\right)}^{1-\frac{\theta}{2}}}\right) \vert\vert H\vert\vert_{C^{\theta}\left(M,E\right)}.
	\end{align*}
	Since $\vert\vert h\vert\vert_{\gamma,\infty} =\vert\vert h\vert\vert_{C^{2,\theta}(M,E)}$ this concludes the proof.
\end{proof}

	\nocite{*}
	\bibliographystyle{amsplain}
	\bibliography{references}

\providecommand{\bysame}{\leavevmode\hbox to3em{\hrulefill}\thinspace}
\providecommand{\MR}{\relax\ifhmode\unskip\space\fi MR }
\providecommand{\MRhref}[2]{%
  \href{http://www.ams.org/mathscinet-getitem?mr=#1}{#2}
}
\providecommand{\href}[2]{#2}
\begin{thebibliography}{10}

\bibitem{albin_witt}
Pierre Albin, \'Eric Leichtnam, Rafe Mazzeo, and Paolo Piazza, \emph{{The
  signature package on Witt spaces}}, Annales scientifiques de l'\'Ecole
  Normale Sup\'erieure \textbf{Ser. 4, 45} (2012), no.~2, 241--310 (en).
  \MR{2977620}

\bibitem{albin2015hodge}
Pierre Albin, Eric Leichtnam, Rafe Mazzeo, and Paolo Piazza, \emph{{Hodge
  theory on {C}heeger spaces}}, Journal für die reine und angewandte
  Mathematik (Crelles Journal) \textbf{2018} (2018), no.~744, 29--102.

\bibitem{Ammann_2004}
Bernd Ammann, Robert Lauter, and Victor Nistor, \emph{{On the geometry of
  Riemannian manifolds with a Lie structure at infinity}}, International
  Journal of Mathematics and Mathematical Sciences \textbf{2004} (2004), no.~4,
  161--193.

\bibitem{ammar2020polyhomogeniete}
Mahdi Ammar, \emph{{Polyhomogénéité des métriques compatibles avec une
  structure de Lie à l'infini le long du flot de Ricci}}, Annales de
  l'Institut Henri Poincaré C, Analyse non linéaire \textbf{38} (2021),
  no.~6, 1795--1840.

\bibitem{AubinThierry}
Thierry Aubin, \emph{Nonlinear analysis on manifolds, monge-ampere equations /
  thierry aubin}, Springer-Verlag, New York, 1982 (eng).

\bibitem{ballmann2006lectures}
W.~Ballmann, \emph{{Lectures on K{\"a}hler Manifolds}}, ESI lectures in
  mathematics and physics, European Mathematical Society, 2006.

\bibitem{berline2003heat}
N.~Berline, E.~Getzler, and M.~Vergne, \emph{{Heat Kernels and Dirac
  Operators}}, Grundlehren Text Editions, Springer Berlin Heidelberg, 2003.

\bibitem{boucksom2013introduction}
S.~Boucksom, P.~Eyssidieux, and V.~Guedj, \emph{{An Introduction to the
  {K}{\"a}hler-{R}icci Flow}}, Lecture Notes in Mathematics, Springer
  International Publishing, 2013.

\bibitem{bryant2008gradient}
Robert~L Bryant, \emph{{Gradient K{\"a}hler Ricci solitons}}, Ast{\'e}risque
  \textbf{321} (2008), 51--97.

\bibitem{Carron_2011}
Gilles Carron, \emph{On the quasi-asymptotically locally euclidean geometry of
  nakajima’s metric}, Journal of the Institute of Mathematics of Jussieu
  \textbf{10} (2011), no.~1, 119–147.

\bibitem{ChaljubSimon1979ProblmesED}
Alice Chaljub-Simon and Yvonne Choquet-Bruhat, \emph{{Probl{\`e}mes elliptiques
  du second ordre sur une vari{\'e}t{\'e} euclidienne {\`a} l'infini}}, Annales
  de la Facult{\'e} des Sciences de Toulouse \textbf{1} (1979), 9--25.

\bibitem{chow2004ricci}
B.~Chow and D.~Knopf, \emph{{The Ricci Flow: An Introduction: An
  Introduction}}, Mathematical surveys and monographs, American Mathematical
  Society, 2004.

\bibitem{chow2006hamilton}
B.~Chow, P.~Lu, and L.~Ni, \emph{{Hamilton's Ricci Flow}}, Graduate studies in
  mathematics, American Mathematical Society/Science Press, 2006.

\bibitem{conlon2016expanding}
Ronan Conlon and Alix Deruelle, \emph{{Expanding K\"ahler-Ricci solitons coming
  out of K\"ahler cones}}, Journal of Differential Geometry \textbf{115}
  (2016).

\bibitem{rochon_QAC}
Ronan~J. Conlon, Anda Degeratu, and Fr\'ed\'eric Rochon,
  \emph{{Quasi-asymptotically conical Calabi\textendash{}Yau manifolds}}, Geom.
  Topol. \textbf{23} (2019), no.~1, 29--100.

\bibitem{conlon2019classification}
Ronan~J. Conlon, Alix Deruelle, and Song Sun, \emph{{Classification results for
  expanding and shrinking gradient K\"ahler\textendash{}Ricci solitons}}, Geom.
  Topol. \textbf{28} (2024), no.~1, 267--351.

\bibitem{Conlon}
Ronan~Joseph Conlon, \emph{{On the Construction of Asymptotically Conical
  Calabi-Yau manifolds}}, Doctoral thesis, Imperial College London, 2011.

\bibitem{Gilbard}
Neil S.~Trudinger David~Gilbarg, \emph{Elliptic partial differential equations
  of second order}, Springer Berlin, Heidelberg, 2001 (eng).

\bibitem{debord2011pseudodifferential}
Claire Debord, Jean-Marie Lescure, and Fr\'ed\'eric Rochon,
  \emph{Pseudodifferential operators on manifolds with fibred corners}, Annales
  de l'Institut Fourier \textbf{65} (2015), no.~4, 1799--1880 (en).

\bibitem{Degeratu_2017}
Anda Degeratu and Rafe Mazzeo, \emph{{Fredholm theory for elliptic operators on
  quasi‐asymptotically conical spaces}}, Proceedings of the London
  Mathematical Society \textbf{116} (2017), no.~5, 1112–1160.

\bibitem{smoothing_deruelle}
Alix Deruelle, \emph{{Smoothing out positively curved metric cones by Ricci
  expanders}}, Geometric and Functional Analysis \textbf{26} (2015), 188--249.

\bibitem{EICHHORN}
J.~Eichhorn, \emph{{Elliptic differential operators on noncompact manifolds}},
  Seminar Analysis of the Karl-Weierstrass-Institute of Mathematics, 1986/87,
  Berlin (1988).

\bibitem{melrose_resolvent}
C.~L. Epstein, R.~B. Melrose, and G.~A. Mendoza, \emph{{Resolvent of the
  Laplacian on strictly pseudoconvex domains}}, Acta Mathematica \textbf{167}
  (1991), no.~none, 1 -- 106.

\bibitem{hebey1996sobolev}
E.~Hebey, \emph{{Sobolev Spaces on Riemannian Manifolds}}, Lecture Notes in
  Mathematics, no. no. 1635, Springer, 1996.

\bibitem{jost2008riemannian}
J.~Jost, \emph{{Riemannian Geometry and Geometric Analysis}}, Universitext,
  Springer Berlin Heidelberg, 2008.

\bibitem{joyce2000compact}
D.D. Joyce, \emph{{Compact Manifolds with Special Holonomy}}, Oxford
  mathematical monographs, Oxford University Press, 2000.

\bibitem{joyceALE}
Dominic Joyce, \emph{{Asymptotically Locally Euclidean metrics with holonomy
  SU(m)}}, Annals of Global Analysis and Geometry \textbf{19} (2001), no.~1,
  55--73.

\bibitem{joyceQALE}
\bysame, \emph{{Quasi-ALE metrics with holonomy SU(m) and Sp(m)}}, Annals of
  Global Analysis and Geometry \textbf{19} (2001), no.~2, 103--132.

\bibitem{joyce_corners}
Dominic Joyce, \emph{{Manifolds with analytic corners}}, arXiv preprint
  arXiv:1605.05913 (2016).

\bibitem{kollar}
János Kollár, \emph{{Lectures on Resolution of Singularities (AM-166)}},
  Princeton University Press, 2007.

\bibitem{kottke2021quasi}
Chris Kottke and Fr{\'e}d{\'e}ric Rochon, \emph{{Quasi-fibered boundary
  pseudodifferential operators}}, arXiv preprint arXiv:2103.16650 (2021).

\bibitem{leeparker}
John~M. Lee and Thomas~H. Parker, \emph{{The Yamabe problem}}, Bulletin (New
  Series) of the American Mathematical Society \textbf{17} (1987), no.~1, 37 --
  91.

\bibitem{Lockhart}
Robert~B. Lockhart and Robert~C. Mc~Owen, \emph{{Elliptic differential
  operators on noncompact manifolds}}, Annali della Scuola Normale Superiore di
  Pisa - Classe di Scienze \textbf{Ser. 4, 12} (1985), no.~3, 409--447 (en).
  \MR{837256}

\bibitem{lunardi_interpolation_PDE}
Alessandra Lunardi, \emph{{How to use interpolation in pde’s}}, Available at
  \url{https://people.dmi.unipr.it/alessandra.lunardi/}.

\bibitem{lunardi_interpolation_method}
\bysame, \emph{{An interpolation method to characterize domains of generators
  of semigroups.}}, Semigroup forum \textbf{53} (1996), no.~3, 321--329.

\bibitem{lunardi_estimate}
\bysame, \emph{{Schauder theorems for linear elliptic and parabolic problems
  with unbounded coefficients in $\mathbb{R}^n$}}, Studia Mathematica
  \textbf{128} (1998), no.~2, 171--198 (eng).

\bibitem{lunardi_interpolation}
\bysame, \emph{{Interpolation Theory}}, 01 2018.

\bibitem{mackenzie_1987}
K.~Mackenzie, \emph{{Lie Groupoids and Lie Algebroids in Differential
  Geometry}}, London Mathematical Society Lecture Note Series, Cambridge
  University Press, 1987.

\bibitem{Rafe_QAC}
Rafe Mazzeo, \emph{{Resolution blowups, spectral convergence and
  quasi-asymptotically conical spaces}}, Journ\'ees \'equations aux
  d\'eriv\'ees partielles (2006) (en).

\bibitem{mazzeo1998pseudodifferential}
Rafe Mazzeo and Richard~B Melrose, \emph{{Pseudodifferential operators on
  manifolds with fibred boundaries}}, arXiv preprint math/9812120 (1998).

\bibitem{McOwen}
Robert McOwen, \emph{{The behavior of the laplacian on weighted sobolev
  spaces}}, Communications on Pure and Applied Mathematics \textbf{32} (1979),
  783--795.

\bibitem{melroseCorners}
R.B. Melrose, \emph{{Differential analysis on manifolds with corners}},
  Available at \url{http://www-math.mit.edu/ rbm/book.html}.

\bibitem{melrose1995geometric}
\bysame, \emph{{Geometric Scattering Theory}}, Geometric Scattering Theory,
  Cambridge University Press, 1995.

\bibitem{Melrose1990PseudodifferentialOC}
Richard~B. Melrose, \emph{{Pseudodifferential operators, corners and singular
  limits}}, 1990.

\bibitem{10.1155/S1073792892000060}
Richard~B. Melrose, \emph{{{Calculus of conormal distributions on manifolds
  with corners}}}, International Mathematics Research Notices \textbf{1992}
  (1992), no.~3, 51--61.

\bibitem{monthubert1999pseudodifferential}
Bertrand Monthubert, \emph{{Pseudodifferential calculus on manifolds with
  corners and groupoids}}, Proceedings of the American Mathematical Society
  \textbf{127} (1999), no.~10, 2871--2881.

\bibitem{nakajima1999lectures}
H.~Nakajima, \emph{Lectures on hilbert schemes of points on surfaces},
  University lecture series, American Mathematical Society, 1999.

\bibitem{NIRENBERG1973271}
Louis Nirenberg and Homer~F Walker, \emph{{The null spaces of elliptic partial
  differential operators in $\mathbb{R}^n$}}, Journal of Mathematical Analysis
  and Applications \textbf{42} (1973), no.~2, 271--301.

\bibitem{stratifiedspaces}
Markus~J. Pflaum, \emph{Analytic and geometric study of stratified spaces},
  Springer-Verlag, 2001 (eng).

\bibitem{rochon2014polyhomogeneite}
Frédéric Rochon, \emph{{Polyhomogénéité des métriques asymptotiquement
  hyperboliques complexes le long du flot de Ricci}}, The Journal of Geometric
  Analysis \textbf{25} (2015).

\bibitem{sher_inverse_poly}
David~A Sher, \emph{{Joint asymptotic expansions for Bessel functions}}, Pure
  and Applied Analysis \textbf{5} (2023), no.~2, 461--505.

\bibitem{shi97}
Wan-Xiong Shi, \emph{{Ricci flow and the uniformization on complete noncompact
  Kähler manifolds}}, Journal of Differential Geometry \textbf{45} (1997),
  no.~1, 94 -- 220.

\bibitem{Siepmann}
Michael Siepmann, \emph{{Ricci flows of Ricci flat cones}}, Doctoral thesis,
  ETH Zurich, Zürich, 2013, Diss., Eidgenössische Technische Hochschule ETH
  Zürich, Nr. 21252, 2013.

\bibitem{vertman2021ricci}
Boris Vertman, \emph{{Ricci de Turck flow on singular manifolds}}, The Journal
  of Geometric Analysis \textbf{31} (2021), 3351--3404.

\end{thebibliography}
	
\end{document}